\documentclass[aap]{imsart}

\RequirePackage{amsthm,amsmath,amsfonts,amssymb}
\RequirePackage[numbers]{natbib}
\RequirePackage[colorlinks,citecolor=blue,urlcolor=blue]{hyperref}
\usepackage{paralist}
\usepackage{mathrsfs}
\usepackage{stmaryrd}

\usepackage{tikz,pgfplots}
\pgfplotsset{compat=1.15}
\usetikzlibrary{arrows}
\usepackage{caption,subcaption}

\startlocaldefs
\theoremstyle{plain}
\newtheorem{lemma}{Lemma}[section]
\newtheorem{theorem}[lemma]{Theorem}
\newtheorem{proposition}[lemma]{Proposition}
\newtheorem{corollary}[lemma]{Corollary}

\theoremstyle{remark}
\newtheorem{definition}[lemma]{Definition}
\newtheorem{remark}[lemma]{Remark}

\DeclareMathOperator*{\argmax}{argmax}

\newcommand{\noi}{\noindent}

\newcommand{\lgeo}{\llbracket}
\newcommand{\rgeo}{\rrbracket}

\newcommand{\un}{{\bf 1}}
\newcommand{\dd}{{\rm d}}
\newcommand{\I}[1]{\boldsymbol{1}_{\{#1\}}}

\newcommand{\bbU}{\mathbb U}

\newcommand{\bbN}{\mathbb N}

\newcommand{\bbT}{\mathbb T}

\newcommand{\ccF}{\mathscr F}

\def\ino{ \! \in \! }

\newcommand{\HS}{\mathcal{S}}

\newcommand{\GW}{\mathsf{GW}}

\newcommand{\Cut}{\operatorname{Cut}}

\def\E{ {\mathbb E }}
\def\P{ {\mathbb P }}

\endlocaldefs

\begin{document}

\begin{frontmatter}
\title{The Horton--Strahler number of Galton--Watson trees with possibly infinite variance}
\runtitle{Horton--Strahler number of GW trees with infinite variance}

\begin{aug}
\author[A]{\fnms{Robin}~\snm{Khanfir}\ead[label=e1]{robin.khanfir@mcill.ca}},
\address[A]{Department of Mathematics and Statistics, McGill University\printead[presep={,\ }]{e1}}

\end{aug}

\begin{abstract}
The Horton--Strahler number, also known as the register function, provides a tool for quantifying the branching complexity of a rooted tree. We consider the Horton--Strahler number of critical Galton--Watson trees conditioned to have size $n$ and whose offspring distribution is in the domain of attraction of an $\alpha$-stable law with $\alpha \! \in \! [1,2]$. We give tail estimates and when $\alpha \! \neq \! 1$, we prove that it grows as $\frac{1}{\alpha}\log_{\alpha/(\alpha-1)}n$ in probability. This extends the result of Brandenberger, Devroye \& Reddad dealing with the finite variance case for which $\alpha \! =\!  2$. We also characterize the cases where $\alpha \! =\! 1$, namely the spectrally positive Cauchy regime, which exhibits more complex behaviors. Our proofs are new and probabilistic; they relate the Horton--Strahler number with other shape parameters such as the height or largest degree.
\end{abstract}



\end{frontmatter}


\section{Introduction}
\label{introsec}
The Horton--Strahler number of a finite rooted tree is an integer that quantifies its branching complexity. One possible formal definition is given below. 
\begin{definition}
\label{defHS} Let $t$ be a finite rooted tree. Its \emph{Horton--Strahler number} $\HS(t)$ is recursively defined as follows. 
\begin{longlist}
\item[$(a)$] If $t$ reduces to a single node, then $\HS(t)=0$.

\item[$(b)$] Otherwise, $\HS(t)$ is the maximum of the Horton--Strahler numbers of the subtrees $t_1,...,t_k$ that are attached to the root, plus one if that maximum is not uniquely achieved:
\[\HS(t)=\max_{1\leq i\leq k}\HS(t_i)+\un_{ \big\{\#\argmax_{1\leq i\leq k}\HS(t_i)\geq 2 \big\}} .\]
\end{longlist}
\end{definition}
\noi
Alternatively, $\HS(t)$ is also the height of the largest perfect binary tree that can be embedded into $t$ (see Section~\ref{genHSsec} for more details). In this article, we provide estimates on the Horton--Strahler number of critical Galton--Watson trees conditioned to be large. Before discussing our results precisely, let us provide a brief history with general references on related topics.   

\subsection*{Background} The Horton--Strahler number was introduced independently by the two hydrogeologists Horton~\cite{horton45} and Strahler~\cite{strahler52} to obtain quantitative empirical laws about river systems, that are represented by trees whose leaves are springs and whose root corresponds to the outlet of the basin. Many key physical characteristics of stream networks have been since modeled with the help of this number: see e.g. Peckham~\cite{peckham}, Fac-Beneda~\cite{facbeneda},  Chavan \& Srinivas~\cite{chavan} and Bamufleh et al.~\cite{bamufleh}. The Horton--Strahler number appears independently in other scientific fields (anatomy, botany, molecular biology, physics, social network analysis, etc). In computer science, it is used 
to optimize the amount of memory or time needed to manipulate data structures. It is sometimes called the \emph{register function} in this context because the minimum number of registers needed to evaluate an arithmetic expression $A$ is equal to the Horton--Strahler number of the syntax tree of $A$. We refer to Viennot~\cite{Viennot} for a survey on those various applications.

The Horton--Strahler number is encountered in many areas of mathematics: see for instance Esparza, Luttenberger \& Schlund~\cite{esparza} for connections with mathematical logic, formal language theory, algebra, combinatorics, topology, approximation theory, and more. In the probability area, let us mention Kovchegov \& Zaliapin~\cite{hortonlaws}, which 
considers the Horton--Strahler number through the prism of pruning operations on trees. 
Here, we rather focus on probabilistic works that discuss the Horton--Strahler number of uniform samples of standard families of combinatorial trees. Flajolet, Raoult \& Vuillemin~\cite{flajolet} and Kemp~\cite{kemp}  
consider the Horton--Strahler number of a uniform random ordered rooted binary tree $T_n$ with $n$ leaves (a uniform $n$-\emph{Catalan tree}) and they prove that  
\[ \E[ \HS(T_n)]=\log_4 n+D(\log_4 n)+o(1)\]
as $n\to \infty$, where $\log_b x=\ln x/\ln b$ stands for the logarithm of $x$ to the base $b$, and $D$ is a $1$-periodic continuous function. Here, all the random variables that we consider are defined on the same probability space $(\Omega, \mathscr F, \P)$ whose expectation is denoted by $\E$. In particular, $\HS(T_n)$ is subject to deterministic oscillations. Moreover, Devroye \& Kruszewski~\cite{devroye95} proved that $\HS(T_n)$ is highly concentrated around its expected value via exponential tail estimates. These results were extended to $k$-ary trees by Drmota \& Prodinger~\cite{drmota}. 

More recently, Brandenberger, Devroye \& Reddad~\cite{brandenberger} showed that the Horton--Strahler number of a critical Galton--Watson tree with finite variance offspring distribution conditioned to have $n$ vertices always grows as $\log_4 n$ in probability, which extends all the results that have been previously obtained on first-order behavior. In a companion paper \cite{Kha24}, we go further and study, among other things, the fluctuations and deterministic oscillations of the Horton--Strahler number of large Catalan trees (see Section~\ref{sec:open} for more discussions).

\subsection*{Framework and main results}
Let us give a precise overview of the present article which provides tail estimates and characterizes the first-order behavior of the Horton--Strahler number of critical Galton--Watson trees conditioned to have size $n$ and whose offspring distribution is in the domain of attraction of an $\alpha$-stable law with $\alpha\in[1,2]$. This framework extends the finite variance case, where $\alpha= 2$, and it includes the so-called spectrally positive Cauchy (or $1$-stable) laws. To that end, let us introduce our basic notation and assumptions.

Throughout this work, $\mu=\left(\mu(k)\right)_{k\in\mathbb{N}}$ stands for a probability measure on the set $\mathbb{N}$ of nonnegative integers. We shall always assume $\mu$ to be \emph{critical} and \emph{non-trivial}, namely that
\begin{equation}
\label{critimu}
\sum_{k\in\mathbb{N}}k\mu(k)=1 \quad \text{ and } \quad \mu(0)>0.
\end{equation} 
We view it as the \emph{critical offspring distribution} of a (rooted and ordered) Galton--Watson tree denoted by $\tau$, which is then almost surely finite. See Section~\ref{GWsec} for a formal definition. Several results are expressed in terms of the following strictly convex functions:
\begin{equation}
\label{psi_def}
\forall s\in[0,1],\quad \varphi(s)=\sum_{k\in \bbN}s^k\mu(k)\quad\text{ and }\quad\psi(s)=\varphi(1-s)-(1-s).
\end{equation} 

\smallskip

Our first contribution consists in four propositions that connect $\HS(\tau)$ to other simple characteristics of $\tau$, under the sole assumption that $\mu$ is critical and non-trivial. More precisely, 
\begin{itemize}
\item Proposition~\ref{first_moment_size_general} provides an upper bound for the size $\# \tau$ of $\tau$ when $\HS(\tau)$ is small, 
\item Proposition~\ref{min_hauteur_selon_H_general} provides a lower bound for the height $|\tau|$ of $\tau$ when $\HS(\tau)$ is large, 
\item Proposition~\ref{maj_hauteur_selon_H_general} provides a lower bound for $\HS(\tau)$ when the height $|\tau|$ of $\tau$ is large,
\item Proposition~\ref{min_H_selon_D} gives a lower bound for $\HS(\tau)$ when the maximal out-degree of $\tau$ is large.
\end{itemize}
We will use them to derive the order of magnitude of $\HS(\tau)$ under $\P (\, \cdot \, |\,  \# \tau = n)$. These new estimates are fairly general and are sufficiently accurate to be interesting in their own right.

\smallskip

Our second goal is to estimate the tails of the Horton--Strahler numbers of critical Galton--Watson trees whose offspring distribution has possibly infinite variance. We work under the assumption that 
\begin{equation}
\label{mustable} \textit{$\mu$ belongs to the domain of attraction of a stable law.}
\end{equation}
We denote the scaling index of (the type of) the limiting law by $\alpha$. Since $\mu$ has a finite mean, we get $\alpha \in [1, 2]$  and since $\mu$ is supported on $[0, \infty)$, we only deal with spectrally positive stable laws: namely, their \emph{skewness parameter} $\beta$ is equal to $1$ and the support of their Lévy measure is included in $(0, \infty)$. By standard results, see e.g.~\cite[Theorem 8.3.1]{RegVar}, Assumption (\ref{mustable}) is equivalent to the existence of a function $L : [0, \infty) \mapsto (0, \infty)$ that is \emph{slowly varying} in the sense of Karamata~\cite{karamata1930mode} (which will be recalled at the start of Section~\ref{limtheosec})  and such that
\begin{equation}
\label{tailmu}
\mu\left([n,\infty)\right)\sim n^{-\alpha} L(n) \; \text{ if $\alpha \in [1, 2)$} \quad \text{ and } \quad \sum_{k=0}^n k^2\mu(k) -1\sim 2 L(n)   \; \text{ if $\alpha =2$.} 
\end{equation}
Note that $ \sum_{k=0}^n k^2\mu(k) -1$ is nondecreasing and ultimately positive since $\mu$ satisfies (\ref{critimu}); see Proposition~\ref{limthtpsfix} for more details. 
We first discuss our results in the cases where $\alpha \in (1, 2]$, and then we consider the cases where $\alpha= 1$ that feature more complicated behaviors.

\paragraph*{The cases where $\alpha  \in  (1, 2]$} In these cases, our results are simply expressed in terms of the index $\alpha$ only. First, we prove that if $\mu$ is in the domain of attraction of an $\alpha$-stable law with $\alpha \in (1, 2]$, then the tail of $\HS(\tau)$ follows a universal exponential decay.
\begin{proposition}
\label{HS_tails_alpha}
Assume that $\mu$ is critical and non-trivial and that it belongs to the domain of attraction of a stable law of index $\alpha\in (1,2]$. Then,
\begin{equation}
\label{Ytail_alpha}
- \ln\P(\HS(\tau) > n)\sim n\ln \frac{\alpha}{\alpha-1} .
\end{equation}
\end{proposition}
\noi
This extends the work of Brandenberger, Devroye \& Reddad~\cite[Lemma 2.2]{brandenberger} who proved the $\alpha=2$ case under the assumption that $\mu$ has a finite variance. 

We next discuss the behaviour of $\HS (\tau)$ under $\P (\, \cdot \, | \, \# \tau =n)$ and to that end, assume that 
\begin{equation}
\label{muaperio}
\emph{$\mu$ is aperiodic}
\end{equation} 
(namely that $\mu$ is not supported by a proper additive subgroup of $\mathbb{Z}$) because this implies that $\P(\#\tau=n)>0$ for all $n$ that are large enough. Then, we prove the following. 
\begin{theorem}
\label{stable_HS_order}
Assume that $\mu$ is critical and non-trivial, aperiodic and that it belongs to the domain of attraction of a stable law of index $\alpha\in(1,2]$. Then, the following convergence holds in probability:
\begin{equation}
\label{alphacvs}
\frac{\alpha  \HS(\tau)}{\log_{\frac{\alpha}{\alpha-1}} n}\;  \text{ under $\P (\, \cdot \, |\, \#\tau  =   n)$ }  \longrightarrow  1 . 
\end{equation}
\end{theorem}
\noi
This extends the work of Brandenberger, Devroye \& Reddad~\cite[Theorem 1.1]{brandenberger} who proved the $\alpha = 2$ case under the assumption that $\mu$ has a finite variance. Let us also mention that (\ref{alphacvs}) holds true when $\mu$ is not aperiodic by restricting to the $n$ such that $\P(\#\tau=n)>0$.

Our strategy to prove Theorem~\ref{stable_HS_order} is new and involves linking $\HS(\tau)$ with the height of $\tau$ with Proposition~\ref{maj_hauteur_selon_H_general}. We then rely on a limit theorem of Duquesne~\cite{duquesne_contour_stable} on the height of $\alpha$-stable trees, and control the local conditioning $\P (\, \cdot \, | \, \# \tau =n)$ with an absolute continuity result of Kortchemski~\cite{kortchemski_simple} on positive excursions of random walks. These results are recalled as Proposition~\ref{height_stable_cond} in Section~\ref{limtheosec}.

\paragraph*{The $1$-stable cases} In these more complex cases, we need to specify a converging sequence of rescaled centered sums of $\mu$-distributed independent random variables. Let $(W_n)_{n\in \mathbb{N}} $ be a (left-continuous) random walk starting at $W_0= 0$, whose jump law is given by 
\begin{equation}
\label{RWW}
\P(W_1= k)= \mu (k+1), \quad k\in \{ -1\} \cup \mathbb{N}.
\end{equation} 
Note that $\E [W_1] = 0$. Then by standard results (see e.g.~\cite[Chapters IX.8 and XVII.5]{feller1971}), it holds that $\mu$ belongs to the domain of attraction of a $1$-stable law if and only if there exists a $(0, \infty)$-valued 
sequence $(a_n)_{n\in \mathbb{N}}$ tending to $\infty$ such that 
\begin{equation}
\label{1stableexpl} 
\frac{W_n+b_n}{a_n} \; \xrightarrow[n\to \infty]{\text{(law)}}X, \quad  \text{ where $\: b_n =  n\E [W_1 \un_{\{ |W_1| >a_n\}} ]\:$ and $\: \E [ e^{-\lambda X} ] = e^{\lambda \ln \lambda}$}    
\end{equation}
for all $n\in\bbN$ and $\lambda\in(0,\infty)$; we refer to Proposition~\ref{limthCauchy1} for details. Let us already mention here that $a_n \sim n L(a_n)$ and that $b_n/ a_n \to \infty$, necessarily. Also, the law of $X$ is a \emph{spectrally positive Cauchy (or $1$-stable) law}; its Fourier transform is given by $\E [ \exp (iu X) ]= \exp (- \tfrac{\pi}{2} |u| - iu \ln |u| )$, $u\in \mathbb{R}$.

The asymptotic behavior of $\HS(\tau)$ in the $1$-stable cases is expressed in terms of the sequence $(b_n)_{n\in \mathbb N}$ and the following function $\Upsilon$ that is derived from $\psi $ in (\ref{psi_def}) as follows:
\begin{equation}
\label{Ups} 
\forall s\in (0, 1), \quad \Upsilon(s)=\int_s^1\frac{\dd r}{r\ln\Lambda(r)}, \quad \text{ where } \quad  \Lambda(s)=\frac{s\psi'(s)}{s\psi'(s)-\psi(s)}.
\end{equation}
We refer to Proposition~\ref{psi'(s)_psi(s)/s} for some basic properties of $\Lambda$ and $\Upsilon$. If $\mu$ is in the domain of attraction of a $1$-stable law, we can provide an estimate for the tail of $\HS(\tau)$ in terms of $\Upsilon$. 
\begin{proposition}
\label{HS_tails_1}
Assume that $\mu$ is critical and non-trivial and that it belongs to the domain of attraction of a $1$-stable law. Then,
\begin{equation}
\label{Ytail_1}
\Upsilon \big( \P (\HS(\tau) > n ) \big)\sim n.
\end{equation}
\end{proposition}
\noi
In contrast to the cases where $\alpha  \in  (1, 2]$ for which $\Upsilon (s) \sim_{0^+}  \log_{\frac{\alpha}{\alpha-1}} 1/s$, the asymptotic behavior of $\Upsilon$ when $\alpha= 1$ depends on the slowly varying function $L$ appearing in (\ref{tailmu}): see Lemmas~\ref{behavior_Lambda_alpha} and \ref{behavior_Lambda_1} (also Proposition~\ref{Ups_SV}). For example, Section~\ref{sec:ex} derives the following.

\begin{itemize}
\item[$(\mathbf a)$] If $L(n) \! \sim \! (\ln n)^{-1-\kappa}$ with $\kappa\! \in \! (0, \infty)$, then it holds that
$\Upsilon (s)  \sim_{0^+}  \frac{\ln 1/s}{ \ln \ln 1/s}$ and thus $ -\! \ln \P (\HS(\tau) >  n) \sim   n \ln n $.

\item[$(\mathbf b)$]  If $L(n) \! \sim \! \exp (-(\ln n)^{\kappa})$ with $\kappa\! \in \! (0, 1)$, then 
$\Upsilon (s)  \sim_{0^+}  \frac{1}{1-\kappa} \frac{\ln 1/s}{ \ln \ln 1/s}$ and therefore $ -\! \ln \P (\HS(\tau) > n) \sim   (1\! -\! \kappa )  n \ln n $.

\item[$(\mathbf c)$] If $L(n)\!\sim\!\exp(-\ln n/\ln \ln n)$ then $\Upsilon (s) \sim_{0^+}  \frac{\ln 1/s}{\ln  \ln \ln 1/s}$ and
$ -\! \ln \P (\HS(\tau) >  n)\sim n \ln \ln n $.
\end{itemize}

When the Galton--Watson tree $\tau$ is conditioned to be large, the size of its Horton--Strahler number is of order 
$\Upsilon(\frac{1}{b_n} )$ as proved by the following theorem that first handles the case where $\tau$ is conditioned to have at least $n$ vertices.  
\begin{theorem}
\label{cauchy_HS_order_atleast}
Assume that $\mu$ is critical and non-trivial and that it belongs to the domain of attraction of a $1$-stable law. 
Then, the following convergence holds in probability:
\[\frac{\HS(\tau)}{\Upsilon(\frac{1}{b_n})} \; \text{ under $\P (\, \cdot \, |\, \# \tau \geq  n)$ }  \longrightarrow 1,\]
where $\Upsilon$ is given by (\ref{Ups}) and $(b_n)$ by (\ref{1stableexpl}). Moreover, $\Upsilon(\frac{1}{b_n}) = o(\ln n)$ and $\ln\ln n = o\big(\Upsilon(\frac{1}{b_n})\big)$.
\end{theorem}
\noi
The main idea behind the proof of Theorem~\ref{cauchy_HS_order_atleast} (and of Theorem~\ref{cauchy_HS_order_equal} below) is new and consists in linking $\HS(\tau)$ with the largest out-degree of $\tau$ with Proposition~\ref{min_H_selon_D}. The proof is also supported by several technical results of Kortchemski \& Richier~\cite{KRcauchy} and of Berger~\cite{berger_cauchy} that specify the asymptotic behavior of positive excursion of the random walk $(W_n)_{n\in \mathbb N}$. They are recalled precisely in Section~\ref{limtheosec}, as Lemma~\ref{numtau1sta} and Propositions~\ref{KRs} and \ref{KRvervaat}.

As discussed by Kortchemski \& Richier~\cite{KRcauchy} and Berger~\cite{berger_cauchy}, it is not clear how one could control $\tau$ under $\P (\, \cdot \, |\, \# \tau = n)$ by assuming only that $\mu\left([n,\infty)\right)\sim L(n)/n$ as in (\ref{tailmu}). 
Here, we work under the stronger assumption that $\mu(n)\sim L(n)/n^2$, which implies the previous one and also implies that $\mu$ is non-trivial and aperiodic.  

\begin{theorem}
\label{cauchy_HS_order_equal}
If $\mu$ is critical and if there is a slowly varying function $L$ such that 
\begin{equation}
\label{H1_loc}
\mu(n)\sim\frac{L(n)}{n^2},
\end{equation}
then the following convergence holds in probability:
\[\frac{\HS(\tau)}{\Upsilon(\frac{1}{b_n})} \; \text{ under $\P (\, \cdot \, |\, \# \tau = n)$ }\longrightarrow 1,\]
where $\Upsilon$ is given by (\ref{Ups}) and $(b_n)$ by (\ref{1stableexpl}). Moreover, $\Upsilon(\frac{1}{b_n}) = o(\ln n)$ and $\ln\ln n = o\big(\Upsilon(\frac{1}{b_n})\big)$.
\end{theorem}
As already mentioned, when $\alpha = 1$, the rescaling sequence $\Upsilon(\frac{1}{b_n})$ depends on the slowly varying function $L$ appearing in (\ref{tailmu}). For example, we derive the following in  Section~\ref{sec:ex}.
\begin{itemize}
\item[$(\mathbf a)$] If $L(n)  \sim  (\ln n)^{-1-\kappa}$ with $\kappa\! \in \! (0, \infty)$, then 
$\Upsilon(\frac{1}{b_n})  \sim  \frac{\ln n}{ \ln \ln n}$. 
\item[$(\mathbf b)$] If $L(n)  \sim  \exp (-(\ln n)^{\kappa})$ with $\kappa \in  (0, 1)$, then 
$\Upsilon(\frac{1}{b_n})  \sim  \frac{1}{1-\kappa}\cdot\frac{\ln n}{ \ln \ln n}$. 
\item[$(\mathbf c)$] If $L(n)\sim\exp(-\ln n/\ln \ln n)$, then 
$\Upsilon(\frac{1}{b_n})  \sim  \frac{\ln n}{ \ln \ln \ln n}$. 
\end{itemize}

\subsection*{Organisation of the paper} In Section~\ref{framework}, we properly set our framework and recall from previous works the tools that we use later on in the paper: Section~\ref{GWsec} is devoted to Galton--Watson trees and Section~\ref{limtheosec} to known limit theorems. In Section~\ref{distribution_HS}, we study the law of the Horton--Strahler number of critical Galton--Watson trees in general: we present alternative definitions and establish a new deterministic technical result (Lemma~\ref{left_right_HS}) in Section~\ref{genHSsec}, and derive a recursive equation for the tails in Section~\ref{sec:recursive}. Section~\ref{estimates_tech} focuses on showing Propositions~\ref{first_moment_size_general}, \ref{min_hauteur_selon_H_general}, \ref{maj_hauteur_selon_H_general}, and \ref{min_H_selon_D} that link the Horton--Strahler number to size, height, and maximal out-degree. Section~\ref{pfThmsec_alpha} is devoted to the proof of our main results in the cases where the index of stability $\alpha$ belongs to $(1,2]$: Proposition~\ref{HS_tails_alpha} in Section~\ref{tailHSsec_alpha} and Theorem~\ref{stable_HS_order} in Section~\ref{pfThm1sec}. Section~\ref{pfThmsec_1} studies the cases where $\alpha=1$: we show Proposition~\ref{HS_tails_1} in Section~\ref{tailHSsec_1}, Theorem~\ref{cauchy_HS_order_atleast} in Section~\ref{pfThm2sec}, and Theorem~\ref{cauchy_HS_order_equal} in Section~\ref{pfThm3sec}. We discuss some examples in Section~\ref{sec:ex}. Finally, Section~\ref{sec:open} contains some future research directions opened by our work.

\section{Framework and tools}
\label{framework}
In this section, we recall a set of well-known results that are used in the rest of the article and in the proofs of 
Theorems~\ref{stable_HS_order}, \ref{cauchy_HS_order_atleast} and \ref{cauchy_HS_order_equal}. With the exception of  Lemma~\ref{numtau1sta}, this section contains no new result.

\subsection{Galton--Watson trees}
\label{GWsec}

\paragraph*{Rooted ordered trees} We recall the standard Ulam formalism on rooted ordered trees (see e.g.~\cite{neveu,kesten86,legall_trees}). Let $\mathbb{N}^*=\{1,2,3,...\}$ be the set of positive integers and let $\mathbb{U}$ be the following set of finite words 
\[\mathbb{U}=\bigcup_{n\in\mathbb{N}}(\mathbb{N}^*)^n\quad \text{ with the convention }(\mathbb{N}^*)^0=\{\varnothing\}.\]
The set of words $\mathbb U$ is totally ordered by the \emph{lexicographic order} denoted by $\leq$. For $u=(u_1,...,u_n)\in\mathbb{U}$ and $v=(v_1,...,v_m)\in\mathbb{U}$, we write $u*v=(u_1,...,u_n,v_1,...,v_m)\in\mathbb{U}$ for the \emph{concatenation} of $u$ and $v$. We denote by $|u|=n$ the \emph{height} of $u$, and if $n\geq 1$ then we denote by $\overleftarrow{u}=(u_1,...,u_{n-1})$ the \emph{parent} of $u$. We also say that $u$ is a \emph{child} of $v$ when $\overleftarrow{u}=v$. The \emph{genealogical order} $\preceq$ is a partial order on $\mathbb{U}$ defined by $u\preceq v\Longleftrightarrow \exists u'\in\mathbb{U},\ v=u*u'$. When $u\preceq v$, we will say that $u$ is an ancestor of $v$. When $u\preceq v$ but $u\neq v$, we write $u\prec v$. Finally, we denote by $u\wedge v\in\mathbb{U}$ the \emph{most recent common ancestor} of $u$ and $v$, that is their common ancestor with maximal height.
\begin{definition}
\label{tree_def}
We say that a subset $t$ of $\mathbb{U}$ is a \emph{tree} when the following is satisfied:
\begin{longlist}
\item[$(a)$] $\varnothing\in t$;
\item[$(b)$] for all $u\in t$, if $u\neq\varnothing$ then $\overleftarrow{u}\in t$;
\item[$(c)$] for all $u\in t$, there exists an integer $k_u(t)\in\mathbb{N}$ such that $ u\! *\! (i) \! \in \! t \Longleftrightarrow 1\! \leq \! i \! \leq \! k_u(t)$.
\end{longlist}
We denote by $\mathbb{T}$ the space of all trees.
\end{definition}
Let $t\in \mathbb{T}$. We view $\varnothing$ as the root of $t$. The size of $t$ is simply the (possibly infinite) number $\# t$ of its vertices and we say that $t$ is finite if $\# t< \infty$.  As a graph, the \emph{edges} of $t$ are given by the unordered pairs $\{ u , \overleftarrow{u} \} $ for all $u \in t\backslash \{ \varnothing\}$. Therefore the \emph{degree} of $u\in t$ is $k_u (t) +1$ if $u\neq \varnothing$ and $k_\varnothing (t)$ otherwise. Namely, $k_u (t)$ is the \emph{out-degree of $u$} (alternatively, if one views $t$ as a family tree whose ancestor is $\varnothing$, then $k_u(t)$ stands for the number of children of $u$). Moreover, the height $|u|$ of a node $u\in t$ is the graph-distance between $u$ and the root $\varnothing$. We use the following notation for the \emph{height of $t$} and its \emph{maximal out-degree}: 
 \begin{equation}
 \label{heightdelta}
 |t|=\max_{u\in t}|u|\quad\text{ and }\quad\Delta(t)=\max_{u\in t}k_u(t).
 \end{equation}
We also denote the subtree stemming from $u\in t$ and the tree pruned at $u$ respectively by 
\[\theta_u t=\{v\in\mathbb{U}\ :\ u*v\in t\}\quad\text{ and }\quad \Cut_u t= t\backslash \{v\in t\ :\ u\prec v\}.\]
Observe that $\theta_u t $ and $\Cut_u t$ both belong to $\mathbb{T}$.

\paragraph*{Galton--Watson trees and the Many-To-One Principle} Here, we adopt the same rigorous definition of Galton--Watson trees as Neveu~\cite{neveu} and refer to \cite{legall_trees,abraham2015introduction} for background about these models. Let us equip the set of trees $\bbT$ with the sigma-field $\ccF(\bbT)$ generated by the sets $\{ t \in \mathbb T : u \in t \}$, where $u$ ranges in $\bbU$. Formally, a \emph{random tree} is a function $\tau : \Omega \to \mathbb T$ that is $(\ccF, \ccF(\bbT))$-measurable. 

\begin{definition}
\label{GWfordef} Let $\mu=(\mu(k))_{k\in \bbN}$ be a probability measure on $\bbN$. A \emph{Galton--Watson tree with offspring distribution $\mu$} (a \emph{$\GW(\mu)$-tree}, for short) is a random tree $\tau$ that satisfies the following properties. 
\begin{longlist}
\item[$(a)$] The random variable $k_\varnothing (\tau)$ has law $\mu$. 
\item[$(b)$]For all $k\ino \bbN^*$ such that $\mu(k)\! >\! 0$, the subtrees $\theta_{(1)} \tau, \ldots , \theta_{(k)} \tau$ under $\P (\, \cdot \, |\,  k_\varnothing (\tau) = k)$ are independent with the same law as $\tau$ under $\P$.
\end{longlist}
\end{definition}
It is well-known that a $\GW(\mu)$-tree $\tau$ is almost surely finite if and only if its offspring distribution is subcritical or critical and non-trivial, namely if (\ref{critimu}) holds, and in that case for all finite trees $t\in \mathbb T$,
\begin{equation}
\label{lawGW}
\P (\tau = t)= \prod_{u\in t} \mu \big( k_u(t) \big).
\end{equation}

As observed by Kesten~\cite{kesten86}, a critical 
$\GW(\mu)$-tree conditioned to be large locally converges in law 
to a tree $\tau_\infty$ with a single infinite line of descent and whose law can be informally described as follows:   
all individuals of $\tau_\infty$ reproduce independently, the individuals of the infinite line of descent 
reproduce according to the $\mu$-size-biased distribution $\left(k\mu(k)\right)_{k\in\mathbb{N}}$ whereas the others reproduce according to $\mu$. More precisely, we introduce the following. 
\begin{definition}
\label{IGWdef}
Let $\mu$ be a critical and non-trivial offspring distribution. A \emph{size-biased $\GW(\mu)$-tree} is a random tree 
$\tau_\infty$ that satisfies the following properties.
\begin{longlist}
\item[$(a)$] For all $n\in \mathbb N$, there is a unique $u\in \tau_\infty$ such that $|u|= n$ and $\# (\theta_u \tau_\infty)= \infty$. Denote this vertex by $U_n$. Note that $U_0= \varnothing$ and that $\overleftarrow{U_{n+1}}=U_n$. Hence, there is a $\mathbb N^*$-valued sequence of random variables $(J_n)_{n\in \mathbb N^*}$ such that $U_n$ is the word $( J_1, \ldots, J_n)$ for all $n\in \mathbb N^*$. 
\item[$(b)$] The random variables $(k_{U_n} (\tau_\infty) , J_{ n+1})$, for $n\!\in\! \mathbb N$, are i.i.d~and distributed as follows:
\[\forall j,k\in \mathbb N^*,\quad\P (J_{n+1}=j \, ; \, k_{U_n} (\tau_\infty) =k) = \I{ j\leq k}  \mu(k).\]
\item[$(c)$]  Conditionally given $(k_{U_n}(\tau_\infty),J_{n+1})_{n\in \mathbb N}$, the finite subtrees stemming from the infinite line of descent, i.e.~the $\theta_{U_n\ast (j)} \tau_{\infty}$ for $j \! \in \! \{1,\ldots ,k_{U_n}(\tau_\infty)\}\backslash\{J_{n+1}\}$ and $n\! \in \!  \mathbb N$, are independent $\GW(\mu)$-trees.
\end{longlist}
\end{definition}

\begin{remark}
\label{jntlw} Note that the individual $U_{n+1}$ on the infinite line of descent has $J_{n+1}-1$ siblings strictly on the left-hand side and $ k_{U_n} (\tau_\infty) -J_{n+1}$  siblings strictly on the right-hand side. Their joint law is given by the following bivariate generating function:
\begin{equation}
\label{bivari}
\forall s, r \in [0, 1],\quad \E \big[r^{J_{n+1}-1} s^{k_{U_n} (\tau_\infty) -J_{n+1}} \big]= \frac{\varphi (r) -\varphi (s)}{r-s},
\end{equation}
where $\varphi(s)=\sum_{k\in\bbN}s^k\mu(k)$ as in (\ref{psi_def}) and where the quotient is equal to $\varphi'(r)$ when $r= s$.
\end{remark}
As mentioned above, size-biased trees are the local limits of critical Galton--Watson trees conditioned to be large and therefore appear in many results concerning the asymptotic behavior of branching processes: we refer to Lyons, Pemantle \& Peres~\cite{LlogLcriteria}, Aldous \& Pitman~\cite{aldouspitman98} and Abraham \& Delmas~\cite{AD2014} for general results in this vein. One key tool involving size-biased $\GW(\mu)$-trees is 
the so-called \emph{Many-To-One Principle}, which is part of folklore (see e.g.~Duquesne~\cite[Equation (24)]{duquesne09} for a proof) and which we use here in the form below.
\begin{proposition}[Many-To-One Principle]
\label{many-to-one}
Let $\tau$ be a $\GW(\mu)$-tree and let $\tau_\infty$ be a size-biased $\GW(\mu)$-tree. We keep the notation of Definition~\ref{IGWdef}. 
Then, for all $n\in \mathbb N$ and for all bounded functions $G_1:\mathbb{T}\times\mathbb{U}\longrightarrow\mathbb{R}$ and $G_2:\mathbb{T}\longrightarrow\mathbb{R}$, it holds that 
\[\E\bigg[\sum_{u\in\tau}\un_{\{ |u|=n\} }G_1\big(\Cut_u\tau,u\big)\, G_2\big(\theta_u\tau\big)\bigg]=\E\big[ G_1\big(\Cut_{U_n}\tau_\infty,U_n\big)\big] \,\E\big[G_2\big(\tau\big)\big].\]
\end{proposition}

\paragraph*{Lukasiewicz path associated with a tree} We present here a standard but key combinatorial tool to study Galton--Watson trees via random walks. 
\begin{definition}
\label{Lukadef} Let $t \in  \mathbb T$ be \emph{finite} and 
let $u(t)=(u_j(t))_{0\leq j < \# t}$ be the sequence of its vertices listed in increasing lexicographic order: 
$ u_0 (t)= \varnothing < u_1(t)   < \ldots < u_j (t)< u_{j+1} (t)< \ldots <u_{\# t-1}(t) $.  The sequence $u(t)$ is often called the \emph{depth-first exploration} of $t$. We then define a $\mathbb Z$-valued path $W(t)= (W_j(t))_{0\leq j \leq \# t}$ by setting 
$W_0(t) = 0$ and 
\[W_{j+1} (t) =  W_{j}(t)  + k_{u_{j}(t)} (t) - 1 \]
for all $0\leq j < \#t$. The process $W(t)$ is called the \emph{Lukasiewicz path} of $t$.
\end{definition}
In probability, Lukasiewicz paths originate from queuing systems theory to study the waiting line of a single server subject to the Last-In-First-Out policy and they have been used by Le Gall \& Le Jan~\cite{lejan_legall} to define Lévy trees. In the following lemma, we recall that Lukasiewicz paths are adapted processes that completely encode finite trees and that are particularly well-suited to study Galton--Watson trees (see e.g.~Le Gall~\cite[Proposition 1.1]{legall_trees} and \cite[Corollary 1.6]{legall_trees} for more details). 
\begin{proposition}
\label{luka_GW} Let $t\in \mathbb T$ be finite. Let $\tau$ be a $\GW(\mu)$-tree where $\mu$ is critical and non-trivial. Then the following holds true. 

\begin{itemize}

\item[$(i)$] Lukasiewicz paths provide a one-to-one correspondence between the set of finite trees and the set of finite nonnegative excursions of left-continuous walks, which is defined by
\[\bigcup_{n\in \mathbb N^*}\!\big\{ (w_j)_{0\leq j\leq n} \in \mathbb Z^n\: : \: w_0\! =\!  0, \, w_n \! =\!  -1, \, w_j\geq 0\, \text{ and } \, w_{j+1}\! -\! w_j \! \geq \! -1\,\text{ for all } \, 0\! \leq \! j\! <\!  n  \big\}.\]

\item[$(ii)$] For all $m\in\mathbb{N}$, denote by $R_mt$ the tree $t$ restricted to its $m+1$ first vertices (with respect to the lexicographic order). Namely, 
\begin{equation}
\label{left_portion}
R_m t = \{u_{j}(t)\ :\  0\leq j\leq \min(m,\#t-1)\}
\end{equation}
where $(u_j(t))$ stands for the vertices of $t$ listed in lexicographic order. Then, $R_m t$ is a measurable function 
of $\big( W_j(t) \ ; \ 0 \leq j  \leq  \min(m,\# t) \big)$. 

\item[$(iii)$] Let $(W_n)_{n\geq 0}$ be a $\mathbb Z$-valued random walk starting at $W_0=0$ and whose jumps distribution is given by $\P(W_1=k)=\mu(k+1)$ for $k\in\{-1\}\cup\bbN$ . 
Set $\mathtt{H}_1= \inf \{ j\in \mathbb N\, :\,  W_j = -1\}$, which is an a.s.~finite stopping time since $\E [W_1]=0$. 
Then $\big( W_j (\tau) \big)_{0\leq j\leq \# \tau}$ has the same law as $\big( W_j \big)_{0\leq j\leq \mathtt{H}_1}$. In particular, $\# \tau$ and $\mathtt{H}_1$ have the same law.

\item[$(iv)$] More generally, for all $p\in \mathbb N$, set
\begin{equation}
\label{ladderdef}
\mathtt{H}_0= 0 \quad \text{ and } \quad \mathtt{H}_p = \inf \{ j\in \mathbb N\ :\ W_j= -p \} .
\end{equation}
Then there is an i.i.d.~sequence $(\tau_p)_{p\in \mathbb N}$ of $\GW(\mu )$-trees such that 
\[\forall p\in \mathbb N, \quad  \big(p+W_{\mathtt H_p +j} \big)_{0\leq j\leq \mathtt H_{p+1}-\mathtt H_p} = \big( W_j (\tau_p) \big)_{0\leq j \leq \# \tau_p}  .\]
\end{itemize}
\end{proposition}

\subsection{Limit theorems of the literature}
\label{limtheosec}
In this section we recall $-$ mostly from Bingham, Goldies \& Teugels~\cite{RegVar} and Feller~\cite{feller1971} $-$ limit theorems for sums of i.i.d.~random variables belonging to the domain of attraction of stable laws. We also recall useful limit theorems for Galton--Watson trees and random walks from Berger~\cite{berger_cauchy}, Duquesne~\cite{duquesne_contour_stable}, Kortchemski~\cite{kortchemski_simple}, and Kortchemski \& Richier~\cite{KRcauchy}.

\paragraph*{Regularly and slowly varying functions} Let us begin by recalling the notions of slow and regular variations introduced by Karamata~\cite{karamata1930mode}. A measurable and locally bounded function $\ell:(0,\infty)\to  (0,\infty)$ is \emph{slowly varying} at infinity (resp.~at $0^+$) if  $\ell(cx)\sim \ell (x)$ as $x\to \infty$ (resp.~as $x\to 0^+)$ for all $c\in (0, \infty)$. A function $f:(0,\infty)\to  (0,\infty)$ is \emph{regularly varying of index} $\alpha\in\mathbb{R}$ at $\infty$ (resp.~at $0^+$) 
if there exists a slowly varying function $\ell$ at $\infty$ (resp.~at $0^+$) such that $f(x)=x^\alpha \ell (x)$.  
Below we gather in a single proposition several well-known results on slowly and regularly varying functions that are used in this article.
\begin{proposition}
\label{potter}
Let $\ell$ be a slowly varying function at $\infty$. Then the following holds. 

\begin{itemize}

\item[$(i)$] \emph{(Potter's bound)\ } For all $\varepsilon\in (0, \infty)$ and all $c \in (1, \infty)$, 
there exists $x_0\in (0, \infty)$ such that for all $x\in [x_0, \infty)$ and all $\lambda \in [1, \infty)$, it holds  
\[\frac{1}{c}\lambda^{-\varepsilon}\leq \frac{\ell( x\lambda )}{\ell(x)} \leq c \lambda^\varepsilon .\]
Therefore, $\ln \ell (x)= o (\ln x)$ and if $f$ is regularly varying with index $\rho \in \mathbb R\backslash\{ 0\}$, then $\ln f(x) \sim \rho \ln x $. Moreover, if $x_n\sim y_n\rightarrow\infty$ then $\ell(x_n)\sim \ell(y_n)$ and $f(x_n)\sim f(y_n)$.

\item[$(ii)$] \emph{(Karamata's Abelian Theorem for Tails)\ }  Let $\rho  \in  (0, \infty)$. Then $\int_1^\infty \! y^{-1-\rho}  \ell (y) \,\dd y  <  \infty$ and, as $x\to \infty$, 
\[ \int_x^\infty \!   y^{-1-\rho} \ell (y)\, \dd y  \,  \sim\,    \tfrac{1}{\rho} x^{-\rho} \ell (x)    \quad \text{ and } \quad  \int_1^x \!  y^{\rho-1} \ell (y)\, \dd y  \, \sim \,   \tfrac{1}{\rho} x^{\rho} \ell (x) .  \] 

\item[$(iii)$] Suppose that $\int_1^\infty \! y^{-1}  \ell (y)\, \dd y  < \infty$ and set $\overline{\ell} (x)  =  \int_x^\infty \!  y^{-1} \ell (y) \, \dd y $ for all $x\in(0,\infty)$. Then $\overline{\ell}$ is slowly varying at $\infty$ and $\lim_{x\rightarrow \infty} \overline{\ell} (x) / \ell (x)  =  \infty$. 

\item[$(iv)$] \emph{(Monotone Density Theorem)\ } Let $u: [0, \infty) \to \mathbb R$ be a locally Lebesgue integrable function and set $U(x)= \int_0^x u(y)\, \dd y $ for all $x\geq 0$. Assume that there are $c,\rho\in (0, \infty)$ such that $U(x) \sim c x^\rho \ell(x)$ as $x\to \infty$ and furthermore assume that  $u$ is ultimately monotone. Then $u(x) \sim c\rho x^{\rho-1} \ell(x)$ as $x\to \infty$. 

\item[$(v)$]  \emph{(Karamata's Abelian Theorem for Laplace Transform)\ }  Let $U: [0, \infty)  \to [0, \infty)$ be a measurable function such that $\widehat{U}(\lambda) : =  \lambda \int_0^\infty \! e^{-\lambda x} U(x)\, \dd x  <  \infty$ for all $\lambda  \in  (0, \infty)$. Assume that there are $c \in (0, \infty)$ and $\rho  \in  (-1, \infty)$ such that $U (x)  \sim \frac{c}{\Gamma (1+\rho)} x^\rho \ell (x)$ as $x\to \infty$, then 
$\widehat{U} (\lambda) \sim c\lambda^{-\rho}\,  \ell(1/\lambda)$ as $\lambda \to 0^+$. 
\end{itemize}
\end{proposition}

\begin{proof} For $(i)$, see e.g.~\cite[Theorem 1.5.6]{RegVar}. For $(ii)$, see e.g.~\cite[Theorem 1.5.8 and Proposition 1.5.10]{RegVar}. For $(iii)$, see e.g.~\cite[Theorem 1.5.9b]{RegVar}. For $(iv)$, see e.g.~\cite[Theorem 1.7.2]{RegVar}. For $(v)$, see e.g.~\cite[Theorem 1.7.6]{RegVar}.
\end{proof}

Before moving on, recall from (\ref{critimu}) that a probability measure $\mu$ on $\bbN$ is critical and non-trivial when $\sum_{k\in\bbN} k\mu(k)=1$ and $\mu(0)>0$. Also recall from (\ref{psi_def}) that for all $s\in[0,1]$, we write $\psi(s)=\varphi(1-s)-(1-s)=-(1-s)+\sum_{k\in\bbN}(1-s)^k\mu(k)$ .

\paragraph*{Limit theorems of the literature: the cases where $\alpha \in (1, 2]$}  Here, we present equivalent formulations of the property for $\mu$ to belong to the domain of attraction of a stable law of index $\alpha \in (1, 2]$ (we handle $1$-stable laws separately). 

\begin{proposition}
\label{limthtpsfix} Let $\mu$ be a critical and non-trivial probability measure on $\mathbb N$.  Let $(W_n)_{n\in \mathbb N}$ be a random walk whose jumps distribution is specified in (\ref{RWW}). Let $\alpha \in (1, 2]$. Then, the following assertions are equivalent.

\begin{longlist}
\item[$(a)$] The probability measure $\mu$ belongs to the domain of attraction of an $\alpha$-stable law. 

\item[$(b)$] There exists $L :  (0, \infty)  \rightarrow  (0, \infty)$ that varies slowly at $\infty$ such that if $\alpha \!  \in \! (1, 2)$ then $\mu ([n ,\infty))  \sim  
n^{-\alpha} L(n)$, and if $\alpha \! = \! 2$ then 
$\sum_{0\leq k\leq n} k^2\mu (k) - 1 \sim  2L(n)$ which is ultimately positive since $\sum_{k\in\bbN} k\mu(k)=1$ and $\sum_{k\geq 2}\mu(k)>0$.

\item[$(c)$]  
If $\alpha \!  \in \! (1, 2)$ then $\psi (s) \sim_{0^+} \frac{\alpha -1}{\Gamma (2-\alpha)} s^{\alpha} L(1/s)$,  and if $\alpha \! = \! 2$ then $\psi (s) \sim_{0^+} s^{2} L(1/s)$. 

\item[$(d)$] There exists a $(0, \infty)$-valued sequence $(a_n)_{n\in \bbN}$ tending to $\infty$ such that $\frac{1}{a_n}W_n$ converges in law to the spectrally positive $\alpha$-stable random variable $X_\alpha$ whose Laplace exponent is given for all $\lambda \in [0, \infty)$ by $\ln  \E [ \exp ( -\lambda X_\alpha)]  =   \frac{\alpha -1}{\Gamma (2-\alpha)} \lambda^\alpha\,  $ if $\alpha \!  \in \! (1, 2)$ and by $\lambda^2\, $ if $\alpha \! = \! 2$. 
\end{longlist}
\noi
Moreover, if one of the four equivalent assumptions from above holds true, then 
$a^\alpha _n\sim n L(a_n)$ and $a_n\sim 
n^{1/\alpha} L^*(n)$ where $L^*: x \in  (0, \infty) \mapsto  x^{-1/\alpha} \inf \{ y\! \in \! (0, \infty) : y^\alpha / L(y) \! >\!  x \}$ is slowly varying at $\infty$. 
 \end{proposition}

\begin{proof} For $(a) \Leftrightarrow (b)$, see e.g.~\cite[Theorem 8.3.1]{RegVar}. The equivalence $(a)\Leftrightarrow (d)$ follows from the definition of the domain of attraction of a stable law: the limiting law is necessarily spectrally positive since $\mu$ is supported by $\mathbb N$ and among the spectrally positive $\alpha$-stable types, it is always possible to choose a centered one because $\alpha \in (1, 2]$ (see e.g~\cite[Section 8.3]{RegVar} and \cite[Chapter XVII.5]{feller1971}). 

For $(b) \Leftrightarrow (c)$, see e.g.~\cite[Theorem 8.1.6]{RegVar}. More precisely, recall from (\ref{psi_def}) the definitions of $\varphi$ and $\psi$. Then, for all $\lambda \! \in \! (0, \infty)$, set 
\begin{equation}
\label{DL2}
f_1 (\lambda)\!:= \varphi (e^{-\lambda}) -1 + \lambda  =  \psi (1 - e^{-\lambda})+ \tfrac{1}{2} \lambda^2+O (\lambda^3). 
\end{equation}
If $\alpha \! \in \! (1, 2)$ then $(b)$ is equivalent to $f_1(\lambda) \sim_{0^+} \frac{\alpha -1}{\Gamma (2-\alpha)} \lambda^{\alpha} L( 1/\lambda)$, as asserted by \cite[Theorem 8.1.6]{RegVar}, which implies that $(b) \Leftrightarrow (c)$ in these cases (with Proposition~\ref{potter} $(i)$). 
If $\alpha\! = \! 2$ then \cite[Theorem 8.1.6]{RegVar} asserts that $(b)$ is equivalent to $f_1(\lambda) \sim_{0^+} \lambda^{2}(\frac{1}{2} +L( 1/\lambda))$ and (\ref{DL2}) implies that $(b) \Leftrightarrow (c)$. 

Let us prove the last point of the proposition. We assume that $(a\text{-}d)$ hold true. 
By Grimvall~\cite[Theorem 2.1]{Grim74}, $\lim_{n\to \infty} \E [ \exp (-\frac{\lambda}{a_n} W_n) ]= \E [ e^{-\lambda X_\alpha} ]$ for all $\lambda \in (0, \infty)$. Observe that $ \E [ \exp (-\lambda W_n) ] =  (e^\lambda (f_1 (\lambda) +1 \! -\! \lambda))^n$. If $\alpha \in (1, 2)$, 
we easily get $nf_1 (\lambda /a_n)  \sim \frac{\alpha -1}{\Gamma (2-\alpha)} \lambda^{\alpha} $ as $n\to \infty$ for all $\lambda \in (0, \infty)$ and thus $a^\alpha_n\sim n L(a_n)$. If $\alpha  = 2$, we get 
$$\tfrac{n\lambda}{a_n} + n \ln \big[ 1 -  \tfrac{\lambda}{a_n}\big(1\! -\!  \tfrac{\lambda}{a_n} \big( \tfrac{1}{2} \! + \! L (a_n) \big) (1+ o(1)) \, \big) \big] \sim  \tfrac{n L(a_n)}{a_n^2}  \lambda^2, $$  
which implies that 
$a_n^2 \sim n L(a_n)$. In both cases, we observe that $a_n^\alpha / L(a_n)\sim n$ and we use \cite[Theorem 1.5.12]{RegVar} to complete the proof of the proposition.
\end{proof}

We next recall standard results on the size of Galton--Watson trees, which are expressed in terms of random walks 
thanks to Proposition~\ref{luka_GW}.

\begin{proposition}
\label{lockemp} Let $\mu$ be a critical and non-trivial probability measure on $\mathbb N$.
Let $(W_n)_{n\in \mathbb N}$ be a random walk with $W_0=0$ and whose jumps distribution is given by (\ref{RWW}). 
Recall from (\ref{ladderdef}) the stopping times $(\mathtt H_p)_{p\in \mathbb N}$. Let $\tau$ be a $\GW(\mu)$-tree. Then the following holds true. 
\begin{itemize}

\item[$(i)$] The process $(\mathtt{H}_p)_{p\in \mathbb N}$ is a random walk with positive jumps.  

\item[$(ii)$] \emph{(Kemperman's formula)\ } $\P (\mathtt H_p= n) = \frac{p}{n} \P (W_n  =  -p)$ for all $n,p\in \mathbb N^*$. 

\item[$(iii)$] Suppose that $\mu$ satisfies $(a$-$d)$ in Proposition~\ref{limthtpsfix} and is aperiodic as in (\ref{muaperio}). Then, $\P (\# \tau  =  n)  \sim c_\alpha /(na_n)$ and 
$\P  (\# \tau  \geq  n)  \sim \alpha c_\alpha/a_n$ where $c_\alpha$ is the value at $0$ of the continuous version of the density of $X_\alpha$, and where $(a_n)$ and $X_\alpha$ are as in $(d)$ in Proposition~\ref{limthtpsfix}. 

\item[$(iv)$] Assume that $\mu$ satisfies $(a$-$d)$ in Proposition~\ref{limthtpsfix} . 
Recall from (\ref{heightdelta}) that $|\tau|$ stands for the height of $\tau$. Then, 
\[\frac{\P(|\tau|\geq n)}{\psi\left(\P(|\tau|\geq n)\right)}\sim (\alpha-1)n.\]
\end{itemize}
\end{proposition}

\begin{proof} Note that $(i)$ is an immediate consequence of the Markov property and the left-continuity of $(W_n)$. The point $(ii)$ has been proved Kemperman~\cite[Equation (7.15)]{kemperman}, but see also Wendel~\cite{wendel} for the now-standard combinatorial proof. Let us prove $(iii)$. By $(ii)$ and by Proposition~\ref{luka_GW} $(iii)$, we see that  $\P (\# \tau= n)= \frac{1}{n} \P (W_n\! = \! -1)$. We next use Gnedenko's local limit Theorem (see e.g.~\cite[Theorem 8.4.1]{RegVar}) to get $\lim_{n\to \infty}|a_n \P (W_n \! = \! -1)  -\!  c_\alpha | \! = \! 0$ and thus $\P (\# \tau \! = \! n) \! \sim \!  \frac{c_\alpha }{na_n}$. 
Since $(a_n)$ varies regularly with index $1/\alpha$ by Proposition~\ref{limthtpsfix}, we get $\P  (\# \tau \! \geq \! n) \sim \frac{\alpha c_\alpha}{a_n} $ by Karamata's Abelian Theorem for Tails (see Proposition~\ref{potter} $(ii)$). For $(iv)$, see Slack~\cite[Lemma 2]{slack68}.
\end{proof}

We next recall two limit theorems that are used to prove Theorem~\ref{stable_HS_order}. 
One follows from the convergence of rescaled $\GW(\mu)$-trees to stable trees due to Duquesne~\cite{duquesne_contour_stable}. 
The other is the uniform integrability of the density of the law of (roughly speaking) 
the Lukasiewicz path of a $\GW(\mu)$-tree $\tau$ under $\P (\, \cdot \, | \, \# \tau \! = \! n)$ with respect to $\P (\, \cdot \, | \, \# \tau \! \geq \! n)$ that has been proved in Kortchemski~\cite{kortchemski_simple}.

\begin{proposition}
\label{height_stable_cond}
Assume that $\mu$ is a critical and non-trivial probability measure on $\bbN$, that it is aperiodic, and that it belongs to the domain of attraction of a stable law of index $\alpha\in(1,2]$. More precisely, we assume that $(d)$ in Proposition~\ref{limthtpsfix} holds true. 
Let $\tau$ be a $\GW(\mu)$-tree. Recall from (\ref{heightdelta}) 
that $|\tau|$ stands for the height of $\tau$, and that $W(\tau)$ stands for its Lukasiewicz path as in Definition~\ref{Lukadef}. Then, the following holds true.  
\begin{itemize}

\item[$(i)$] There exists a random variable $M\in(0,\infty)$ such that $\frac{a_n}{n}|\tau|$ under $\P(\, \cdot\, |\, \#\tau=n)$ converges in law to $M$ as $n\to \infty$. 

\item[$(ii)$] Let $r\!\in\!(0,1)$. Then for all large enough $n$, there is a function $D^{(r)}_{n}\! :\! 
\mathbb{N}\!\to\!  [0,\infty)$ such that
\begin{equation}
\label{abs_continuity_equation}
\E\big[ f\big( W_{\min(\lfloor nr \rfloor,\cdot)} (\tau)\big) \, \big|\,  \#\tau \!=\! n \big] = \E\big[ f\big(W_{\min(\lfloor nr \rfloor,\cdot)} (\tau)\big)  
D_n^{(r)} \big( W_{\lfloor rn\rfloor}(\tau) \big) \, \big| \, \#\tau \! \geq\!  n  \big] 
\end{equation}
for all bounded functions $f : \mathbb N^{\mathbb N}\rightarrow [0,\infty)$. Moreover, these functions satisfy
\begin{equation}
\label{abs_continuity_estimate}
\lim_{c\rightarrow\infty}\limsup_{n\rightarrow\infty} 
\E\big[  \un_{\{ D_n^{(r)}\! (W_{\lfloor rn\rfloor}(\tau))\geq c\}}
D_n^{(r)}  \big( W_{\lfloor rn\rfloor}(\tau) \big) \, \big| \, \#\tau \geq  n  \big] =0. 
\end{equation}
\end{itemize}
\end{proposition}

\begin{proof} The point $(i)$ is a consequence of the convergence of rescaled Galton--Watson trees to the $\alpha$-stable tree: see Duquesne~\cite[Theorem 3.1]{duquesne_contour_stable}. Here, $M$ is the height of the normalized $\alpha$-stable tree that is a $(0, \infty)$-valued random variable. See Kortchemski~\cite[Equation (1) and Lemma 2]{kortchemski_simple} for (\ref{abs_continuity_equation}). For (\ref{abs_continuity_estimate}), \cite[Equation (12)]{kortchemski_simple} shows that $D_n^{(r)}(a_n\,\cdot\,)$ uniformly converges on all compact intervals of $(0,\infty)$ towards a continous function. Moreover, the laws of $a_n^{-1}W_{\lfloor rn\rfloor}(\tau)$ under $\P(\,\cdot\,|\,\#\tau=n)$ are tight in $(0,\infty)$ (see \cite[Equation (17)]{kortchemski_simple}), which completes the proof of (\ref{abs_continuity_estimate}) by (\ref{abs_continuity_equation}). 
\end{proof}

\paragraph*{Limit theorems of the literature: the $1$-stable cases} We now consider the domain of attraction of the spectrally positive $1$-stable law, which features more complicated behaviors.

\begin{proposition}
\label{limthCauchy1} Let $\mu$ be a critical and non-trivial probability measure on $\mathbb N$. Let $(W_n)_{n\in \mathbb N}$ be a random walk starting at $W_0=0$ whose jumps distribution is specified in (\ref{RWW}). 
\begin{itemize}

\item[$(i)$] The following assertions are equivalent. 

\begin{itemize}

\item[$(a)$] The probability measure $\mu$ belongs to the domain of attraction of a $1$-stable law. 

\item[$(b)$] 
There exists $L\! :\!  (0, \infty)\!  \rightarrow\!  (0, \infty)$ that varies slowly at $\infty$ such that $\int_1^\infty \! y^{-1} L(y) \, \dd y \! <\! \infty$ and such that 
$\mu ([n ,\infty))  \sim  n^{-1} L(n)$. 

\item[$(c)$] There exists a $(0, \infty)$-valued sequence $(a_n)_{n\in \bbN}$ tending to $\infty$ such that $\frac{1}{a_n} (W_n+b_n)$ converges in law to $X$, where $b_n =  n\E [W_1 \un_{\{ |W_1| >a_n\}} ]$ and where $X$ is the spectrally positive $1$-stable random variable whose Laplace exponent is given for all $\lambda \in (0, \infty)$ by $\ln  \E [ \exp ( -\lambda X)]  =  \lambda \ln \lambda $. 
\end{itemize}

\item[$(ii)$] Assume that $(a$-$c)$ hold true. Then, $\sum_{k\geq n} k\mu (k) \sim \overline{L} (n)$ where $\overline{L}$ is the slowly varying function defined by 
 \begin{equation}
 \label{deflL}\forall x\in(0,\infty),\quad \overline{L} (x) = \int_x^\infty  \frac{L(y)}{y} \, \dd y. 
 \end{equation}
 This implies that $\psi (s) \sim_{0^+} s \overline{L} (1/s)$. Moreover, $L(x)  =  o(\overline{L} (x))$ as $x\to \infty$. 

\item[$(iii)$]  Assume that $(a$-$c)$ hold true. Then, 
$a_n\sim n L(a_n)$ and $a_n\sim n L^*(n)$ where the function $L^*: x \in  (0, \infty)\mapsto x^{-1} \inf \{ y\! \in \! (0, \infty) : y/ L(y) \! >\!  x \}$ varies slowly at $\infty$. Furthermore, it holds that $b_n \sim n \overline{L} (a_n)$ where $\overline{L}$ is given by (\ref{deflL}), and therefore $a_n=o(b_n)$.  
\end{itemize}
\end{proposition}

\begin{proof} The equivalences $(a) \Leftrightarrow (b)$ and $(a) \Leftrightarrow (c)$ are proved as in Proposition~\ref{limthtpsfix} (for the form of $b_n$, see e.g.~\cite[Chapter IX.8, Equation (8.15)]{feller1971}). To prove $(ii)$, first observe that 
$\sum_{k\geq n} k\mu (k)  =  \sum_{j\geq 1} \mu ([\max(n,j) , \infty))= n \mu ([n , \infty)) + 
\sum_{j>n} \mu ([ j , \infty))$. Then note that $n \mu ([n , \infty)) \sim L(n)$ and that $\sum_{j>n} \mu ([ j , \infty)) \sim \overline{L} (n)$ where $\overline{L}$ is defined by (\ref{deflL}). By Proposition~\ref{potter} $(iii)$, $L(n)= o(\overline{L} (n))$ and we get $\sum_{k\geq n} k\mu (k) \sim \overline{L} (n)$. By \cite[Theorem 8.1.6]{RegVar}, this is equivalent to $f_1(\lambda) \sim_{0^+} \lambda \overline{L} (1/\lambda)$ 
where $f_1$ is given by (\ref{DL2}). This implies  $\psi (s) \sim_{0^+} s \overline{L} (1/s)$ by Proposition~\ref{potter} $(i)$, which completes the proof of $(ii)$. 
Then, we prove that $a_n\sim n L(a_n)$ and $a_n\sim 
n L^*(n)$ as in Proposition~\ref{limthtpsfix}. Next observe that $b_n  =  n \sum_{k>a_n}  k\mu (k+1)  \sim n\overline{L} (a_n) $. Thus $a_n/b_n \sim L(a_n)/ \overline{L} (a_n) \to 0$ by $(ii)$. This completes the proof of $(iii)$.
\end{proof}

Suppose that $(c)$ in Proposition~\ref{limthCauchy1} holds true. 
Since $a_n=o(b_n)$, it holds $b_n^{-1} W_{\lfloor ns \rfloor}\to -s$ in law, and thus in probability, for all $s\in [0, \infty)$. Standard arguments (or a stronger result such as Skorokhod~\cite[Theorem 2.7]{skorokhod57}) entail the following convergence 
\begin{equation}
\label{luka_func_cauchy}
\big(\tfrac{1}{b_n}W_{\lfloor ns\rfloor} \ ; \ s\in [0, \infty) \big)\longrightarrow  \big( -\!s\ ;\  s\in [0, \infty)\big)
\end{equation}
in probability for the topology of uniform convergence on all compact intervals. Recall from (\ref{ladderdef}) the definition of the stopping times $(\mathtt{H}_p)_{p\in \mathbb N}$. Then, (\ref{luka_func_cauchy}) implies for all $x\in [0, \infty)$ that 
\begin{equation}
\label{rela_stabi_H}
\tfrac{1}{n}\mathtt{H}_{\lfloor b_n x \rfloor } \longrightarrow  x \; 
\end{equation}  
in probability. Namely the law of $\mathtt{H}_1$ is \emph{relatively stable} (see e.g.~\cite[Section 8.8 $\S$ 1]{RegVar}). Since the total size of a $\GW(\mu)$-tree $\tau$ has the same distribution as $\mathtt{H}_1$ (by Proposition~\ref{luka_GW} $(iii)$), the law of $\# \tau$ is thus relatively stable. 
By use of  Berger~\cite[Theorem 2.4 and Lemma 4.3]{berger_cauchy} and Kortchemski \& Richier~\cite[Proposition 12]{KRcauchy}, we get the following result. 
\begin{lemma}
\label{numtau1sta}
Let $\mu$ be a critical and non-trivial probability measure on $\mathbb N$ that belongs to the domain of attraction of a $1$-stable law. Let $\tau$ be a $\GW(\mu)$-tree. Then, 
\begin{equation}
\label{PS_tails_size}
\P (\# \tau \geq  n) \sim  \frac{L(b_n)}{b_n \overline{L} (b_n)}, 
\end{equation}
where $(b_n)$, $L$ and $\overline{L}$ are as in Proposition~\ref{limthCauchy1}. Moreover, if $\mu(n)\sim L(n)/n^2$ as in (\ref{H1_loc}) then
\begin{equation}
\label{local_tails_size}
\P(\#\tau=n)\sim \frac{1}{n}\P(\#\tau\geq n)\sim\frac{L(b_n)}{b_n^2}.
\end{equation}
\end{lemma}

\begin{proof} Since $b_n / n \sim \overline{L} (a_n) \to 0$ as $n\to \infty$, (\ref{rela_stabi_H}) implies that there is a sequence $(c_n)$ tending to $\infty$ such that $\frac{1}{c_n} \mathtt{H_n} \to 1$ in probability: namely, the law of $\mathtt H_1$ is relatively stable. Then, \cite[Theorem 8.8.1]{RegVar} asserts the existence of a function $\ell$ that slowly varies at $\infty$ such that 
$\sum_{0\leq k\leq n} \P ( \mathtt{H}_1\! \geq\!  k)  \sim \ell(n)$ and $1-\E [ e^{-\lambda \mathtt{H}_1} ] \sim_{0^+} \lambda  \ell (1/\lambda)$. Thus, writing $\E[e^{- \mathtt{H}_{\lfloor b_n\rfloor}/n}]=\E[e^{-\mathtt{H}_1/n}]^{\lfloor b_n\rfloor}$ thanks to Proposition~\ref{lockemp} $(i)$, the convergence (\ref{rela_stabi_H}) entails that $\overline{L} (a_n) \ell (n) \to 1$. By Berger~\cite[Lemma 4.3]{berger_cauchy}, we get $\overline{L} (a_n)\sim \overline{L} (b_n) $ and thus, 
$\sum_{0\leq k\leq n} \P (\# \tau \!  \geq \!  k)  =    \sum_{0\leq k\leq n} \P ( \mathtt{H}_1 \! \geq \!  k)   
\sim  1/\overline{L} (b_n)$.  By Kortchemski \& Richier~\cite[Proposition 12]{KRcauchy}, we get (\ref{PS_tails_size}) because within the notation of \cite{KRcauchy}, we necessarily get $\Lambda (n) \! \sim \! 1/\overline{L} (b_n) $. Moreover, Berger~\cite[Theorem 2.4]{berger_cauchy} asserts that if $\mu(n)\sim L(n)/n^2$ then $\P (W_n = \! -1)  \sim  nL(b_n)/b_n^2$, so (\ref{local_tails_size}) follows from Kemperman's formula (Proposition~\ref{lockemp} $(ii)$). 
\end{proof}
 
 \begin{remark}
\label{newesti} 
Although the difficult part of (\ref{PS_tails_size}) is the very content of Kortchemski \& Richier~\cite[Proposition 12]{KRcauchy}, the relatively explicit form of the right member of (\ref{PS_tails_size}) seems novel under the sole assumption that $\mu$ is in the domain of attraction of a $1$-stable law.
\end{remark}

We next recall two limit theorems on the maximal out-degree of a $\GW(\mu)$-tree $\tau$ when $\mu$ belongs to the domain of attraction of a $1$-stable law. They are part of more general results due to Kortchemski \& Richier~\cite{KRcauchy}. 
\begin{proposition}
\label{KRs}
\label{D_stable_cond_atleast}
Let $\tau$ be a $\GW(\mu)$-tree whose offspring distribution $\mu$ is critical and non-trivial and that belongs to the domain of attraction of a $1$-stable law. Recall from (\ref{heightdelta}) that $\Delta (\tau)$ stands for the maximal out-degree of $\tau$. Then, the following holds true. 
\begin{itemize}
\item[$(i)$] The following convergence holds in distribution on $[0, \infty)$:
\[\frac{1}{b_n}\Delta(\tau) \; \text{ under } \P (\, \cdot \, |\, \# \tau  \geq n)  \, \longrightarrow J,\]
where the law of $J$ is given by $\P(J\geq x)=1/x$ for all $x\in [1, \infty)$.  

\item[$(ii)$] Under the stronger assumption (\ref{H1_loc}), the following convergence holds in probability:  
\[\frac{1}{b_n}\Delta(\tau) \; \text{ under } \P (\, \cdot \, |\, \# \tau = n)  \, \longrightarrow 1. \]
\end{itemize}
\end{proposition}

\begin{proof} For $(i)$, see \cite[Theorem 6]{KRcauchy}. For $(ii)$, see \cite[Theorem 1]{KRcauchy}.
\end{proof}

Recall that $W(\tau)$ stands for the Lukasiewicz path of $\tau$, as in Definition~\ref{Lukadef}. We conclude this section by recalling a result from Kortchemski \& Richier~\cite{KRcauchy} that shows that the law of $W(\tau)$ under $\P (\, \cdot \, |\, \# \tau =  n)$ is closed in variation distance to the law of the Vervaat transform of the path $(W_0, W_1, \ldots, W_{n-1} , -1)$ under $\P$. More precisely, let $(W_n)_{n\in \mathbb N}$ be as in Proposition~\ref{limthCauchy1}. For all $n\in \mathbb N^*$, we introduce the following notation. 
\begin{align}
\label{minargmin}
I_n &= -   \min_{ 0\leq j \leq n-1}  W_j \quad \text{ and } \quad \sigma_n =  \inf\{0  \leq  k \leq  n - 1\ :\ W_k  = -I_n  \},\quad \text{ and }\\
\label{Vervdef}
 Z_j^{(n)}  &= \begin{cases}
				  W_{\sigma_n +j}  + I_n &  \text{ if } \; \, 0 \leq  j < n - \sigma_n  , \\
				 I_n  - 1 + W_{j-(n-\sigma_n)}& \text{ if } \; \, n  \geq  j  \geq  n - \sigma_n. 
			\end{cases}
\end{align}			
Namely, $Z^{(n)}$ is constructed by reading the increments of 
$(W_0,W_1,...,W_{n-1},-1)$ from left to right in cyclic order by starting at time $\sigma_n$: this is a kind of \emph{Vervaat transform} of $(W_0,W_1,...,W_{n-1},-1)$ (see Vervaat~\cite{vervaat} for more details).  
We shall use Kortchemski \& Richier~\cite[Theorem 21]{KRcauchy}, which we re-state by convenience 
into the following proposition. 
\begin{proposition}
\label{KRvervaat}
We keep the above notation and assume that $\mu$ is critical and non-trivial and satisfies (\ref{H1_loc}). Let $\mathscr{B}(\mathbb{R}^{n+1})$ stand for the Borel sigma-field of $\mathbb{R}^{n+1}$. Then, 
\[\sup_{A\in\mathscr{B}(\mathbb{R}^{n+1} )}\Big|\, \P\big(W (\tau)  \in  A \, \big|\, \#\tau  = n\big)-
\P\big(Z^{(n)} \in  A\big)\, \Big| \longrightarrow 0.\]
\end{proposition}

\begin{proof} See Kortchemski \& Richier~\cite[Theorem 21]{KRcauchy} and note that $\P(I_n>1)\rightarrow 1$ by using (\ref{luka_func_cauchy}) for example.
\end{proof}

\section{Distribution of the Horton--Strahler number of Galton--Watson trees}
\label{distribution_HS}

\subsection{Alternative definitions of the Horton--Strahler number and basic properties}
\label{genHSsec}
In this section, we prove basic results on the Horton--Strahler number. Firstly we provide alternative definitions, then we state a key upper bound in Lemma~\ref{left_right_HS}.

The first alternative definition of the Horton--Strahler number uses \emph{Horton pruning} of a finite tree $t\in \mathbb T$ that is defined as follows: \emph{remove the leaves of $t$ and merge each line into one edge (a \emph{line} in $t$ is a maximal sequence of vertices $v_0, \ldots, v_n \in t$ such that $k_{v_1} (t)  =  \ldots  =  k_{v_{n-1}} (t) =  1$  and $v_j =  \overleftarrow{v_{j+1}}$ for all  $0 \leq  j  <  n$). The resulting tree is called the Horton-pruned tree, which we denote here by $\mathtt{Prun} (t)$.} Then, $\HS(t)$ is the minimal number of Horton prunings that are necessary to obtain 
$\{ \varnothing\} $ from $t$. Namely, 
\begin{equation}
\label{HSprundef} 
\HS(t)=  \min\big\{n  \in  \mathbb N\, :\, \mathtt{Prun}_n (t) = \{ \varnothing \} \big\}, 
\end{equation}
where $\mathtt{Prun}_n$ stands for the $n$-th iteration of $\mathtt{Prun}$ for all $n\in\bbN^*$ and where $\mathtt{Prun}_0$ stands for the operation that merges each line into one edge. We refer to Kovchegov \& Zaliapin~\cite[Section 2.3]{hortonlaws} for a proof and more details. 

Another useful definition uses embeddings of perfect binary trees. More precisely, let $t, t^\prime   \in  \mathbb T$ be finite. Then, $\phi: t\to t^\prime$ is an \emph{embedding if it is injective and if $\phi (u\wedge v) \! = \! \phi (u) \! \wedge \! \phi (v)$ for all $u, v  \in  t$}. For all $n\in \mathbb N$, denote by $T_{2,n} =  \bigcup_{0\leq k\leq n} \{ 1,2\}^k$ the $n$-perfect binary tree, with the convention that $\{ 1, 2\}^0 =  \{ \varnothing \}$. Then, for all finite trees $t\in \mathbb T$, 
\begin{equation}
\label{HSbindef}
\HS(t)= \max \big\{ n\in \mathbb N\, :\, \exists \phi : T_{2, n} \to t\, \text{ embedding} \big\}. 
\end{equation}    
This result seems to be `part of the folklore'. Let us however provide a short proof. 

\begin{proof}[Proof of (\ref{HSbindef})] We reason by induction on the height $|t|$ of $t$, as defined by (\ref{heightdelta}). Note that (\ref{HSbindef}) obviously holds true 
if $|t|=0$. Now assume that $k:=k_\varnothing(t)\geq 1$ and let $\phi:T_{2,n+1}\to t$ be an embedding. We separate the cases according to the positions of $\phi(\varnothing),\phi(1),\phi(2)$. 

\begin{itemize}
\item If $(i)\preceq \phi(\varnothing)$ with $1  \leq  i \leq k$, then we have $(i)\preceq\phi(u)$ for all $u\in T_{2,n+1}$ by definition of embeddings, and we can check that setting $\phi(u)=(i)*\phi_i(u)$ defines an embedding $\phi_i:T_{2,n+1}\to\theta_{(i)}t$. Conversely, an embedding $\phi_i:T_{2,n+1}\to\theta_{(i)}t$ induces an embedding $\phi:T_{2,n+1}\to t$ such that $(i)\preceq \phi(\varnothing)$. Thus, it holds that
\[\max\big\{n\in\bbN\, :\, \exists \phi : T_{2, n} \to t\, \text{ embedding},(i)\preceq\phi(\varnothing)\big\}=\HS(\theta_{(i)}t).\]

\item Otherwise, $\varnothing=\phi(\varnothing)=\phi(1)\wedge\phi(2)$ so we have distinct $1\leq i, j \leq k$ such that $(i)\preceq \phi(1)$ and $(j)\preceq \phi(2)$. Similarly as before, we see that setting $\phi((1)*u)=(i)*\phi_i(u)$ and $\phi((2)*u)=(j)*\phi_j(u)$ respectively defines two embeddings $\phi_i:T_{2,n}\to\theta_{(i)}t$ and $\phi_j:T_{2,n}\to\theta_{(j)}t$. Conversely, two embeddings $\phi_i:T_{2,n}\to\theta_{(i)}t$ and $\phi_j:T_{2,n}\to\theta_{(j)}t$ induces an embedding $\phi:T_{2,n+1}\to t$ such that $\phi(\varnothing)=\varnothing$, $(i)\preceq \phi(1)$, and $(j)\preceq\phi(2)$. Thus,
\begin{multline*}
\max\big\{n\in\bbN\, :\, \exists \phi : T_{2, n} \to t\, \text{ embedding},(i)\preceq\phi(1),(j)\preceq\phi(2)\big\}\\
=1+\min\big(\HS(\theta_{(i)}t),\HS(\theta_{(j)}t)\big).
\end{multline*}
\end{itemize}
Taking the maximum over $\phi(\varnothing),\phi(1),\phi(2)$  concludes the proof by Definition~\ref{defHS}.
\end{proof}

\noi
The definition (\ref{HSbindef}) immediately implies the following. Let $t, t^\prime\in \mathbb T$ be finite. 
 \begin{equation}
 \label{monotony_HS} 
 \textit{If there is an embedding $\phi: t \to t^\prime$},\quad \textit{ then } \quad \HS(t)\leq \HS(t^\prime). 
\end{equation}

We next use (\ref{HSbindef}) and (\ref{monotony_HS}) to get an upper bound of $\HS(t)$ in terms of $\HS(R_m t)$, where $R_m t$ is defined in (\ref{left_portion}) (recall that it 
is the tree consisting in the first $m+1$ vertices of $t$ in lexicographic order) and of $R_m t^\star$, where $t^\star$ is the \emph{mirror image of $t$} that is formally defined just below. Informally, the mirror image of $t$ is the tree obtained by reversing the left-right order of $t$ while preserving its genealogical structure. See Figure~\ref{fig:mirror} for an example.

\begin{figure}
    \centering
    \begin{subfigure}{0.4\textwidth}
    \centering
    \begin{tikzpicture}[baseline={(current bounding box.center)}, every node/.style={circle,fill,inner sep=0pt, minimum size=5pt}, level
distance=10mm,level 1/.style={sibling distance=20mm},
     level 2/.style={sibling distance=10mm},
     level 3/.style={sibling distance=5mm}]
     \node[label={$\varnothing$}]  {} [grow=up]
     child {node{}
     	child {node{}
                child {node[label={[shift={(0.3,-0.1)}]$v$}]{}
                    child {node{}}
                }
     		child {node{}}
     		child {node{}
     			child {node{}}
     			child {node{}}
     			}
     		}
     	child {node[label={[shift={(-0.3,-0.1)}]$u$}]{}
     		child {node{}}
                }
     	}
     child {node{}};
    \end{tikzpicture}
    \end{subfigure}
    \begin{subfigure}{0.4\textwidth}
    \centering
    \begin{tikzpicture}[baseline={(current bounding box.center)}, every node/.style={circle,fill,inner sep=0pt, minimum size=5pt}, level
distance=10mm,level 1/.style={sibling distance=20mm},
     level 2/.style={sibling distance=10mm},
     level 3/.style={sibling distance=5mm}]
     \node[label={$\varnothing$}]  {} [grow=up]
     child {node{}}
     child {node{}
            child {node[label={[shift={(0.3,-0.1)}]$u^\star$}]{}
     		child {node{}}
                }
     	child {node{}
                child {node{}
     			child {node{}}
     			child {node{}}
     			}
                child {node{}}
                child {node[label={[shift={(-0.3,-0.1)}]$v^\star$}]{}
                    child {node{}}
                    }
     		}
     	}
      ;
    \end{tikzpicture}
    \end{subfigure}
    \caption{An example of a tree $t$ on the left, and its mirror image $t^\star$ on the right. Two vertices $u,v$ of $t$ are depicted, as well as their mirror images $u^\star,v^\star$ on $t^\star$.}
    \label{fig:mirror}
\end{figure}

Let $t\in \mathbb T$ and $u \in t \backslash \{ \varnothing \} $ be the word $(j_1, \ldots, j_n)$.  We set $u_{|0}\! = \! \varnothing$ and $u_{| p}\! = \! (j_1, \ldots, j_p)$ for all $1 \leq p \leq  n$. Then, the mirror image of $u$ is the word $u^\star =  (j_1^\star, \ldots, j_n^\star)$ where  
$$ j^\star_p:= k_{u_{|p-1}} (t)-j_p +1 , \quad 1\leq p \leq n.$$
Also set $\varnothing^\star=\varnothing$. Then $t^\star =  \{ u^\star\,  :\,  u \in  t \}$ and it is easy to show that $t^\star\in \mathbb T$. See Figure~\ref{fig:mirror} for an example of mirror images of vertices on a tree. Note that $u\mapsto u^\star$ is a bijective embedding so $\HS(t^\star)=\HS(t)$ by (\ref{monotony_HS}). We stress that the word $u^\star$ depends on the tree $t$ on which $u$ is observed. Nevertheless, this notation should not lead to confusion here because the underlying tree will always be clear according to context. Since $k_{u^\star} (t^\star)= k_u(t)$, (\ref{lawGW}) implies that if $\tau$ is a critical $\GW(\mu)$-tree, then so is $\tau^\star$. Furthermore, since $\# \tau^\star =  \#\tau$, if $n  \in \mathbb N^*$ is such that $\P(\# \tau \! = \! n) > 0$, then we easily check that 
\begin{equation}
\label{symmtau}
\text{under $\P(\, \cdot \, | \, \# \tau  =  n)$, } \quad \tau^\star \overset{\text{(law)}}{=} \tau.
\end{equation}
The following lemma plays a key role in the proofs of Theorems~\ref{stable_HS_order} and \ref{cauchy_HS_order_equal}. 

\begin{lemma}
\label{left_right_HS}
Let $t\in \mathbb T$ be finite. Then the following holds true. 

\begin{itemize}
\item[$(i)$] For all $u\in  t$, set $t_{\leq u} =  \{ v  \in  t\, :\, v  \leq  u\} $. Then, 
\begin{equation}
\label{mirsli0}
\HS(t) \leq 1+ \max \Big( \HS(t_{\leq u}) \, , \,  \max\{ \HS(\theta_v t) \, : \, v  \in  t,  \overleftarrow{v}  \preceq  u  \text{ and }  v > u \} \Big).
\end{equation}

\item[$(ii)$] Let $m\in\mathbb{N}$ be such that $2m\geq\#t+|t|$. Recall from (\ref{left_portion}) that $R_m t$ is the tree consisting in the first $m+1$ vertices of $t$ in lexicographic order. Then, 
\begin{equation}
\label{keybnd}
\HS(t) \leq1+\max \big(\HS(R_m t),\HS(R_m t^\star)\big).
\end{equation}
\item[$(iii)$] If $\tau$ is a random finite tree such that $\tau^\star$ has the same law as $\tau$, then for all $n, m\in \mathbb N^*$,
\[\P (\HS(\tau) \geq n ) \leq 2\P \big( \HS( R_m \tau) \geq  n \! -\!  1 \big)+ \P ( \# \tau + |\tau|  >  2m ).  \]
\end{itemize}
\end{lemma}

\begin{proof}
 We first prove $(i)$. If $\HS (t) =  0$, then (\ref{mirsli0}) is obviously true. We next suppose that $\HS(t) =  n  \geq 1$, and to simplify notation, set 
\[\lgeo \varnothing , u\rgeo =\{ v \in t\, :\,  v \preceq  u\} \quad \text{ and } \quad  B= \{ v \in t \, :\,   \overleftarrow{v}  \preceq  u \text{ and } v > u\} . \]
By (\ref{HSbindef}), there is an embedding $\phi  :  T_{2, n}  \to  t$. Observe that $\phi (1)$ and $\phi (2)$ cannot both belong to $\lgeo \varnothing , u \rgeo$, otherwise it would imply $\phi (1)\wedge \phi (2) = \phi (\varnothing)\in\{\phi(1),\phi(2)\}$ (by definition of embeddings since $\varnothing = (1) \wedge  (2)$), which contradicts the injectivity of $\phi$. Thus, there is $j \in  \{ 1, 2\}$ such that either 
$\phi (j) \in  t_{\leq u} \backslash \lgeo \varnothing , u \rgeo $ or $\phi (j) \in  t\backslash  t_{\leq u} $. 
In the first case, by definition of embeddings, $\phi(j)\preceq\phi((j)*v)$ for all $v\in T_{2,n-1}$ and so $\phi((j)*T_{2, n-1} )  \subset  t_{\leq u}$. Then, (\ref{HSbindef}) entails that $n\! -\! 1 \leq  \HS(t_{\leq u})$. Suppose next that $\phi (j) \in  t\backslash  t_{\leq u} $. Since 
$ t\backslash  t_{\leq u}$ is the disjoint union of the $v  \ast  (\theta_v t)$ for $v \in  B$, there exists $v \in  B$ such that 
$\phi ((j)\ast T_{2, n-1} )  \subset  \phi (j) \ast (\theta_{\phi (j)} t ) \subset  v\ast ( \theta_v t )$. Then, (\ref{HSbindef}) entails that $n\! -\! 1 \leq  \max_{v\in B} \HS(\theta_v t )$, which completes the proof of (\ref{mirsli0}). 
 
\smallskip

Let us now prove $(ii)$. To that end, denote by $\overline{u}$ the $\leq$-minimal leaf of $\theta_u t$ and set $t_{\geq u}  =  \{ v  \in  t \, :\, v  \geq  \overline{u} \} \cup \lgeo \varnothing, \overline{u} \rgeo$. By definition, $t =  t_{\leq u} \cup t_{\geq u}$ and $\lgeo \varnothing, u \rgeo =    t_{\leq u} \cap t_{\geq u}$. Then, note that $v \ast (\theta_v t)  \subset   t_{\geq u}$ for all $v \in  B$, and that $(t_{\geq u})^\star =  t^\star_{\leq \overline{u}^\star}$. Thus, (\ref{monotony_HS}) and $(i)$ imply that
\begin{equation}
\label{mirsli}
\HS(t) \leq   1+ \max \big( \HS(t_{\leq u}), \HS(t_{\geq u}) \big) =1+ \max \big( \HS(t_{\leq u}), \HS(t^\star_{\leq \overline{u}^\star}) \big).
\end{equation}
Next set $m\!+\!1 = \# t_{\leq u}  =  \# \{ v \in  t\, :\,  v \leq  u \}$ and 
$m^\prime\!+\! 1 = \# t_{\geq u}  =  \# \{ v \in  t^\star\, :\,  v \leq  \overline{u}^\star \}$.  
Observe that $R_m t =  t_{\leq u}$ and $R_{m^\prime} t^\star =  t^\star_{\leq\overline{u}^\star}$. 
Moreover, $\#t  =  \# t_{\leq u} + \# t_{\geq u}  - \# \lgeo \varnothing , u \rgeo  =  m+ m'  + 1 - |u|$. 
Thus, $m+ m^\prime  <  \# t + |t|$. If  
$2m\geq\#t+|t|$, then $m^\prime  <  \frac{1}{2}  (\# t + |t|)  \leq  m$ and $R_{m^\prime} t^\star \! \subset \! R_m t^\star$. By (\ref{monotony_HS}), we get $\HS(R_{m^\prime}t^\star)\! \leq \!   \HS(R_{m}t^\star)$ and we obtain (\ref{keybnd}) from (\ref{mirsli}). 

\smallskip

The point $(iii)$ is an easy consequence of $(ii)$: we leave the details to the reader. 
\end{proof}
 
\subsection{Recursive equation for the distribution}
\label{sec:recursive}

Let us fix $\mu$ a probability measure on $\bbN$. Recall from (\ref{critimu}) that we say that $\mu$ is critical and non-trivial when $\sum_{k\in\bbN}k\mu(k)=1$ and $\mu(0)>0$. In this section, we state in Proposition~\ref{equHS} a recursive equation satisfied by the tail of the Horton--Strahler number of $\GW(\mu)$-trees that is the starting point of the analysis of their asymptotic behavior. This equation involves the functions $\varphi$ and $\psi$ introduced in (\ref{psi_def}), that are given by $\varphi(s)=\sum_{k\in\bbN}s^k\mu(k)$ and $\psi(s)=\varphi(1-s)-(1-s)$ for $s\in[0,1]$.

\begin{proposition}
\label{equHS}
Assume that $\mu$ is critical and non-trivial and let $\tau$ be a $\GW(\mu)$-tree.  For all $n\in \mathbb N$, set $q_n= \P(\HS(\tau)>n)$. Then it holds that
\begin{equation}
\label{rec_tails_HS} 
1-q_0= \frac{\mu (0)}{1-\mu(1)} \quad \text{ and } \quad 1-q_{n+1}=\varphi(1-q_n)+(q_n-q_{n+1})\varphi'(1-q_n), \quad n\in \mathbb N.
\end{equation}
This equation can be rewritten in terms of $\psi$ as $q_n - q_{n+1} =   \psi (q_n)/\psi' (q_n) $. 
\end{proposition}

\begin{proof}
By Definition~\ref{defHS} of the Horton--Strahler number, $\HS(\tau) =  0$ if and only if $k_\varnothing (\tau) =  0$ or $(k_\varnothing (\tau)\! = \! 1 ;  \HS(\theta_{(1)} \tau )\! = \! 0)$. Thus, by Definition~\ref{GWfordef} of $\GW(\mu)$-trees, we get $\P(\HS(\tau) \! = \! 0) =  \mu (0)+ \mu(1) \P(\HS(\tau) \! = \! 0)$, which gives the first equality in (\ref{rec_tails_HS}). 

Let us prove the recursive relation in (\ref{rec_tails_HS}). Let $n\in\mathbb{N}$. By Definition~\ref{defHS}, $\HS(\tau )  \leq  n\!+\!1$ if and only if $\HS(\theta_{u}\tau)  \leq  n$ for all children $u$ of $\varnothing$ in $\tau$ (if any) with the possible exception of one child $v$, which may satisfy $\HS(\theta_v\tau)=n\!+\!1$. More precisely, we have
\begin{equation}
\label{decomHStau}
\un_{\{ \HS(\tau)\leq n+1\}}= \!\!\!\!\!\! \prod_{\quad 1\leq j \leq k_\varnothing(\tau)}\!\!\!\!\!  \I{\HS(\theta_{(j)}\tau)\leq n }
 +\!\!\! \!\!\!\!\!\! \sum_{\quad 1\leq j \leq k_\varnothing(\tau)} \!\!\!\! \I{ \HS(\theta_{(j)}\tau)=n+1} \!\!\! \!\!\!\!\!\! \prod_{\substack{\quad 1\leq i\leq k_\varnothing(\tau)\\ i\neq j}}\!\!\!\!  \I{\HS(\theta_{(i)}\tau)\leq n } . 
\end{equation}
As $\varphi$ is the generating function of $\mu$, taking the expectation yields (\ref{rec_tails_HS}) by Definition~\ref{GWfordef}.
\end{proof}

\begin{remark}
\label{distribution_HS_exact_stable}
While it seems difficult to solve (\ref{rec_tails_HS}) explicitly in general, it can be done for the so-called 
\emph{$\alpha$-stable offspring distribution $\mu_\alpha$}, $\alpha  \in (1,2]$, 
whose generating function is \[\forall s\in[0,1],\quad \varphi_\alpha(s):=\sum_{k\in\bbN}s^k\mu_\alpha(k)=s+\tfrac{1}{\alpha}(1-s)^\alpha.\]
Indeed, if $\tau_\alpha$ is a $\GW(\mu_\alpha)$-tree then the law of $\HS(\tau_\alpha)$ is a geometric with parameter $\frac1\alpha$, i.e.
\[\P (\HS(\tau_\alpha)=n)= \tfrac{1}{\alpha} \big(1\! -\! \tfrac{1}{\alpha}\big)^{n}, \quad n\in \mathbb N.\]
This is explicitly proved in Kovchegov \& Zaliapin~\cite[Lemma 10]{kovchegov} and earlier for the $\alpha\! = \!  2$ case, see Burd, Waymire \& Winn~\cite[Proposition 2.5]{burd}.

Since their emergence in Zolotarev~\cite{zolotarev}, stable offspring distributions have appeared in numerous papers and proved to be of the utmost importance among all critical offspring distributions. This is due to their nice combinatorial properties~\cite{marchal}, their unique invariances by pruning~\cite{duquesne_winkel}, and their close links with universal scaling limits of random trees~\cite{levytree_DLG}.
\end{remark}

Motivated by Lemma~\ref{equHS}, we end this section with a lemma that lists basic properties of $\psi$, that are useful to analyse (\ref{rec_tails_HS}). Furthermore, we give the consequences for the functions $\Lambda$ and $\Upsilon$, which have been defined in (\ref{Ups}) in terms of $\psi$ as follows: 
\[\forall s\in (0, 1), \quad \Lambda(s)=\frac{s\psi'(s)}{s\psi'(s)-\psi(s)} \quad \text{ and } \quad \Upsilon(s)=\int_s^1\frac{\dd r}{r\ln\Lambda(r)}  .\]

\begin{lemma}
\label{psi'(s)_psi(s)/s}
If $\mu$ is critical and non-trivial, then the following holds true. 
\begin{itemize}
\item[$(i)$] The function $\psi $ is nonnegative, increasing, strictly convex, and analytic on $(0, 1]$. Moreover, $\psi'$ is increasing, concave, and $\psi (0) =  \psi' (0) =  0$. 

\item[$(ii)$] For all $s\in (0, 1]$, it holds that $\frac{1}{2}\psi'(s)\leq \psi(s)/s \leq \psi'(s)$.

\item[$(iii)$]  The functions $\Lambda$ and $\Upsilon$ are well-defined on $(0,1)$ and $\Lambda \geq 2$. Moreover, $\Upsilon$ is continuous, positive, and decreasing on $(0,1)$.
\end{itemize}
\end{lemma}

\begin{proof} The point $(i)$ is elementary.  
The upper bound in $(ii)$ is a consequence of the convexity of $\psi$. To obtain the lower bound in $(ii)$, we apply \emph{Hermite-Hadamard inequality} that asserts that for all convex functions $f: [a, b] \to  \mathbb R$, it holds
\begin{equation}
\label{hermhada}
f \big( \tfrac{1}{2}(a+b) \big) \leq \frac{1}{b - a} \int_a^b \! f(t)\, \dd t \, \leq \tfrac{1}{2} \big( f(a) + f(b) \big).
\end{equation}
Since $ \frac{1}{s} \psi (s) =  \frac{1}{s}\int_0^s\psi'(s)\, \dd s$, applying (\ref{hermhada}) to $-\psi'$ yields that $ \frac{1}{s} \psi (s)\geq \frac{1}{2}\psi'(s)$ and so $(ii)$. Then, $(ii)$ entails that $\Lambda(s)\geq 2$ for all $s\in(0,1)$, and $(iii)$ follows immediately.
\end{proof}

\subsection{Tail estimates of joint laws}
\label{estimates_tech}
In this section, we provide four estimates of the tail of the joint distribution of the Horton--Strahler number of a Galton--Watson tree $\tau$ with either its size $\#\tau$, its height $|\tau|$, or its maximal out-degree $\Delta(\tau)$. As in (\ref{critimu}) and Section~\ref{sec:recursive}, fix a critical and non-trivial probability measure $\mu$ on $\bbN$, and recall from (\ref{psi_def}) that $\varphi(s)=\sum_{k\in\bbN}s^k\mu(k)$ and $\psi(s)=\varphi(1-s)-(1-s)$ for $s\in[0,1]$. Moreover, observe that $\psi' (1) =  1\! -\! \varphi' (0)  =  1\! -\! \mu(1) > 0$ since $\mu$ is non-trivial. Finally, it is convenient to write
\[q_n = \P(\HS(\tau)>n), \quad n \in \{ -1\} \cup \mathbb N,\]
where $q_{-1} =  1$ obviously.

\begin{proposition}
\label{first_moment_size_general}
Let $\tau$ be a $\GW(\mu)$-tree. If $\mu$ is critical and non-trivial then
\begin{equation}
\label{euhenn}
\E \big[\#\tau\I{\HS(\tau)= 0 } \big] = \frac{ \mu(0)}{(1-\mu (1))^2}\quad\text{ and }\quad\psi'(q_{n-1})\E \big[ \#\tau \I{\HS(\tau)\leq n }\big] \leq 2, \quad n\in \mathbb N. 
\end{equation}
\end{proposition}

\begin{proof}
 Although $\E [\# \tau ]  =  \infty$, let us first prove that $e_{n} :=  \E  [\#\tau \un_{\{ \HS(\tau)\leq n\} }] < \infty$. By Definition~\ref{defHS} of the Horton--Strahler number, if $\HS(\tau)\leq n$ then $\HS(\theta_u\tau)\leq n$ for all children $u$ of $\varnothing$ in $\tau$ (if any) and $\HS(\theta_u\tau)= n$ for at most one child. 
This implies that for all $m, n\in \bbN$,
\begin{multline}
\label{fini1}
\min(m,\# \tau) \I{ \HS(\tau)\leq n } \leq 1+\sum_{i=1}^{k_\varnothing(\tau)}  \min\big(m,\#\theta_{(i)}\tau \big) \I{\HS(\theta_{(i)}\tau)\leq n-1 }  \\
+\sum_{i=1}^{k_\varnothing(\tau)} \min\big( m, \#\theta_{(i)}\tau \big)  \I{\HS(\theta_{(i)}\tau)\leq n} \prod_{\substack{1\leq j\leq k_\varnothing(\tau) \\ j\neq i}} \I{\HS(\theta_{(j)}\tau)\leq n-1}
\end{multline}
which makes sense even when $n\! = \! 0$, the first sum in the right-hand side of (\ref{fini1}) being then null. 
To simplify notation, set $e_{n} (m) =  \E  [\min( m,\#\tau ) \I{ \HS(\tau)\leq n}]$ for all integers $m\! \geq \! 0$ and $n\! \geq \! -1$ (with $e_{-1} (m)  =  0$). 
Taking the expectation in (\ref{fini1}) gives $e_{n} (m)\! \leq \! 1+e_{n-1} (m)+ e_{n} (m) \varphi'(1\! -\! q_{n-1} )$.   
It easily implies that $e_n  \leq (1+ e_{n-1})/(1\! -\! \varphi'(1\! -\! q_{n-1} ))$ because $\lim_{m\to \infty} e_n(m) =  e_n$. This recursively entails that $e_n <  \infty$ for all integers $n \geq  -1$. 

To simplify notation, set $k_\varnothing  =  k_\varnothing (\tau)$, 
$\tau_j =  \theta_{(j)} \tau$ and $\HS_j  =  \HS ( \theta_{(j)} \tau)$ for all $1\! \leq \! j\! \leq \!  k_\varnothing (\tau)$. 
We first compute $e_0$ by observing that $\# \tau \I{ \HS(\tau)  =  0} =  \I{k_\varnothing =0}+   \I{k_\varnothing =1} (1+ \tau_1) \I{\HS_1 = 0 }$. Taking the expectation entails that $e_0 =  \mu (0)+ \mu(1) (1\! -\! q_0 + e_0)$. By (\ref{rec_tails_HS}), this becomes
\begin{equation}
\label{ezero}
e_0= \E  [\#\tau \I{\HS(\tau)= 0}] =\frac{1-q_0}{1-\mu(1)}= \frac{\mu(0)}{(1-\mu (1))^2}.
\end{equation}
Next, by using the fact that $\# \tau  =  1+ \sum_{1\leq j\leq k_\varnothing } \# \tau_j $ and the decomposition (\ref{decomHStau}), we get 
\begin{multline*}
\label{fini2}
\# \tau \I{ \HS(\tau)\leq n+1 } = \I{ \HS(\tau)\leq n+1 }  + \sum_{j=1}^{k_\varnothing} \#\tau_j \I{ \HS_j \leq n } \prod_{\substack{1\leq i\leq k_\varnothing\\ j\neq i}} \I{\HS_i \leq n }\\
+\sum_{j=1}^{k_\varnothing} \#\tau_j   \I{ \HS_j = n+1 } \prod_{\substack{1\leq i\leq k_\varnothing\\ j\neq i}}  \I{\HS_i \leq n }    + \sum_{\substack{1\leq i, j\leq k_\varnothing\\ j\neq i}}  \# \tau_j  \I{\HS_j \leq n }\I{ \HS_i =n+1 } \prod_{\substack{1\leq l\leq k_\varnothing\\ l\neq i,j }} \I{\HS_l \leq n}.
\end{multline*} 
Taking the expectation term-by-term, we obtain that 
\[e_{n+1}=1-q_{n+1}+e_n\varphi'(1-q_n)+(e_{n+1}-e_n)\varphi'(1-q_n)+e_n(q_n-q_{n+1})\varphi''(1-q_n)\]
for all $n\in\mathbb{N}$. Recall from (\ref{psi_def}) that $\psi (s) \! = \! \varphi (1\! -\! s) \! -\! 1+s$. By Proposition~\ref{equHS}, we find
\[e_{n+1}\psi'(q_n)=1-q_{n+1}+e_n (q_n-q_{n+1})\psi''(q_n)= 1-q_{n+1}+e_n \frac{\psi''(q_n)}{\psi'(q_n)}\int_0^{q_n}\!  \psi'(s)\,\dd s\]
since $\psi(0) =  0$ by Lemma~\ref{psi'(s)_psi(s)/s} $(i)$. Still from Lemma~\ref{psi'(s)_psi(s)/s}, we know that $\psi'$ is concave, so we get that $\psi'(s) \leq \psi'(q_n)  -  (q_n - s)\psi''(q_n)$. Thus, if we set $x:= q_n\psi''(q_n) / \psi'(q_n)$ then
\[\frac{\psi''(q_n)}{\psi'(q_n)}\int_0^{q_n}\! \psi'(s)\, \dd s\leq q_n \psi''(q_n) -  \frac{(q_n\psi''(q_n))^2}{2\psi'(q_n)}= \tfrac{1}{2} \psi'(q_n) \big( 1 - (1\!  -\! x)^2 \big).\]
Note that $x=\psi''(q_n) / (\psi'(q_n)/q_n)$ belongs to $[0, 1]$ since $\psi'$ is concave. Thus, we get 
\[ x_{n+1}:= e_{n+1}\psi'(q_n)\leq 1+ \tfrac{1}{2} e_n \psi'(q_n) \leq  1+ \tfrac{1}{2} e_n \psi'(q_{n-1}) = 1+   \tfrac{1}{2} x_n\]
since  $(q_n)$ is decreasing and $\psi'$ is increasing. This entails that $x_n \leq 2 -2^{-n} (1-x_0)$, which leads to (\ref{euhenn}) 
because $x_0 =  \psi'(1)e_0 =  (1\! -\! \mu(1))e_0 = 1 - q_0 <  1$ by (\ref{ezero}).
\end{proof}

For any $t\in \mathbb T$, set $Z(t)=\max \big\{ |u|\, :\, u \in t\text{ such that } \HS(\theta_u t) = \HS(t) \big\}$. Note that $Z(t) \leq |t|$, where recall from (\ref{heightdelta}) that $|t|$ stands for the height of $t$. 
\begin{proposition}
\label{min_hauteur_selon_H_general}
Let $\tau$ be a $\GW(\mu)$-tree. Assume that $\mu$ is critical and non-trivial. Then $\P (Z (\tau )  \geq  m \, \big| \, \HS(\tau) \! =\!  n)  =  (1 - \psi'(q_{n-1}))^m$ for all $m, n\!\in\! \mathbb N$. Furthermore, for all $\lambda\! >\! 0$,
\begin{equation}
\label{pseudogeo}
\limsup_{n\rightarrow\infty}\P\big(  \psi' (q_{n-1})|\tau|\leq \lambda \,  \big|\,  \HS(\tau)= n \big)\leq 1-e^{-\lambda}.
\end{equation}
\end{proposition}

\begin{proof}
By Definition~\ref{defHS}, for all integers $n \geq  0$ and $m \geq  1$, it holds that
\[\I{ Z(\tau)\geq m\, ;\, \HS(\tau)=n}=\sum_{i=1}^{k_\varnothing(\tau)}\I{ Z(\theta_{(i)}\tau)\geq m-1\,  ; \,  \HS(\theta_{(i)}\tau)=n}  \prod_{\substack{1\leq j \leq k_\varnothing(\tau)\\ j\neq i}}  \I{\HS(\theta_{(j)}\tau)\leq n-1}. \] 
By taking the expectation, we get
\[\P\big( Z(\tau) \geq  m \, ;\, \HS(\tau) = n \big) = \P\big( Z(\tau) \geq  m - 1\, ;\, \HS(\tau) = n \big)\,  \varphi' (1 - q_{n-1}) ,\]
which yields the desired equality as $ \varphi' (1 \!- q_{n-1}) =  1 \!- \psi'(q_{n-1})$. Since $\psi' (q_{n-1}) \rightarrow \! \psi'(0)\! =\! 0$, $\lim_{n\rightarrow\infty}\P (Z (\tau )  \geq  \lambda /  \psi' (q_{n-1}) \, | \, \HS(\tau) \! = \! n)  =  e^{-\lambda}$. This implies (\ref{pseudogeo}) since $Z(\tau)  \leq  |\tau|$.  
\end{proof}

In \cite{brandenberger}, Brandenberger, Devroye \& Reddad study the Horton--Strahler number of a Galton--Watson tree 
conditioned to have exactly $n$ vertices under the assumption that the variance of the offspring distribution 
is finite. To that end, they use a (little more than) local convergence of this tree towards the corresponding size-biased tree. We adapt and extend this idea in a more general context using only 
the Many-To-One Principle to get the following.  
\begin{proposition}
\label{maj_hauteur_selon_H_general}
Let $\tau$ be a $\GW(\mu)$-tree. Assume that $\mu$ is critical and non-trivial. For all integers $n ,m \in  \mathbb N$, if $2\psi'(\P(|\tau|  \geq  \lfloor n/2\rfloor ))  \leq  \psi'(q_m)$ then it holds that
\begin{equation}
\label{heightcond}
\P\big(\HS(\tau)  \leq  m\, \big|\, |\tau|  \geq  n\big) \leq \exp \big(  \!-\! \tfrac{1}{8}\,  n \psi'(q_m)\big). 
\end{equation}
\end{proposition}
\begin{proof}
For all $n\in  \mathbb N$, denote by $\varphi_n$ the $n$-iterate of $\varphi$ with the convention $\varphi_0 =  \mathrm{Id}$. It is classic that $\P(|\tau|  <  n)  =  \varphi_n (0)$. Then, observe the following: the $\leq$-smallest vertex of $\tau$ at height $n$ is the only vertex $u\in\tau$ such that $|u|=n$ and such that for all $v \in \tau$ with $v <  u$ and $\overleftarrow{v} \prec u$, 
we have $|\theta_v\tau|+|v|<n$.
Moreover for all $p \in  \mathbb N$, (\ref{monotony_HS}) implies that if $\HS(\tau)\leq m$ then $\HS(\theta_v\tau)\leq m$ for all $v\in\tau$. Therefore,
 \begin{equation}
 \label{1rsthght}
 \I{ |\tau|  \geq  n \, ; \, \HS(\tau)  \leq  m} \leq 
 \sum_{u\in\tau}\I{ |u|=n } \!\!\!\! 
  \prod_{\substack{w\in\tau,i\geq 1\\ w*(i)\preceq u}}\!\! 
  \Big( \prod_{j=1}^{i-1}\I{|\theta_{w*(j)}\tau|+|w*(j)|<n} \Big) \Big(\prod_{j=i+1}^{k_w(\tau)}\I{\HS(\theta_{w*(j)}\tau)\leq m}\Big) .
 \end{equation} 
Here we adopt the following convention: a product over an empty set of indices is taken equal to $1$. Recall Definition~\ref{IGWdef} of the size-biased $\GW(\mu)$-tree $\tau_\infty$ and recall from (\ref{bivari}) in Remark~\ref{jntlw} the joint law of the number of left/right siblings of individuals on the infinite line of descent. We now use the Many-To-One Principle (Proposition~\ref{many-to-one}) after taking the expectation in (\ref{1rsthght}) to get
\begin{equation}
\label{convex_tool_1}
\P\big( |\tau|  \geq  n \, ; \, \HS(\tau) \leq  m \big)\leq \prod_{p=0}^{n-1}\frac{\varphi (1\! -\! q_m) - \varphi(\varphi_p(0))}{1\! -\! q_m  -  \varphi_p(0)}.
\end{equation}
We now use convexity properties of $\varphi$ and $\psi$ given in Lemma~\ref{psi'(s)_psi(s)/s} to get an upper bound of the right-hand side of (\ref{convex_tool_1}). 
First observe that the convexity of $\varphi$ implies that for all real numbers $s,r  \in  [0, 1]$ such that $s  \leq r $, we have 
$\frac{\varphi(r)  - \varphi(s) }{r - s} \leq \frac{1 - \varphi(s)}{1-s}$. Therefore,  
\begin{equation}
\label{convex_tool_2}
\prod_{p=0}^{\lfloor n/2 \rfloor-1}\frac{\varphi (1\! -\! q_m) - \varphi(\varphi_p(0))}{1\! -\! q_m  -  \varphi_p(0)} \leq 
\prod_{p=0}^{\lfloor n/2 \rfloor-1}\frac{1 - \varphi(\varphi_p(0))}{1 - \varphi_p(0)}= 1 - 
\varphi_{\lfloor n/2 \rfloor} (0)\; .
\end{equation}
To get an upper bound of $\frac{\varphi (1 - q_m) - \varphi(\varphi_p(0))}{1 - q_m  -  \varphi_p(0)}$ when 
$p \geq \lfloor n/2 \rfloor$, we perform the following computation: let $s, r\in [0, 1]$ and suppose that $2\psi' (s)  \leq  \psi'(r) $, then
\begin{eqnarray}
\frac{\varphi(1\! -\! r)  - \varphi(1\! -\! s) }{s - r}& = & \frac{1 - \varphi(1\! -\! s) }{s}+
\frac{\varphi(1\! -\! s)- 1 }{s} + \frac{\varphi(1\! -\! r)  - \varphi(1\! -\! s) }{s - r} \nonumber  \\
& =&  \frac{1 - \varphi(1\! -\! s) }{s}  +\frac{\psi (s)}{s}   - \frac{\psi(r) - \psi(s)}{r - s} \nonumber \\
& = &  \frac{1 - \varphi(1\! -\! s) }{s}  +\frac{\psi (s)}{s} - \frac{1}{r - s} \int_s^r \!\! \psi'(x)\, \dd x  \nonumber \\
& \leq & \frac{1 - \varphi(1\! -\! s) }{s} +\psi'(s) - \tfrac{1}{2} \big(\psi'(r) + \psi'(s)  \big)\nonumber \\
& \leq &  \frac{1 - \varphi(1 - s) }{s} -\tfrac{1}{4} \psi'(r) 
 \leq   \frac{1 - \varphi(1 - s) }{s} \big( 1 - \tfrac{1}{4}  \psi'(r) \big) . \label{concav}
\end{eqnarray}
Here, we have used Hermite-Hadamard inequality (\ref{hermhada}) on the convex function $-\psi'$, and the two convexity inequalities $\frac{1}{s}\psi(s)\leq \psi'(s)$ and $0\leq \frac{1}{s} (1\! -\! \varphi(1\! -\! s) )  \leq  \varphi'(1)=1$.  

Let $m, n  \in  \mathbb N$ such that $2\psi'( 1\! - \! \varphi_{\lfloor n/2 \rfloor} (0))  \leq  \psi' (q_m) $. Then, for all $p \geq  \lfloor n/2 \rfloor$, we have
$2\psi'( 1\! - \! \varphi_p(0))  \leq  \psi' (q_m) $. Applying (\ref{concav}) successively with $s =    1\! - \! \varphi_p(0)$ and $r =  q_m$ gets us 
\begin{align*}
 \P(|\tau|\geq n\,;\, \HS(\tau)\leq m)&\leq  
\big( 1 - \varphi_{\lfloor n/2 \rfloor} (0) \big)  \big( 1 - \tfrac{1}{4} \psi'(q_m) \big)^{n/2}   \prod_{p=\lfloor n/2 \rfloor }^{n-1}  \frac{1 - \varphi_{p+1}(0)}{1 -  \varphi_p(0)}\\
& = \big( 1 - \tfrac{1}{4} \psi'(q_m) \big)^{n/2} 
 \P (|\tau|  \geq  n), 
\end{align*}
by (\ref{convex_tool_1}) and  (\ref{convex_tool_2}). This easily entails (\ref{heightcond}) since $\ln (1\! -\! x) \leq -x$ for all $x\in [0, 1)$. 
 \end{proof}

\noi
Although Proposition~\ref{maj_hauteur_selon_H_general} holds under the sole assumption that $\mu$ is critical and non-trivial, its application requires knowing the behavior of the tail of the height of the $\GW(\mu)$-tree. When $\mu$ is in the domain of attraction of a stable law of index $\alpha\!\in\!(1,2]$, Proposition~\ref{lockemp} $(iv)$ provides such information, and Proposition~\ref{maj_hauteur_selon_H_general} then yields the more convenient result below.
\begin{corollary}
\label{maj_hauteur_selon_H_stable}
Let $\tau$ be a $\GW(\mu)$-tree. Assume that $\mu$ is critical and non-trivial. If $\mu$ belongs to the domain of attraction of a stable law of index $\alpha\in(1,2]$, then there exists a constant $C_\mu\in(0,\infty)$ that only depends on $\mu$ such that for all integers $n ,m \in  \mathbb N$,
\begin{equation}
\label{heightcond_stable}
\P\big(\HS(\tau)  \leq  m\, \big|\, |\tau|  \geq  n\big) \leq C_\mu \exp \big(  \!-\! \tfrac{1}{8}\,  n \psi'(q_m)\big). 
\end{equation}
\end{corollary}
\begin{proof}
To simplify notation, set $s_n=\P(|\tau|\geq \lfloor n/2\rfloor)$ for all $n\in\bbN$. By Proposition~\ref{lockemp} $(iv)$, the constant $2c_\mu:=\sup_{n\in\bbN}(n+1)\frac{\psi(s_n)}{s_n}$ is finite and positive. Then, we get from Lemma~\ref{psi'(s)_psi(s)/s} $(ii)$ that $2\psi'(s_n)\leq 8c_\mu/(n+1)$ for all $n\in\bbN$, by definition of $c_\mu$. Let us set $C_\mu=e^{c_\mu}\in(0,\infty)$ and let $n,m\in\bbN$. If $\psi'(q_m)\geq 8 c_\mu/(n+1)$, then $2\psi'(s_n)\leq \psi'(q_m)$ so Proposition~\ref{maj_hauteur_selon_H_general} yields (\ref{heightcond_stable}) because $1\leq C_\mu$. Otherwise, $\psi'(q_m)<8c_\mu/(n+1)$ and $C_\mu \exp(-\frac{1}{8}n\psi'(q_m))\geq 1$ by choice of $C_\mu$. Thus, (\ref{heightcond_stable}) clearly holds in that case too.
\end{proof}

Our last estimate relies on the same intuition as Proposition~\ref{maj_hauteur_selon_H_general}: if the maximal out-degree of a Galton--Watson tree is large, then the tree contains several independent copies of itself. 
\begin{proposition}
\label{min_H_selon_D} 
Let $\tau$ be a $\GW(\mu)$-tree and assume that $\mu$ is critical and non-trivial. Recall from (\ref{heightdelta}) that $\Delta (\tau)$ stands for the maximal out-degree of $\tau$. 
Then, for all integers $n,m\geq 1$ such that $\P(\Delta(\tau)\geq n)>0$, the following inequality holds true:
\begin{equation}
\label{Deltesti}
\P\big(\HS(\tau)  \leq   m\, \big|\, \Delta(\tau)  \geq  n\big)\leq e^{-nq_m}.
\end{equation}
Moreover, for all $\lambda >0$, we also have 
\begin{equation}
\label{Deltesti2}
\E\big[e^{-\lambda \# \tau } \, \big|\, \Delta(\tau)  \geq  n\big] \leq  \E \big[ e^{-\lambda \# \tau }\big]^n .
\end{equation}
\end{proposition}
\begin{proof} When $\HS(\tau) \leq m$ and $\Delta (\tau) \geq  n$, we decompose $\tau$ along the ancestral line of the $\leq$-first vertex $u$ with $k_u (\tau) \! \geq \! n$ and, by (\ref{monotony_HS}), we see $\HS( \theta_{u \ast (j)} \tau) \! \leq \! m$ for all $1\! \leq \! j\!\leq\! n$. 
Hence, 
\[\I{ \HS(\tau)\leq m \, ; \, \Delta(\tau)\geq n} \leq \sum_{u\in\tau}\I{k_u(\tau)\geq n}\Big(\!\! \!\! \!\!\prod_{\quad v\in\tau: v<u}\!\! \!\! \!\!\I{k_v(\tau)<n}\Big)\Big(\prod_{j=1}^n\I{\HS(\theta_{u*(j)}\tau)\leq m}\Big).\]
We take the expectation and apply the Many-To-One Principle (Proposition~\ref{many-to-one}) `forwards and backwards' to obtain 
\begin{align*}
\P\big(\HS(\tau)\leq m\, ;\, \Delta(\tau)\geq n\big) & \leq  \E\bigg[\sum_{u\in\tau}\I{k_u(\tau)\geq n}
 \!\! \!\! \!\! \prod_{\quad v\in\tau: v<u}\!\! \!\! \!\!  
 \I{k_v(\tau)<n} \bigg] \P \big( \HS(\tau)  \leq  m\big)^n \\
&= \P \big( \Delta (\tau)  \geq  n\big)  \big( 1 - q_m \big)^n, 
\end{align*}
which entails (\ref{Deltesti}) since $\ln (1\! -\! x) \leq -x$ for all $x\in [0, 1)$. To prove (\ref{Deltesti2}), we observe that 
\[ e^{-\lambda \# \tau } \I{  \Delta(\tau)\geq n} \leq \sum_{u\in\tau}\I{k_u(\tau)\geq n}\Big(
\!\! \!\! \!\! \prod_{\quad v\in\tau: v<u}\!\! \!\! \!\! 
\I{k_v(\tau)<n}\Big)\Big(\prod_{j=1}^n e^{-\lambda \# \theta_{u*(j)}\tau }\Big).\]
Then, as in the previous argument, we take the expectation and apply the Many-To-One Principle (Proposition~\ref{many-to-one}) `forwards and backwards' to get the desired result. 
\end{proof}

\section{Asymptotics of the Horton--Strahler number: the cases where $\alpha\in(1,2]$}
\label{pfThmsec_alpha}

In all this section, whose goal is to prove Proposition~\ref{HS_tails_alpha} and Theorem~\ref{stable_HS_order}, we fix a probability measure $\mu$ on $\bbN$ and let $\tau$ denote a $\GW(\mu)$-tree. Moreover, we shall assume that $\mu$ is critical and non-trivial as in (\ref{critimu}), i.e.~that $\sum_{k\in\bbN}k\mu(k)=1$ and $\mu(0)>0$, and that $\mu$ belongs to the domain of attraction of an $\alpha$-stable law with $\alpha \in (1,2]$. More precisely, we shall assume that $(a$-$d)$ in Proposition~\ref{limthtpsfix} hold. Finally, recall from (\ref{psi_def}) the function $\psi$ given by $\psi(s)=-(1-s)+\sum_{k\in\bbN}(1-s)^k\mu(k)$ for all $s\in[0,1]$.

\subsection{Proof of Proposition~\ref{HS_tails_alpha}}
\label{tailHSsec_alpha}

To derive the desired estimate of the tail of the Horton--Strahler number of $\tau$, let us first recall from (\ref{Ups}) the functions $\Lambda$ and $\Upsilon$, defined as follows:
\[\forall s\in (0, 1), \quad \Lambda(s)=\frac{s\psi'(s)}{s\psi'(s)-\psi(s)}\quad \text{ and } \quad  \Upsilon(s)=\int_s^1\frac{\dd r}{r\ln\Lambda(r)}.\]
Under the assumption that $\mu$ belongs to the domain of attraction of an $\alpha$-stable law with $\alpha\in(1,2]$, the following lemma characterizes the asymptotic behavior of $\Lambda$ and $\Upsilon$ near $0^+$. 

\begin{lemma}
\label{behavior_Lambda_alpha}
Assume that $\mu$ is critical and non-trivial and is in the domain of attraction of a stable law of index $\alpha\in (1,2]$. Then $\lim_{s\to 0^+}\Lambda(s)  = \frac{\alpha}{\alpha-1}$ and $\Upsilon  (s) \sim_{0^+} \log_{\frac{\alpha}{\alpha-1}}1/s$.
\end{lemma}

\begin{proof} 
By Proposition~\ref{limthtpsfix}, $\psi$ is regularly varying of index $\alpha$ at $0^+$. Then, recall that $\psi'$ is increasing, and since $\psi (s) =  \int_0^s \! \psi'(x) \, \dd x$, the Monotone Density Theorem (given as Proposition~\ref{potter} $(iv)$) implies that $s\psi'(s)\sim_{0^+}  \alpha \psi (s)$. This entails that $\lim_{s\to 0^+}\Lambda(s)  = \frac{\alpha}{\alpha-1}$ and thus $\Upsilon  (s) \sim_{0^+} \log_{\frac{\alpha}{\alpha-1}}1/s$.
\end{proof}

\begin{proof}[Proof of Proposition~\ref{HS_tails_alpha}]
Set $q_n =  \P (\HS(\tau) \! >\! n ) $ and note that $q_n\to 0$. Proposition~\ref{equHS} asserts that 
$q_n - q_{n+1} =  \psi (q_n)/ \psi'(q_n)$, namely $q^{-1}_{n+1} =  q^{-1}_n \Lambda (q_n)$. To simplify, set 
$Q_n = -\ln q_n$ so that $e^{-Q_n}\to 0^+$. Therefore, we have $Q_{n+1}=Q_n+\ln\Lambda(e^{-Q_n})$ for any $n\in \mathbb N$. By Lemma~\ref{behavior_Lambda_alpha}, $\ln\Lambda(e^{-Q_n})\longrightarrow \ln \frac{\alpha}{\alpha-1}$. Cauchy's limit theorem then yields that
\[\frac{Q_n-Q_0}{n}=\frac{1}{n}\sum_{i=0}^{n-1} \ln\Lambda(e^{-Q_i})\longrightarrow \ln\frac{\alpha}{\alpha-1},\] 
which is exactly the desired asymptotic estimate (\ref{Ytail_alpha}). 
\end{proof}

\subsection{Proof of Theorem~\ref{stable_HS_order}}
\label{pfThm1sec}

Let $(a_n)$ be as in $(d)$ in Proposition~\ref{limthtpsfix}, which we assume to be true. Here, we also suppose that $\mu$ is aperiodic, so that $\P(\#\tau=n)>0$ for all $n$ large enough. To lighten notation, let us set $\gamma=\ln \tfrac{\alpha}{\alpha-1}$ so that (\ref{Ytail_alpha}) becomes
\begin{equation}
\label{HS_tails_stable}
-\ln \P(\HS(\tau)>n)\sim \gamma n.
\end{equation}
Recall from (\ref{heightdelta}) that $|\tau|$ denotes the height of $\tau$. The proof of Theorem~\ref{stable_HS_order} relies on the relation between $\HS(\tau)$ and $|\tau|$, and consists of two parts: a lower bound and an upper bound.

\begin{proof}[Proof of the lower bound in Theorem~\ref{stable_HS_order}]
We prove for all $\varepsilon \in  (0, 1)$ that 
\begin{equation}
\label{lwbound1}
\P\big( \alpha\gamma\HS(\tau)  \leq  (1 - \varepsilon)\ln n\ \big|\ \#\tau=n\big)\xrightarrow[n\to \infty]{} 0.
\end{equation}
The idea is to apply Proposition~\ref{maj_hauteur_selon_H_general}, or rather Corollary~\ref{maj_hauteur_selon_H_stable}, and 
to switch the size of $\tau$ in the conditioning with the heigh of $\tau$. To that end, we use Proposition~\ref{height_stable_cond} $(i)$ to find 
\[\limsup_{\eta\rightarrow0}\limsup_{n\rightarrow\infty}\P\big( \, |\tau| \leq \eta \tfrac{n}{a_n}\ \big|\ \#\tau=n\big)=0.\]
Therefore, to prove (\ref{lwbound1}), we only need to prove for all $\eta  \in (0, 1)$  that
\[\limsup_{n\rightarrow\infty}\P\big( \, |\tau|> \eta \tfrac{n}{a_n} \, ; \, \alpha\gamma\HS(\tau)  \leq  (1\! -\! \varepsilon)\ln n\ \big|\ \#\tau=n\big)=0.\]
We roughly bound the conditional probability by the ratio of the probabilities as follows. 
\begin{equation}
\label{staiaip}
\P\big( |\tau| \! > \! \eta \tfrac{n}{a_n} \, ; \, \alpha\gamma\HS(\tau)  \leq  (1\! -\! \varepsilon)\ln n\, \big|\, \#\tau \! =\! n \, \big)
\leq\frac{\P\big( \alpha\gamma\HS(\tau)  \leq  (1\! -\! \varepsilon)\ln n\, \big|\, |\tau| \! > \! \eta \tfrac{n}{a_n} \, \big) }{\P(\#\tau\! =\! n)} \! .
\end{equation}
Since $a_n\sim n^{\frac{1}{\alpha}+o(1)}$ by Proposition~\ref{limthtpsfix} and Potter's bound (see Proposition~\ref{potter} $(i)$), Proposition~\ref{lockemp} $(iii)$ entails the following for the denominator of the right-hand side of (\ref{staiaip}):
\begin{equation}
\label{lwlwbbdd}
\P (\# \tau  =  n )\sim \frac{c_\alpha}{n a_n} \sim n^{-1-\frac{1}{\alpha}+o(1)}.
\end{equation}
To bound the right-hand side of (\ref{staiaip}), we want to apply Corollary~\ref{maj_hauteur_selon_H_stable} and to do so, we first control $\psi' \big( \P \big(  \HS(\tau)\!  > \! \tfrac{1-\varepsilon}{\alpha \gamma} \ln n \big)\big)$. By Potter's bound and $(c)$ in Proposition~\ref{limthtpsfix}, we get that $\ln \psi(s)\sim_{0^+}\alpha\ln s$. Lemma~\ref{psi'(s)_psi(s)/s} $(ii)$ then yields that $\ln \psi'(s)\sim_{0^+} (\alpha-1)\ln s$. Together with (\ref{HS_tails_stable}), this implies that
\begin{equation}
\label{lwbbdd}
\psi' \big( \P \big(  \HS(\tau)  >  \tfrac{1-\varepsilon}{\alpha \gamma} \ln n \big) \big)\sim n^{-(1-\varepsilon) \frac{\alpha -1}{\alpha} +o(1)}.
\end{equation}
We finally apply Corollary~\ref{maj_hauteur_selon_H_stable}, and then (\ref{lwlwbbdd}) and (\ref{lwbbdd}), to get that 
\[\P \big(  \HS(\tau)  \leq  \tfrac{1-\varepsilon}{\alpha \gamma} \ln n \ \big| \  |\tau| \! >\! \tfrac{\eta n}{a_n}  \big)\leq C_\mu\exp\Big(\! -\! \tfrac{\eta n}{8 a_n} \,  \psi'\big( \P \big(  \HS(\tau)\!  > \! \tfrac{1-\varepsilon}{\alpha \gamma} \ln n \big) \big)\Big) \leq \exp \big( -n^{\varepsilon\frac{\alpha-1}{2\alpha}}\big)\]
for all large $n$.  The previous upper bound combined with (\ref{staiaip}) and (\ref{lwlwbbdd}) implies (\ref{lwbound1}). 
\end{proof}

\begin{proof}[Proof of the upper bound in Theorem~\ref{stable_HS_order}]
We shall prove for all $\varepsilon  \in  (0, 1)$ that 
\begin{equation}
\label{upbound1}
\P\big( \alpha\gamma\HS(\tau) > (1 + 2\varepsilon)\ln n\ \big|\ \#\tau=n\big)\xrightarrow[n\to \infty]{} 0.
\end{equation}
To that end, we first show that
\begin{equation}
\label{maj_stable_atleast}
 \P\big( \alpha\gamma\HS(\tau) > (1+\varepsilon)\ln n\ \big|\ \#\tau \geq  n\big) \xrightarrow[n\to \infty]{} 0.
\end{equation}
To do this, we first write the following rough bound 
\[ \P\big( \alpha\gamma\HS(\tau) >  (1+\varepsilon)\ln n\ \big|\ \#\tau \geq  n\big) \leq\frac{\P\big(\alpha\gamma\HS(\tau)> (1+\varepsilon)\ln n\big)}{\P(\#\tau\geq n)}. \]
Then, by Propositions~\ref{lockemp} $(iii)$ and~\ref{limthtpsfix} on the one hand, and by (\ref{HS_tails_stable}) on the other, we get
\[\P(\#\tau \geq  n) \sim  \frac{c_\alpha}{a_n} \sim n^{-\frac{1}{\alpha}+o(1)} \quad  \text{ and } \quad  \P\big(\alpha\gamma\HS(\tau) >  (1\! +\! \varepsilon)\ln n\big)\sim n^{-(1+\varepsilon) \frac{1}{\alpha} +o(1)}, \]
which implies (\ref{maj_stable_atleast}).

\smallskip

Now recall from (\ref{symmtau}) that under $\P (\, \cdot \, | \, \# \tau \! =\! n)$, $\tau$ and its mirror image have the same law. Therefore, we can apply Lemma~\ref{left_right_HS} $(iii)$ with $m =  \lfloor \frac{7}{8}n \rfloor$ and thus get, for all $n$ large enough, 
\begin{multline*}
 \P\big( \HS(\tau)  >  \tfrac{1 + 2\varepsilon}{\alpha\gamma}\ln n\ \big|\ \#\tau=n\big) \leq
2\P\big( \HS(R_{\lfloor \frac{7}{8}n \rfloor }\tau)  >  \tfrac{1 + \varepsilon}{\alpha\gamma}\ln n\ \big|\ \#\tau=n\big)\\
+ 
\P \big(  |\tau| > \tfrac{1}{2}n \ \big|\ \#\tau=n\big).
\end{multline*}
By Proposition~\ref{height_stable_cond} $(i)$, $\P \big( |\tau|  > \tfrac{1}{2}n \, \big|\, \#\tau=n\big) \to 0$ when $n\to \infty$. To control the first term of the right-hand side of the previous inequality, we use Proposition~\ref{height_stable_cond} $(ii)$: to that end, recall from Proposition~\ref{luka_GW} $(ii)$ that $R_{\lfloor \frac{7}{8}n \rfloor }\tau$ is a measurable function of the Lukasiewicz path 
$(W_k(\tau))_{0\leq k \leq \lfloor \frac{7}{8}n \rfloor}$. Therefore, for all $n\in\bbN$ and $c\in(0,\infty)$,
\begin{align*}
x_n&:=  \P\big( \HS(R_{\lfloor \frac{7}{8}n \rfloor } \tau)  > \tfrac{1 + \varepsilon}{\alpha\gamma}\ln n\, \big|\, \#\tau=n\big) \\
& = \E \Big[  \un_{\big\{   \HS(R_{\lfloor \frac{7}{8}n \rfloor }\tau)  > \tfrac{1 + \varepsilon}{\alpha\gamma}\ln n \big\} } D^{(7/8)}_n \big( W_{\lfloor \frac{7}{8}n \rfloor } (\tau)\big)  \, \Big| \, \# \tau \geq  n \Big] \\
& \leq  \E \Big[  \un_{\{   \HS(\tau)  >  \tfrac{1 + \varepsilon}{\alpha\gamma}\ln n \} } D^{(7/8)}_n \big( W_{\lfloor \frac{7}{8}n \rfloor } (\tau)\big)  \, \Big| \, \# \tau \geq  n \Big] \\
& \leq  c \, \P \big( \HS(\tau)  >  \tfrac{1 + \varepsilon}{\alpha\gamma}\ln n  \, \big| \, \#\tau \! \geq \! n \big) + \E \Big[ \un_{\big\{\! D^{(7/8)}_n ( W_{\lfloor \frac{7}{8}n \rfloor } (\tau))  \, \geq c  \big\} } D^{(7/8)}_{n} \big( W_{\lfloor \frac{7}{8}n \rfloor } (\tau)\big)  \, \Big| \, \#\tau \! \geq \! n \Big]
\end{align*}
by (\ref{monotony_HS}). We let $n\to\infty$ then $c\to\infty$, and thus, by (\ref{maj_stable_atleast}) and (\ref{abs_continuity_estimate}) in Proposition~\ref{height_stable_cond}, we get $ \limsup_{n\to \infty} x_n=0$. This implies (\ref{upbound1}) and readily completes the proof of Theorem~\ref{stable_HS_order}.
\end{proof}

\section{Asymptotics of the Horton--Strahler number: the $1$-stable cases}
\label{pfThmsec_1}

The main goals of this section are to prove Proposition~\ref{HS_tails_1} and Theorems~\ref{cauchy_HS_order_atleast} and \ref{cauchy_HS_order_equal}. Throughout this section, denote by $\mu$ a probability measure on $\bbN$ and by $\tau$ a $\GW(\mu)$-tree. We both assume that $\mu$ is critical and non-trivial as in (\ref{critimu}), i.e.~that $\sum_{k\in\bbN}k\mu(k)=1$ and $\mu(0)>0$, and that $\mu$ belongs to the domain of attraction of a $1$-stable law. More precisely, we shall assume that $(a$-$c)$ in Proposition~\ref{limthCauchy1} hold, so that $L$ stands for a slowly varying function at $\infty$ such that $\mu\big([n,\infty)\big)\sim L(n)/n$. Also recall from (\ref{1stableexpl}), or $(c)$ in Proposition~\ref{limthCauchy1}, the sequence $(b_n)$. We make extensive use of the function $\psi:s\in[0,1]\mapsto -(1-s)+\sum_{k\in\bbN} (1-s)^k\mu(k)$, as defined in (\ref{psi_def}), and of functions $\Lambda$ and $\Upsilon$ given by (\ref{Ups}) as follows:
\[\forall s\in (0, 1), \quad \Lambda(s)=\frac{s\psi'(s)}{s\psi'(s)-\psi(s)} \quad \text{ and } \quad \Upsilon(s)=\int_s^1\frac{\dd r}{r\ln\Lambda(r)}.\]

\subsection{Proof of Proposition~\ref{HS_tails_1}}
\label{tailHSsec_1}

Similarly as for the cases where $\alpha\in(1,2]$, we begin by describing the behavior of $\Lambda$ at $0^+$ under the assumption that $\mu$ belongs to the domain of attraction of a $1$-stable law. The following lemma shows that $\Lambda$ varies slowly at $0^+$. 

\begin{lemma}
\label{behavior_Lambda_1}
Assume that $\mu$ is critical and non-trivial and that it belongs to the domain of attraction of a $1$-stable law. More precisely, let $L$ be such that $\mu ([n ,\infty))  \sim  
n^{-1} L(n)$ and $\overline{L}$ be the slowly varying function given by $\overline{L} (x)  =  \int_x^\infty \! y^{-1} L (y) \, \dd y$ for all $x \in  (0, \infty)$ (see Proposition~\ref{limthCauchy1}). Then,    
\[\Lambda(s)\sim_{0^+}\frac{\overline{L}(1/s)}{L(1/s)}\quad\text{ and }\quad\lim_{s\to0^+}\Lambda(s)=\infty.\]
\end{lemma}

\begin{proof} 
Let $\xi$ be a random variable whose distribution is $\mu$. Thus, $x\P(\xi>x)\sim_\infty L(x)$. Proposition~\ref{limthCauchy1} $(ii)$ asserts that $\E[\xi \I{  \xi>x  }] \sim_\infty \overline{L} (x)$ and $\psi (s) \sim_{0^+} s \overline{L} (1/s)$. The Monotone Density Theorem, i.e.~Proposition~\ref{potter} $(iv)$, then entails that $s\psi' (s) \sim_{0^+} s \overline{L} (1/s)$. 
\noi
We next consider $s\psi'(s)\! -\! \psi(s)$, $s\! \in\! (0,1)$ being fixed. To simplify notation, set $\lambda\! =\! -\ln(1\! -\! s)$ 
and first observe that
\begin{eqnarray*}
s\psi'(s)-\psi(s)& =& \E\left[1-e^{-\lambda\xi}-\tfrac{s\xi}{1-s}e^{-\lambda\xi}\right]=\E\left[\int_0^\xi \! \big( \lambda e^{-\lambda x}-\tfrac{s}{1-s} e^{-\lambda x}+\tfrac{s x\lambda}{1-s}e^{-\lambda x} \big)\,\dd x\right],\\
\text{(by Fubini)}& =& 
\tfrac{s\lambda}{1- s}\int_0^\infty\! \P(\xi>x) x e^{-\lambda x} \, \dd x\ \ - \ \ \big( \tfrac{s}{1-s} -\lambda\big)\! \int_0^\infty \! \P(\xi>x)e^{-\lambda x}\, \dd x.
\end{eqnarray*}
We estimate the second term of the right-hand side by writing
\begin{equation}
\label{floum}
\big(\tfrac{s}{1-s}+\ln(1-s)\big) \int_0^\infty\! \P(\xi>x)e^{-\lambda x}\, \dd x  \sim_{0^+}  \tfrac{1}{2} s^2\E[\xi]=\tfrac{1}{2}s^2.
\end{equation}
Next, Karamata's Abelian Theorem for Laplace transform (see Proposition~\ref{potter} $(v)$) yields 
\begin{equation}
\label{floumi}\lambda \int_0^\infty\!  x\P(\xi >x) \,  e^{-\lambda x}\,\dd x \, \underset{\lambda\to 0^+}{\sim} L(1/\lambda)
\end{equation}
Since $\lambda\! =\! -\! \ln(1\! -\! s)$, (\ref{floum}) and (\ref{floumi}) entail that $s\psi'(s)\! -\! \psi (s) \sim_{0^+} s L(1/s)$ thanks to Potter's bound (see Proposition~\ref{potter} $(i)$), which implies the desired estimate for $\Lambda$ because $s\psi' (s) \sim_{0^+} s \overline{L} (1/s)$. Finally, $\lim_{s\to\infty}\Lambda(s)=\infty$ comes from an application of Proposition~\ref{potter} $(iii)$.
\end{proof}

\begin{proof}[Proof of Proposition~\ref{HS_tails_1}]
As in the proof of Proposition~\ref{HS_tails_alpha}, set $Q_n= -\ln q_n = -\ln \P (\HS(\tau) \! >\! n )$ and note that $Q_n\to \infty$. Also recall from the proof of Proposition~\ref{HS_tails_alpha} that
 \begin{equation}
 \label{expequ}
 Q_{n+1}=Q_n+\ln\Lambda(e^{-Q_n}), \quad n\in \mathbb N. 
 \end{equation}
By Lemma~\ref{behavior_Lambda_1}, $\Lambda$ varies slowly at $0^+$. Thus, applying Potter's bound, as stated in Proposition~\ref{potter} $(i)$, to $\ell (y)  =  \Lambda (1/y)$ with $c= e^{\varepsilon}$ and $\lambda  =  e^{x-Q_n}$ ensures the following: for any $\varepsilon  \in   (0, \frac{1}{10})$, there exists $n_\varepsilon  \in \mathbb{N}$ such that for all $n\geq n_\varepsilon$ and all $x\in[Q_n,Q_{n+1}]$, it holds that
\[\big| \ln\Lambda(e^{-x})-\ln \Lambda(e^{-Q_n})\big| = \Big|\ln\Big( \frac{\Lambda (e^{-(x-Q_n)} e^{-Q_n})}{\Lambda (e^{-Q_n})}\Big)  \Big|  \leq \varepsilon+\varepsilon ( x - Q_n),\]
which further yields that $\big| \ln\Lambda(e^{-x})-\ln \Lambda(e^{-Q_n})\big|\leq  \varepsilon+\varepsilon \ln \Lambda(e^{-Q_n})$ by (\ref{expequ}). This implies 
\begin{align*}
\Big| \frac{1}{\ln \Lambda(e^{-x})} - \frac{1}{\ln \Lambda(e^{-Q_n})}\Big| & \leq \frac{\varepsilon}{\ln \Lambda(e^{-Q_n})}\cdot 
 \frac{1+ \ln \Lambda(e^{-Q_n})}{\ln \Lambda(e^{-x}) }  \\
& \leq  \frac{\varepsilon}{\ln \Lambda(e^{-Q_n})}\cdot 
 \frac{1+ \ln \Lambda(e^{-Q_n})}{ (1-\varepsilon)\ln \Lambda(e^{-Q_n})-\varepsilon} \\
 &\leq  \frac{\varepsilon}{\ln \Lambda(e^{-Q_n})}\cdot  \frac{(\ln 2)^{-1}+ 1}{ 1-\varepsilon-\varepsilon (\ln 2)^{-1}} \leq  \frac{100\varepsilon}{\ln \Lambda(e^{-Q_n})}
\end{align*}
because $\Lambda \geq 2$ by Lemma~\ref{psi'(s)_psi(s)/s} $(iii)$, and since $\varepsilon  \in   (0, \frac{1}{10})$. Therefore, by  (\ref{expequ}), we get that
 \[\int_{Q_n}^{Q_{n+1}}\Big| \frac{1}{\ln \Lambda(e^{-x})} - \frac{1}{\ln \Lambda(e^{-Q_n})}\Big| \dd x\leq 100 \varepsilon \cdot \frac{Q_{n+1} - Q_n}{\ln \Lambda(e^{-Q_n})}=100\varepsilon.\]
Hence, it follows that for all $\varepsilon  \in   (0, \frac{1}{10})$ and all $n \geq  n_\varepsilon$,
\[\bigg| n-n_\varepsilon- \int_{Q_{n_\varepsilon}}^{Q_n}\frac{\dd x}{\ln\Lambda(e^{-x})} \bigg| = \bigg|  \sum_{n_\varepsilon \leq k< n} \frac{Q_{k+1} - Q_k}{\ln \Lambda(e^{-Q_k})} - \int_{Q_{n_\varepsilon}}^{Q_n}\frac{\dd x}{\ln\Lambda(e^{-x})} \bigg| \leq  100\varepsilon (n-n_\varepsilon).\]
Thus, $\lim_{n\to \infty}\frac{1}{n}\int_{Q_0}^{Q_n}\frac{\dd x}{\ln\Lambda(e^{-x})}\! = \! 1$. This proves (\ref{Ytail_1}) after the change of variable $r=e^{-x}$. 
\end{proof}

Lemma~\ref{behavior_Lambda_1} also entails some information on $\Upsilon$, useful to show Theorems~\ref{cauchy_HS_order_atleast} and \ref{cauchy_HS_order_equal}.

\begin{proposition}
\label{Ups_SV}
Assume that $\mu$ is critical and non-trivial and that it belongs to the domain of attraction of a $1$-stable law. Then, $\Upsilon (s)=o(\ln 1/s)$ and $\ln\ln 1/s=o(\Upsilon(s))$ as $s\to 0^+$. Moreover, for all $\lambda\in  (0, \infty)$ and all $\kappa\in\mathbb{R}$, it holds that
\[\Upsilon(\lambda s\ln^\kappa 1/s)\sim_{0^+}\Upsilon(s)\quad\text{ and }\quad\Upsilon\big(s\Lambda(s)^\kappa\big)\sim_{0^+}\Upsilon(s).\]
In particular, $\Upsilon$ is slowly varying at $0^+$.
\end{proposition}

\begin{proof}
By Lemma~\ref{behavior_Lambda_1}, $\ln\Lambda(r)\longrightarrow\infty$ as $r\to 0^+$, so $\Upsilon(s)$ becomes negligible compared to $\int_s^1 \!  \tfrac{\dd r}{r} =\ln 1/s$ as $s\to 0^+$. We also know from Lemma~\ref{behavior_Lambda_1} that $\Lambda$ is slowly varying at $0^+$, so $\ln\Lambda(r)=o(\ln 1/r)$ as $r\to 0^+$ by Proposition~\ref{potter} $(i)$. This allows us to write 
\[\ln\ln 1/s\sim_{0^+}\int_s^{1/2}\frac{\dd r}{r\ln 1/r}\underset{s\to0^+}{=}o\big(\Upsilon(s)\big).\]
Next, without loss of generality, we may assume that $\lambda  \in  [1, \infty)$, so that $1\leq \lambda \ln^{|\kappa|} 1/s$ when $s$ is small enough. Recall from Lemma~\ref{psi'(s)_psi(s)/s} $(iii)$ that $\Lambda \geq 2$. Thus,
\[\big|\Upsilon(\lambda s \ln^\kappa 1/s)-\Upsilon(s)\big|\leq \int_{\frac{s}{\lambda}\ln^{-|\kappa|} 1/s}^{\lambda s\ln^{|\kappa|} 1/s}\frac{\dd r}{r\ln \Lambda(r)}\leq\frac{1}{\ln 2}\ln\big(\lambda^2 \ln^{2|\kappa|} 1/s\Big)\underset{s\to 0^+}{=}o\big(\Upsilon(s)\big),\]
which is the third desired estimate. 
Then, observe that $\lim_{s\to 0^+}s\Lambda(s)^{|\kappa|}=0$ because $\Lambda$ slowly varies at $0^+$. Another application of Proposition~\ref{potter} $(i)$, together with $\Lambda\geq 2$, entails that
$\ln\Lambda(r)\geq \frac{1}{2}\ln \Lambda(s)$ for all small enough $s$ and for all $r\in(0,1)$ such that $s\Lambda(s)^{-|\kappa|}\leq r\leq s\Lambda(s)^{|\kappa|}$. Therefore, we can complete the proof by writing
\[\big|\Upsilon(s\Lambda(s)^\kappa)-\Upsilon(s)\big|\leq \int_{s\Lambda(s)^{-|\kappa|}}^{s\Lambda(s)^{|\kappa|}}\frac{\dd r}{r\ln \Lambda(r)}\leq \frac{2}{\ln\Lambda(s)}\ln \Lambda(s)^{2|\kappa|}=4|\kappa|\underset{s\to 0^+}{=}o\big(\Upsilon(s)\big).\text{\qedhere}\]
\end{proof}

\subsection{Proof of Theorem~\ref{cauchy_HS_order_atleast}}
\label{pfThm2sec}

First, recall from Proposition~\ref{limthCauchy1} $(iii)$ that $b_n$ varies regularly with index $1$, which implies that $\ln b_n \sim \ln n$ by Proposition~\ref{potter} $(i)$. Thus, the comparisons $\Upsilon(\frac{1}{b_n}) = o(\ln n)$ and $\ln\ln n = o\big(\Upsilon(\frac{1}{b_n})\big)$ follow from Proposition~\ref{Ups_SV}. 
For the rest of the proof of Theorem~\ref{cauchy_HS_order_atleast}, it is convenient to set 
\[ \forall x\in [0, \infty), \quad q_x = \P\big( \HS(\tau) > x \big).\]
Recall from (\ref{heightdelta}) that $\Delta (\tau)$ stands for the maximal out-degree of $\tau$, which is the relevant quantity to consider in the $1$-stable cases. Note that $\P(\Delta(\tau)\!\geq\! n)\geq\mu\big([n,\infty)\big)>0$ for all $n\in\bbN$ because $n\mu\big([n,\infty)\big)$ is slowly varying. We first prove the following estimates.

\begin{lemma}
\label{estitechni}
Assume that $\mu$ is critical and non-trivial and that it belongs to the domain of attraction of a $1$-stable law. Then, the following holds for all $\varepsilon  \in  (0, 1)$ and all $\kappa  \in  (0, \infty)$. 
\begin{itemize}
\item[$(i)$] There is $n_0\in\bbN$ that depends on $\varepsilon$ and $\kappa$ such that for all integers $n \geq  n_0$, 
\[ b_n\,  q_{(1+ \varepsilon) \Upsilon(\frac{1}{b_n})} \leq  (\ln n)^{-\kappa} \quad \text{ and } \quad  b_n \,  q_{(1- \varepsilon) \Upsilon(\frac{1}{b_n})} \geq  (\ln n)^{\kappa}.\]

\item[$(ii)$] It holds that $ \lim_{n\to \infty} b_n \Lambda (\frac{1}{b_n}) \, q_{(1+ \varepsilon) \Upsilon(\frac{1}{b_n})} = 0$. 

\item[$(iii)$] It holds that $\lim_{n\to \infty}  n^\kappa \, \P \big(\HS(\tau)  \leq  (1\! -\! \varepsilon) \Upsilon \big(\frac{1}{b_n} \big) \, \big| \, \Delta (\tau)  >  \tfrac{1}{2} b_n \big) = 0$. 
\end{itemize}
\end{lemma}

\begin{proof} Let us prove $(i)$. By Proposition~\ref{HS_tails_1} and then Proposition~\ref{Ups_SV}, we have 
\[ \Upsilon \big( q_{(1\pm \varepsilon) \Upsilon(\frac{1}{b_n})} \big) \sim (1\pm \varepsilon) \Upsilon \big(\tfrac{1}{b_n}) \sim (1\pm \varepsilon) \Upsilon \big( \tfrac{(\ln b_n)^{\mp \kappa}}{b_n} \big). \]
Then, $ \Upsilon \big( q_{(1+ \varepsilon) \Upsilon(\frac{1}{b_n})} \big)  > \Upsilon \big( \tfrac{(\ln b_n)^{- \kappa}}{b_n} \big)$ and $\Upsilon \big( q_{(1 -  \varepsilon) \Upsilon(\frac{1}{b_n})} \big)  < \Upsilon \big( \tfrac{(\ln b_n)^{\kappa}}{b_n} \big)$ for all sufficiently large $n$. Since we know from Lemma~\ref{psi'(s)_psi(s)/s} $(iii)$ that $\Upsilon$ decreases, it follows that
\[b_n \, q_{(1+ \varepsilon) \Upsilon(\frac{1}{b_n})} \leq  (\ln b_n)^{-\kappa} \quad \text{ and } \quad  b_n\,  q_{(1- \varepsilon) \Upsilon(\frac{1}{b_n})} \geq  (\ln b_n)^{\kappa}.\]
By Proposition~\ref{limthCauchy1} $(iii)$, $b_n$ varies regularly with index $1$ so we have $\ln b_n \sim \ln n$ thanks to Proposition~\ref{potter} $(i)$, which implies the desired result.

We use similar arguments to prove $(ii)$: by Proposition~\ref{HS_tails_1} then Proposition~\ref{Ups_SV}, we get
\[\Upsilon \big( q_{(1+ \varepsilon) \Upsilon(\frac{1}{b_n})} \big) \sim  (1+ \varepsilon) \Upsilon \big(\tfrac{1}{b_n})  \sim (1+ \varepsilon) \Upsilon \big( \tfrac{1}{b_n} \Lambda \big(\tfrac{1}{b_n} \big)^{-1-\kappa} \big),\]
so $ \Upsilon \big( q_{(1+ \varepsilon) \Upsilon(\frac{1}{b_n})} \big)  >  \Upsilon \big( \tfrac{1}{b_n} \Lambda \big(\tfrac{1}{b_n} \big)^{-1-\kappa} \big) $ for all sufficiently large $n$, which implies that 
\[b_n \Lambda \big(\tfrac{1}{b_n}\big) \, q_{(1+ \varepsilon) \Upsilon(\frac{1}{b_n})}  \leq  \Lambda \big(\tfrac{1}{b_n} \big)^{-\kappa} .\]
The point $(ii)$ follows since $\lim_{x\to \infty} \Lambda (1/x) =  \infty$ by Lemma~\ref{behavior_Lambda_1}.

To prove $(iii)$, we first use $(i)$ with e.g.~$\kappa  =  2$. Then, Proposition~\ref{min_H_selon_D} entails that
\begin{align*}
n^\kappa\, \P \big( \HS(\tau) \leq  (1\! -\! \varepsilon) \Upsilon \big(\tfrac{1}{b_n} \big) \, \big| \, \Delta (\tau)  >  \tfrac{1}{2} b_n \big)  &\leq  n^\kappa \exp \big(\! -\! \tfrac{1}{2} b_n \, q_{(1- \varepsilon) \Upsilon(\frac{1}{b_n})}  \big) 
\\
&\leq n^\kappa \exp \big(\! -\! \tfrac{1}{2} (\ln n)^2  \big)
\end{align*}
for all sufficiently large $n$, which entails the desired result.
\end{proof}

\noi
The proof of Theorem~\ref{cauchy_HS_order_atleast} is cut into two parts: an upper bound and a lower bound. 
\begin{proof}[Proof of the upper bound in Theorem~\ref{cauchy_HS_order_atleast}]
We first prove for all $\varepsilon  \in  (0, 1)$ that 
\begin{equation}
\label{upbound2}
\P\big( \HS(\tau)  > (1 + \varepsilon) \Upsilon \big(\tfrac{1}{b_n}\big)\ \big|\ \#\tau  \geq  n\big)\xrightarrow[n\to \infty]{} 0.
\end{equation}
To that end, we use a direct upper bound, then we combine Lemma~\ref{numtau1sta} with Lemma~\ref{behavior_Lambda_1}: 
\[\P\big( \HS(\tau) > (1\! +\! \varepsilon) \Upsilon \big(\tfrac{1}{b_n}\big)\ \big|\ \#\tau  \geq  n\big) \leq \frac{q_{(1+ \varepsilon) \Upsilon (\frac{1}{b_n})}}{\P (\#\tau  \geq  n)}   \sim b_n \Lambda \big(\tfrac{1}{b_n} \big)  \, q_{(1+ \varepsilon) \Upsilon (\frac{1}{b_n})}.  \]
Next, we use Lemma~\ref{estitechni} $(ii)$ to conclude.
\end{proof}

\begin{proof}[Proof of the lower bound in Theorem~\ref{cauchy_HS_order_atleast}]
We next prove for all $\varepsilon  \in  (0, 1)$ that 
\begin{equation}
\label{lwbound2}
\P\big(  \HS(\tau)  \leq  (1 - \varepsilon) \Upsilon \big(\tfrac{1}{b_n}\big)\ \big|\ \#\tau  \geq  n\big)\xrightarrow[n\to \infty]{} 0.
\end{equation}
To that end, we first prove that 
\begin{equation}
\label{stepun}
\lim_{n\rightarrow\infty}
\P\big( \#\tau  < n \ \big| \ \Delta(\tau) \geq   2 b_n \big) \xrightarrow[n\to \infty]{} 0 .
\end{equation}
By (\ref{Deltesti2}) in Proposition~\ref{min_H_selon_D} combined with Markov inequality, for all $\lambda  \in  (0, \infty) $, it holds that
\[\P\big( \#\tau  < n \, \big| \, \Delta(\tau) \geq   2 b_n \big) \leq e^{\lambda }\E \big[ e^{-\frac{\lambda}{n} \# \tau } \big]^{\lfloor 2 b_n\rfloor }.\]
Recall from Proposition~\ref{luka_GW} $(iv)$ that $\E \big[ \exp \big(\!  -\! \frac{\lambda}{n} \# \tau \big) \big]^{\lfloor  2 b_n\rfloor } =  \E \big[ \exp \big(\!  -\! \frac{\lambda}{n} \mathtt H_{\lfloor  2b_n\rfloor} \big) \big]$. The convergence (\ref{rela_stabi_H}) with $x = 2$ then implies that $\lim_{n\to \infty}\E \big[ e^{-\frac{\lambda}{n} \# \tau } \big]^{\lfloor  2 b_n\rfloor } =  e^{-2\lambda}$. Thus,  
\[\limsup_{n\rightarrow\infty}
\P\big( \#\tau  < n \, \big| \, \Delta(\tau) \geq   2b_n \big)  \leq e^{-\lambda }  \xrightarrow[\lambda \to \infty]{} 0,\]
which entails (\ref{stepun}) as desired.

\smallskip

We next prove that for all $c \in  (0, \infty)$,
\begin{equation}
\label{stepdeuz}
\limsup_{n\rightarrow\infty}
\frac{\P\big(\Delta(\tau) \geq  cb_n\big)}{\P(\#\tau  \geq  n)}\, \leq 4 c^{-1}.
\end{equation}
Let us begin with the case where $c =  2$. Observe that 
\[u(n):= \frac{\P\big(\Delta(\tau) \geq  2b_n\big)}{\P(\#\tau  \geq  n)} =  \P\big(\Delta(\tau) \geq  2b_n \, \big| \, \#\tau  \geq  n \big) + u(n) \P\big( \#\tau  < n \, \big| \, \Delta(\tau) \geq   2 b_n \big) .\]
By (\ref{stepun}), for all sufficiently large $n$, we have $\P\big( \#\tau  < n \, \big| \, \Delta(\tau) \geq   2 b_n \big)  \leq  \frac{1}{2}$. It therefore follows that $u(n)  \leq  2  \P\big(\Delta(\tau) \geq  2b_n \, \big| \, \#\tau  \geq  n \big)  \longrightarrow  2\P (J \! \geq\!  2) =  1$ by Proposition~\ref{D_stable_cond_atleast} $(i)$. This proves (\ref{stepdeuz}) when $c =  2$. Now, recall from Proposition~\ref{limthCauchy1} $(iii)$ that $(b_n)$ is $1$-regularly varying. Then for all sufficiently large $n$, we get $cb_n  \geq  2b_{\lfloor \frac{1}{4}cn \rfloor }$ and thus 
\[\frac{\P\big(\Delta(\tau) \geq c b_n\big)}{\P(\#\tau \geq n)}\leq  \frac{\P\big(\Delta(\tau) \geq 2b_{\lfloor \frac{1}{4}cn \rfloor }\big)}{\P(\#\tau  \geq  n)} = u\big(\lfloor \tfrac{1}{4}cn\rfloor\big) 
 \frac{\P(\#\tau  \geq  \lfloor\frac{1}{4} cn\rfloor)}{\P(\#\tau \geq n)} .\]
This implies the desired bound (\ref{stepdeuz}) since we know from Lemma~\ref{numtau1sta} that $\P (\# \tau \! \geq \! n)$ varies regularly with exponent $-1$,  
and since $\, \limsup_{n\to \infty} u(\lfloor \tfrac{1}{4}cn\rfloor) \leq 1$ as proved above.

\smallskip

Finally, we complete the proof of (\ref{lwbound2}) as follows: 
\begin{eqnarray*} 
v_n &:=& \P\big(\HS(\tau)  \leq  (1\! -\! \varepsilon)\Upsilon \big( \tfrac{1}{b_n} \big)\ \big|\ \#\tau  \geq  n\big) \\
& \leq & \frac{\P \big( \HS(\tau)  \leq  (1\! -\! \varepsilon)\Upsilon \big( \tfrac{1}{b_n} \big) \, ; \, \Delta(\tau)  > \tfrac{1}{2} 
b_n\big)}{ \P ( \#\tau \geq n) } + \P \big( \Delta(\tau) \leq \tfrac{1}{2}  b_n \  \big|\   \#\tau  \geq  n\big)   \\
& \leq &  \frac{\P \big(  \Delta(\tau)  >  \tfrac{1}{2} b_n\big)}{ \P ( \#\tau \geq n) } \P \big( \HS(\tau) \leq (1\! -\! \varepsilon)\Upsilon \big( \tfrac{1}{b_n} \big)\ \big| \  \Delta(\tau) > \tfrac{1}{2}  b_n\big) + \P \big(  \Delta(\tau) \! \leq \! \tfrac{1}{2}  b_n \   \big|\   \#\tau  \geq  n\big)  . 
\end{eqnarray*} 
This first term of the above right-hand side tends to $0$ by (\ref{stepdeuz}) with $c=1/2$ and by Lemma~\ref{estitechni} $(iii)$. Moreover, Proposition~\ref{D_stable_cond_atleast} $(i)$ ensures that $\lim_{n\to \infty}\P \big(  \Delta(\tau)  \leq  \tfrac{1}{2}  b_n \,   \big|\,   \#\tau  \geq  n\big)=0$. This completes the proof of (\ref{lwbound2}), and of Theorem~\ref{cauchy_HS_order_atleast}.
\end{proof}

\subsection{Proof of Theorem~\ref{cauchy_HS_order_equal}}
\label{pfThm3sec}

In addition to the hypotheses taken at the start of Section~\ref{pfThmsec_alpha}, here we further assume that (\ref{H1_loc}) holds, i.e.~that $\mu (n) \sim n^{-2} L(n)$ where $L$ is the same slowly varying function at $\infty$ introduced above. Indeed, observe from Proposition~\ref{potter} $(ii)$ that $\mu (n) \sim n^{-2} L(n)$ implies that $\mu \big([n, \infty)\big) \sim n^{-1} L(n)$, so that $\mu$ belongs to the domain of attraction of a $1$-stable law by Proposition~\ref{limthCauchy1}. In particular, we still have $\Upsilon(\frac{1}{b_n}) = o(\ln n)$ and $\ln\ln n = o\big(\Upsilon(\frac{1}{b_n})\big)$ as in Theorem~\ref{cauchy_HS_order_atleast}. The proof of Theorem~\ref{cauchy_HS_order_equal} is separated into two parts: firstly a lower bound and 
secondly an upper bound.

\begin{proof}[Proof of the lower bound in Theorem~\ref{cauchy_HS_order_equal}]
We prove for all $\varepsilon  \in (0, 1)$ that 
\begin{equation}
\label{lwbound3}
\P\big( \HS(\tau) \leq (1\! -\! \varepsilon) \Upsilon \big(\tfrac{1}{b_n}\big)\ \big|\ \#\tau = n\big)\xrightarrow[n\to \infty]{} 0.
\end{equation}
Indeed, observe that 
\begin{eqnarray*} 
v_n &:=& \P\big(\HS(\tau) \leq (1\! -\! \varepsilon)\Upsilon \big( \tfrac{1}{b_n} \big)\ \big|\ \#\tau = n\big) \\
& \leq & \frac{\P \big( \HS(\tau) \leq (1\! -\! \varepsilon)\Upsilon \big( \tfrac{1}{b_n} \big) \, ; \, \Delta(\tau)  > \tfrac{1}{2} 
b_n\big)}{ \P ( \#\tau = n) } + \P \big(  \Delta(\tau) \leq \tfrac{1}{2}  b_n \   \big|\   \#\tau = n\big)   \\
& \leq &  \frac{\P \big(  \Delta(\tau)\! >\! \tfrac{1}{2} b_n\big)}{ n\P ( \#\tau  = n) } \, n \,  \P \big( \HS(\tau) \leq (1\! -\! \varepsilon)\Upsilon \big( \tfrac{1}{b_n} \big)\,  \big| \,  \Delta(\tau) >  \tfrac{1}{2}  b_n\big) + \P \big(  \Delta(\tau) \leq \tfrac{1}{2}  b_n \, \big|\,   \#\tau = n\big)  . 
\end{eqnarray*} 
Then recall from (\ref{local_tails_size}) in Lemma~\ref{numtau1sta} that $n\P ( \#\tau = n)\sim \P ( \#\tau \geq n)$. Thus, by (\ref{stepdeuz}), we get 
$$ \limsup_{n\to \infty}  \frac{\P \big(  \Delta(\tau) > \tfrac{1}{2} b_n\big)}{ n\P ( \#\tau = n) }\leq 8.$$
By Lemma~\ref{estitechni} $(iii)$,  $\lim_{n\to \infty}n \P \big( \HS(\tau) \! \leq \! (1\! - \varepsilon)\Upsilon \big( \tfrac{1}{b_n} \big)\,  \big| \,  \Delta(\tau) \! >\!  \tfrac{1}{2}  b_n\big)\! = \! 0$. Then, Proposition~\ref{D_stable_cond_atleast} $(ii)$ ensures that $\P \big(  \Delta(\tau)  \leq  \tfrac{1}{2}  b_n \,   \big|\,   \#\tau \! = \! n\big) \longrightarrow 0$, which finally implies (\ref{lwbound3}).
\end{proof}

\begin{proof}[Proof of the upper bound in Theorem~\ref{cauchy_HS_order_equal}]
We finally show for all $\varepsilon \in (0, 1)$ that 
\begin{equation}
\label{upbound3}
\P\big( \HS(\tau) \geq (1\! +\! \varepsilon) \Upsilon \big(\tfrac{1}{b_n}\big)\ \big|\ \#\tau  = n\big)\xrightarrow[n\to \infty]{} 0.
\end{equation}
To prove (\ref{upbound3}), we use the result of Kortchemski \& Richier~\cite[Theorem 21]{KRcauchy}, presented as Proposition~\ref{KRvervaat}, that shows that the law of the Lukasiewicz path $(W_j(\tau))_{0\leq j \leq n}$ of $\tau$ under $\P (\, \cdot \, | \, \#\tau  =  n)$ is close in total variation distance to the law of $(Z^{(n)}_j)_{0\leq j \leq n}$ as defined in (\ref{Vervdef}). 

First, let us briefly recall the definition of $Z^{(n)}$. Let $(W_n)_{n\in \mathbb N}$ be a left-continuous random walk starting at $0$ and whose jumps distribution is given by (\ref{RWW}). We recall from (\ref{minargmin}) the two integers $I_n =  -   \min_{ 0\leq j \leq n-1} W_j$ and $\sigma_n  =   \inf\{0\!  \leq \! k\! \leq \! n\!  - \! 1 : W_k \! = \! - I_n \}$. Then,
\[ Z_j^{(n)}=  W_{\sigma_n +j}  + I_n \; \, \text{if}\; \,  0\! \leq \! j\! <\! n\! -\! \sigma_n \quad \text{ and } \quad   Z_j^{(n)}=  I_n \! -\! 1 +\! W_{j-(n-\sigma_n)} \; \, \text{if}\; \, n \! \geq \! j \! \geq \! n\! -\! \sigma_n. \]
Next, we interpret Proposition~\ref{KRvervaat} in terms of trees, i.e.~we view $Z^{(n)}$ as the Lukasiewicz path of a random tree $\tau^{(n)}$. Namely, observe that $Z_0^{(n)} = 0$, $Z_n^{(n)} = -1$, and the other values of $Z^{(n)}$ are nonnegative. Moreover, maybe except at the cutting time $n\! -\! \sigma_n \! -\! 1$ when $Z^{(n)}_{n-\sigma_n}  - Z^{(n)}_{n-\sigma_n-1} =  -1 - W_{n-1}$, the jumps of $Z^{(n)}$ are larger or equal to $- 1$. Consequently, if $W_{n-1}  \leq  0$ then Proposition~\ref{luka_GW} $(i)$ entails that $Z^{(n)}$ is the Lukasiewicz path associated with a tree $\tau^{(n)}$. If $W_{n-1}  > 0$ then we take $\tau^{(n)}$ equal to the \emph{star-tree with $n\! -\! 1$ leaves}, which we denote by $\bigstar_n$; this choice is arbitrary, we could have chosen any fixed tree with $n$ vertices.
\begin{align*}
\text{If }\quad W_{n-1} \leq 0\quad&\text{ then }\quad (W_j(\tau^{(n)}))_{0\leq j \leq n}  =  (Z^{(n)}_j)_{0\leq j \leq n}\, ;\\
\text{if }\quad W_{n-1}  > 0\quad&\text{ then }\quad\tau^{(n)} =  \bigstar_n.  
\end{align*}
Moreover, observe that $(c)$ in Proposition~\ref{limthCauchy1} implies that $\P \big( \tau^{(n)} \!\!  = \!  \bigstar_n \big) =  
\P (W_{n-1} \! >\! 0)$ tends to $0$. Thus, thanks to Proposition~\ref{KRvervaat}, in order to prove (\ref{upbound3}), we only need to show that
\begin{equation}
\label{upbound4}
\P\big( \tau^{(n)} \neq  \bigstar_n \, ; \,  \HS(\tau^{(n)}) \geq  (1\! +\! \varepsilon) \Upsilon \big(\tfrac{1}{b_n}\big)\big)\xrightarrow[n\to \infty]{} 0.
\end{equation}

To do so, we discuss a decomposition of $\tau^{(n)}$ at the `cutting time' $n\! -\! \sigma_n\! -\! 1$ and use from Proposition~\ref{luka_GW} $(iv)$ the existence of an i.d.d.~sequence $(\tau_p)_{p\in \mathbb N}$ of $\GW(\mu)$-trees such that
\[\forall p\in \mathbb N, \quad  (p+W_{\mathtt H_p +j} )_{0\leq j\leq \mathtt H_{p+1}-\mathtt H_p} = \big( W_j (\tau_p) \big)_{0\leq j \leq \# \tau_p},\]
where $\mathtt{H}_{p} = \inf \{j  \in  \mathbb N\, :\, W_j = -p\}$. Let us denote by $u^{(n)}_0 \! = \! \varnothing \! < \! u^{(n)}_1 \! < \! \ldots \! < \! u^{(n)}_{n-1}$ the vertices of $\tau^{(n)}$ listed in lexicogaphic order. From the definition of $Z^{(n)}$ and thanks to Proposition~\ref{luka_GW} $(ii)$, we first derive that if $W_{n-1}  \leq  0$ then
\begin{equation}   
\label{lasttaun}
 R_{n-\sigma_n -1} \big(  \tau^{(n)}\big)= R_{n-\sigma_n -1} \big(  \tau_{I_n}\big) \quad \text{ and } \quad k_{u^{(n)}_{n-\sigma_n -1}} (\tau^{(n)} )= -W_{n-1},
\end{equation}
where we recall from (\ref{left_portion}) that $R_m (t)=R_m t$ stands for the tree consisting of the first $m+1$ vertices of $t$ taken in lexicographic order. We now look at the subtrees that are grafted either at $u^{(n)}_{n-\sigma_n-1}$ or to the right of its ancestral line: namely, the subtrees that are grafted at the following set of vertices:
\[ B = \big\{ v \in  \tau^{(n)}\, :\,  \overleftarrow{v} \preceq u^{(n)}_{n-\sigma_n-1}  \text{ and } 
 v >  u^{(n)}_{n-\sigma_n-1} \big\}.\]
Denote by $v(0) \! < \! v(1) \! < \! \ldots \! < \! v(\#B-1)$ the vertices of $B$ listed in lexicographic order. By definition of $Z^{(n)}$ and by definition of the trees $\tau_p$ as recalled above, if $W_{n-1} \leq  0$ then
\begin{equation}   
\label{posttaun}
\theta_{v(p)} \tau^{(n)} = \tau_p , \quad 0  \leq  p \leq  I_n-1=\#B-1.
\end{equation}
To prove (\ref{upbound4}), we then use Lemma~\ref{left_right_HS} $(i)$ to get that on the event $\{  W_{n-1} \leq 0 \} $, it holds that
\[\HS \big( \tau^{(n)}\big) \leq 1 +\max \big( \HS\big( R_{n-\sigma_n -1} \big(  \tau^{(n)}\big)\big), \max_{v\in B} \HS\big( \theta_v  \tau^{(n)} \big) \big).\]
Then, the identities (\ref{lasttaun}) and (\ref{posttaun}), together with the monotony property (\ref{monotony_HS}) of $\HS$, yield that
\[\HS \big( \tau^{(n)}\big) \leq 1+ \max_{0\leq p \leq I_n}  \HS\big(\tau_p \big).\]
Therefore, we can write the following computation:
\begin{align*}
x_n & := \P\Big( \tau^{(n)}  \neq \bigstar_n \, ; \,  \HS(\tau^{(n)})  \geq  (1\! +\! \varepsilon) \Upsilon \big(\tfrac{1}{b_n}\big)\Big)\\
& \leq  \P \Big( \!\! \!\! \!\!   \max_{ \quad 0\leq p \leq \lfloor 2b_n \rfloor} \!\! \!\!  \HS\big(\tau_p \big)  \geq  (1\! +\! \varepsilon) \Upsilon \big(\tfrac{1}{b_n}\big) - 1\Big)+ \P \big(I_n > 2b_n \big) \\
& \leq  (2b_n+1)\, \P \Big( \HS(\tau)  \geq (1\! +\! \varepsilon) \Upsilon \big(\tfrac{1}{b_n}\big) - 1 \Big) + \P \big( \tfrac{1}{n} \mathtt H_{ \lfloor 2b_n \rfloor} \leq 1 \big),  
\end{align*} 
since the $\tau_p$ are $\GW(\mu)$-trees. Then, recall from (\ref{rela_stabi_H}) that $\tfrac{1}{n} \mathtt H_{ \lfloor 2b_n \rfloor}\to 2$ in probability, so that $\P \big( \tfrac{1}{n} \mathtt H_{ \lfloor 2b_n \rfloor} \! \leq \! 1 \big) \to 0$. Then, 
Lemma~\ref{estitechni} $(i)$ ensures that $b_n \P \big( \HS(\tau) \! \geq \! (1\! +\! \varepsilon) \Upsilon \big(\tfrac{1}{b_n}\big)\! -\! 1 \big) \to 0$, which finally implies (\ref{upbound4}) and which completes the proof of Theorem~\ref{cauchy_HS_order_equal}.
\end{proof} 

\subsection{Examples}
\label{sec:ex}

We conclude this section by considering three examples of functions $L$ that may govern the tails of $\mu$ via the estimate $n\mu ([n , \infty)) \sim  L(n)$. Recall from Proposition~\ref{limthCauchy1} $(i)$ that for $\mu$ to be critical and non-trivial and to belong to the domain of attraction of a $1$-stable law, $L$ needs to be slowly varying at $\infty$ and such that $\int_1^\infty y^{-1} L(y) \, \dd y < \infty$. In the three examples considered, we translate the estimate of $\P(\HS(\tau)>n)$ given by Proposition~\ref{HS_tails_1} and provide an explicit asymptotic equivalent of the rescaling sequence $\Upsilon \big(\frac{1}{b_n} \big)$ within Theorems~\ref{cauchy_HS_order_atleast} and \ref{cauchy_HS_order_equal}. To lighten notation, we write in each case, for $n\in\bbN$ and $x\in (0,\infty)$,
\[Q_n=-\ln \P(\HS(\tau)>n)\quad\text{ and }\quad\overline{L}(x)=\int_x^\infty\! y^{-1}L(y)\,\dd y.\]
Before we begin, let us also recall from the start of Section~\ref{pfThm2sec} that $\ln b_n \sim \ln n$.

\begin{longlist}
\item[$(\mathbf{a})$] Let $\kappa \in (0, \infty)$. Let us consider the case where $L(x)= (\ln x)^{-1-\kappa}$ for $x\in [e, \infty)$. Indeed, this function varies slowly at $\infty$ and is such that $\int_e^\infty y^{-1} L(y) \, \dd y < \infty$. Clearly, we have $\overline{L} (x) = \frac{1}{\kappa} (\ln x)^{-\kappa}=  \frac{1}{\kappa} L(x) \ln x$ for $x\geq e$. Thus, Lemma~\ref{behavior_Lambda_1} yields that
$$ \Lambda (s) \sim_{0^+} \tfrac{1}{\kappa} \ln 1/s  \quad \text{ and } \quad \Upsilon (s) \sim_{0^+}  \frac{\ln 1/s }{\ln \ln 1/s }.$$
Proposition~\ref{HS_tails_1} then asserts that $\frac{Q_n}{\ln Q_n}\sim  n$. Thus, $\ln Q_n\sim \ln n\sim \ln b_n
$ and it follows that
\[-\ln \P (\HS(\tau)>n) \sim   n\ln n\quad \text{ and }\quad \Upsilon \big(\tfrac{1}{b_n} \big) \sim \frac{\ln n}{\ln \ln n} .\]

\item[$(\mathbf{b})$] Let $\kappa  \in  (0, 1)$. Let us consider the case where $L(x)= \exp (-(\ln x)^{\kappa})$ for $x\in [1, \infty) $. This function varies slowly at $\infty$ and verifies $\int_1^\infty y^{-1} L(y) \, \dd y \! < \! \infty$. Integrating by parts gives
\[ \overline{L} (x) =  \tfrac{1}{\kappa}(\ln x)^{1-\kappa} L(x) +  \tfrac{1-\kappa}{\kappa} \int_x^\infty \!\! 
\frac{L(y)}{y (\ln y)^{\kappa}}\,\dd y\sim_\infty \tfrac{1}{\kappa}(\ln x)^{1-\kappa} L(x),\] 
and Lemma~\ref{behavior_Lambda_1} thus implies that 
$$ \Lambda (s) \sim_{0^+} \tfrac{1}{\kappa}( \ln 1/s)^{1-\kappa} \quad \text{ and } \quad \Upsilon (s) \sim_{0^+} \tfrac{1}{1-\kappa}\cdot  \frac{\ln 1/s }{\ln \ln 1/s }.$$
Proposition~\ref{HS_tails_1} asserts that $\frac{Q_n}{\ln Q_n}\sim (1\! -\! \kappa ) n$. Thus, $\ln Q_n \sim \ln n \sim \ln b_n
$ and it follows that
\[-\ln \P (\HS(\tau)>n) \sim  (1\! -\! \kappa ) n\ln n\quad \text{ and }\quad \Upsilon \big(\tfrac{1}{b_n} \big) \sim  \frac{1}{1\! -\! \kappa}\cdot \frac{\ln n}{\ln \ln n}.\]

\item[$(\mathbf{c})$] Let us consider the case where $L(x)=\exp(-\ln x/\ln\ln x)$ for all $x\in [e^e, \infty) $. This function still slowly varies and verifies $\int_{e^e}^\infty y^{-1} L(y) \, \dd y  <  \infty$. We make the change of variable $z=\ln y/\ln\ln y$ and then an integration by parts to compute that
\[\overline{L}(x)\sim_\infty \int_{\frac{\ln x}{\ln\ln x}}^\infty \!  \!\!\!  e^{-z}\ln(z) \,\dd z\, \sim_\infty \, \Big[-e^{-z}\ln z\Big]_{\frac{\ln x}{\ln\ln x}}^\infty \, \sim_\infty \, L(x) \ln \ln x.  \]
Therefore, Lemma~\ref{behavior_Lambda_1} implies that
\[ \Lambda(s) \sim_{0^+} \ln \ln 1/s \quad \text{ and } \quad \Upsilon (s) \sim_{0^+}  \frac{\ln 1/s}{\ln \ln \ln 1/s}.\]
Proposition~\ref{HS_tails_1} then asserts that $ \frac{Q_n}{\ln \ln Q_n}   \sim n $. Thus, $\ln Q_n\sim \ln n\sim \ln b_n
$ and it follows that
\[-\ln \P(\HS(\tau)>n) \sim n \ln \ln n\quad \text{ and }\quad \Upsilon \big(\tfrac{1}{b_n} \big) \sim \frac{\ln n}{\ln \ln \ln n}.\]
\end{longlist}

\section{Extensions and open questions}
\label{sec:open}

We conclude this paper by mentioning some directions for future research that are naturally suggested by our results.

\begin{longlist}
\item[$(i)$] Theorems~\ref{stable_HS_order}, \ref{cauchy_HS_order_atleast} and \ref{cauchy_HS_order_equal} are laws of large numbers, i.e.~they give a deterministic equivalent for the Horton--Strahler number of large Galton--Watson trees. It is thus natural to look for a central limit theorem, i.e.~describe the rate of convergence and ask whether the fluctuations admit an asymptotic distribution. While we leave these questions open in general, the companion paper \cite{Kha24} answers them in the particular case of the $\alpha$-stable offspring distributions, for $\alpha\in(1,2]$, briefly presented in Remark~\ref{distribution_HS_exact_stable}. Then, the fluctuations around the equivalent $\frac{1}{\alpha}\log_{\frac{\alpha}{\alpha-1}}n$ are of constant order but they do not actually converge because of some deterministic oscillations. However, compensating for these reveals a genuine asymptotic distribution, characterized in terms of the scaling limit of the trees.

We expect that the ideas and methods developed in this paper will be key to determining the deviations from $\frac{1}{\alpha}\log_{\frac{\alpha}{\alpha-1}}n$ for general critical offspring distributions in the domain of attraction of an $\alpha$-stable law with $\alpha\in(1,2]$. Indeed, Propositions~\ref{first_moment_size_general} and \ref{min_hauteur_selon_H_general} and Corollary~\ref{maj_hauteur_selon_H_stable} are the main tools used in \cite[Proposition 4.2]{Kha24} to derive the correct constant order discussed above. However, \cite{Kha24} also relies on the exact formula $\P (\HS(\tau_\alpha)>n)= \big(1\! -\! \tfrac{1}{\alpha}\big)^{n+1}$ (see Remark~\ref{distribution_HS_exact_stable}). Thus, the main obstacle to generalizing the results of~\cite{Kha24} seems to be obtaining a more precise estimate of $\P(\HS(\tau)> n)$ than (\ref{Ytail_alpha}), which might require some additional regularity assumptions on $\psi$. In fact, we do not believe that $\P(\HS(\tau)> n)\sim \text{Cst}\cdot \big(1\! -\! \tfrac{1}{\alpha}\big)^{n}$ would universally hold, and thus expect the rate of convergence to be model-dependent.

\item[$(ii)$] Comparing Theorem~\ref{stable_HS_order} with Theorems~\ref{cauchy_HS_order_atleast} and \ref{cauchy_HS_order_equal} yields that if the critical offspring distribution $\mu$ is in the domain of attraction of an $\alpha$-stable law, then the Horton--Strahler number is of the same order as the logarithm of the size if and only if $\alpha\in (1,2]$. It would be nice to find a condition on $\mu$ for which $\HS(\tau)/\ln n$ under $\P(\, \cdot\, |\, \#\tau=n)$ converge to $0$ in probability, without any domain-of-attraction assumptions. This would fit in a recent line of work devoted to finding, without regularity assumptions, strong uniform bounds on shape parameters of random trees such as the height~\cite{DonLAB24}, the width~\cite{BranHamKerLAB22}, or their product~\cite{DonKha24}.

\item[$(iii)$] In this paper, we focused on large Galton--Watson trees with a fixed offspring distribution. It would be interesting to study the behavior of the Horton--Strahler number when the offspring distribution can vary with the size $n$, or for more general models of random trees such as multi-type Galton--Watson trees or uniform trees with fixed degree sequences.
\end{longlist}

%
%

\begin{acks}[Acknowledgments]
I am deeply indebted to my Ph.D.~advisor Thomas Duquesne for introducing me to the Horton--Strahler number, for suggesting me several problems about it, and for many insightful discussions. I am especially grateful for his help to prove Lemma~\ref{behavior_Lambda_1} and to improve the quality of this paper. I warmly thank Guillaume Boutoille et Yoan Tardy for a stimulating discussion about the correct order of magnitude that should appear in the Cauchy regime. Many thanks are due to Quentin Berger for some feedback about the proof of the estimate (\ref{PS_tails_size}). I thank Igor Kortchemski for some precisions about his work. Finally, I would like to thank the anonymous referees for their constructive comments and suggestions for future research directions.
\end{acks}
\begin{funding}
This research has been supported by the Natural Sciences and Engineering Research Council of Canada (NSERC) via a Banting postdoctoral fellowship [BPF-198443].
\noi
Cette recherche a été financée par le Conseil de recherches en sciences naturelles et en génie du Canada (CRSNG) via une bourse postdoctorale Banting [BPF-198443].
\end{funding}



\bibliographystyle{imsart-number} 
\bibliography{tech_ht_gw_15juin}       

\begin{thebibliography}{46}

\bibitem{AD2014}
\begin{barticle}[author]
\bauthor{\bsnm{Abraham},~\bfnm{Romain}\binits{R.}} \AND
  \bauthor{\bsnm{Delmas},~\bfnm{Jean-Fran{\c{c}}ois}\binits{J.-F.}}
(\byear{2014}).
\btitle{{Local limits of conditioned Galton-Watson trees: the infinite spine
  case}}.
\bjournal{Electronic Journal of Probability}
\bvolume{19}
\bpages{1 -- 19}.
\bdoi{10.1214/EJP.v19-2747}
\end{barticle}
\endbibitem

\bibitem{abraham2015introduction}
\begin{barticle}[author]
\bauthor{\bsnm{Abraham},~\bfnm{Romain}\binits{R.}} \AND
  \bauthor{\bsnm{Delmas},~\bfnm{Jean-Fran{\c{c}}ois}\binits{J.-F.}}
(\byear{2015}).
\btitle{{An introduction to Galton-Watson trees and their local limits}}.
\bjournal{Preprint available on arXiv}.
\bnote{arXiv:1506.05571}.
\end{barticle}
\endbibitem

\bibitem{BranHamKerLAB22}
\begin{barticle}[author]
\bauthor{\bsnm{Addario-Berry},~\bfnm{Louigi}\binits{L.}},
  \bauthor{\bsnm{Brandenberger},~\bfnm{Anna}\binits{A.}},
  \bauthor{\bsnm{Hamdan},~\bfnm{Jad}\binits{J.}} \AND
  \bauthor{\bsnm{Kerriou},~\bfnm{C{\'e}line}\binits{C.}}
(\byear{2022}).
\btitle{{Universal height and width bounds for random trees}}.
\bjournal{Electronic Journal of Probability}
\bvolume{27}
\bpages{1 -- 24}.
\bdoi{10.1214/22-EJP842}
\end{barticle}
\endbibitem

\bibitem{DonLAB24}
\begin{barticle}[author]
\bauthor{\bsnm{Addario-Berry},~\bfnm{Louigi}\binits{L.}} \AND
  \bauthor{\bsnm{Donderwinkel},~\bfnm{Serte}\binits{S.}}
(\byear{2024}).
\btitle{{Random trees have height $O(\sqrt{n})$}}.
\bjournal{The Annals of Probability}
\bvolume{52}
\bpages{2238 -- 2280}.
\bdoi{10.1214/24-AOP1694}
\end{barticle}
\endbibitem

\bibitem{aldouspitman98}
\begin{barticle}[author]
\bauthor{\bsnm{Aldous},~\bfnm{David}\binits{D.}} \AND
  \bauthor{\bsnm{Pitman},~\bfnm{Jim}\binits{J.}}
(\byear{1998}).
\btitle{{Tree-valued Markov chains derived from Galton-Watson processes}}.
\bjournal{Annales de l'Institut Henri Poincaré (B) Probability and Statistics}
\bvolume{34}
\bpages{637 -- 686}.
\bdoi{https://doi.org/10.1016/S0246-0203(98)80003-4}
\end{barticle}
\endbibitem

\bibitem{bamufleh}
\begin{barticle}[author]
\bauthor{\bsnm{Bamufleh},~\bfnm{Sameer}\binits{S.}},
  \bauthor{\bsnm{Al-Wagdany},~\bfnm{Abdullah}\binits{A.}},
  \bauthor{\bsnm{Elfeki},~\bfnm{Amro}\binits{A.}} \AND
  \bauthor{\bsnm{Chaabani},~\bfnm{Anis}\binits{A.}}
(\byear{2020}).
\btitle{{Developing a geomorphological instantaneous unit hydrograph (GIUH)
  using equivalent Horton-Strahler ratios for flash flood predictions in arid
  regions}}.
\bjournal{Geomatics, Natural Hazards and Risk}
\bvolume{11}
\bpages{1697 -- 1723}.
\bdoi{10.1080/19475705.2020.1811404}
\end{barticle}
\endbibitem

\bibitem{berger_cauchy}
\begin{barticle}[author]
\bauthor{\bsnm{Berger},~\bfnm{Quentin}\binits{Q.}}
(\byear{2019}).
\btitle{{Notes on Random Walks in the Cauchy Domain of Attraction}}.
\bjournal{Probability Theory and Related Fields}
\bvolume{175}
\bpages{1 -- 44}.
\bdoi{10.1007/s00440-018-0887-0}
\end{barticle}
\endbibitem

\bibitem{RegVar}
\begin{bbook}[author]
\bauthor{\bsnm{Bingham},~\bfnm{Nicholas~H.}\binits{N.~H.}},
  \bauthor{\bsnm{Goldie},~\bfnm{Charles~M.}\binits{C.~M.}} \AND
  \bauthor{\bsnm{Teugels},~\bfnm{Jozef~L.}\binits{J.~L.}}
(\byear{1989}).
\btitle{Regular Variation}.
\bseries{Encyclopedia of Mathematics and its Applications}
\bvolume{27}.
\bpublisher{Cambridge University Press}.
\end{bbook}
\endbibitem

\bibitem{brandenberger}
\begin{barticle}[author]
\bauthor{\bsnm{Brandenberger},~\bfnm{Anna}\binits{A.}},
  \bauthor{\bsnm{Devroye},~\bfnm{Luc}\binits{L.}} \AND
  \bauthor{\bsnm{Reddad},~\bfnm{Tommy}\binits{T.}}
(\byear{2021}).
\btitle{{The Horton–Strahler number of conditioned Galton–Watson trees}}.
\bjournal{Electronic Journal of Probability}
\bvolume{26}
\bpages{1 -- 29}.
\bdoi{10.1214/21-EJP678}
\end{barticle}
\endbibitem

\bibitem{burd}
\begin{barticle}[author]
\bauthor{\bsnm{Burd},~\bfnm{Gregory~A.}\binits{G.~A.}},
  \bauthor{\bsnm{Waymire},~\bfnm{Edward~C.}\binits{E.~C.}} \AND
  \bauthor{\bsnm{Winn},~\bfnm{Ronald~D.}\binits{R.~D.}}
(\byear{2000}).
\btitle{{A Self-Similar Invariance of Critical Binary Galton-Watson Trees}}.
\bjournal{Bernoulli}
\bvolume{6}
\bpages{1 -- 21}.
\end{barticle}
\endbibitem

\bibitem{chavan}
\begin{barticle}[author]
\bauthor{\bsnm{Chavan},~\bfnm{Sagar~R.}\binits{S.~R.}} \AND
  \bauthor{\bsnm{Srinivas},~\bfnm{Venkata~V.}\binits{V.~V.}}
(\byear{2015}).
\btitle{{Effect of DEM source on equivalent Horton–Strahler ratio based GIUH
  for catchments in two Indian river basins}}.
\bjournal{Journal of Hydrology}
\bvolume{528}
\bpages{463 -- 489}.
\bdoi{https://doi.org/10.1016/j.jhydrol.2015.06.049}
\end{barticle}
\endbibitem

\bibitem{devroye95}
\begin{barticle}[author]
\bauthor{\bsnm{Devroye},~\bfnm{Luc}\binits{L.}} \AND
  \bauthor{\bsnm{Kruszewski},~\bfnm{Paul}\binits{P.}}
(\byear{1995}).
\btitle{A note on the {H}orton-{S}trahler number for random trees}.
\bjournal{Information Processing Letters}
\bvolume{56}
\bpages{95 -- 99}.
\bdoi{https://doi.org/10.1016/0020-0190(95)00114-R}
\end{barticle}
\endbibitem

\bibitem{DonKha24}
\begin{barticle}[author]
\bauthor{\bsnm{Donderwinkel},~\bfnm{Serte}\binits{S.}} \AND
  \bauthor{\bsnm{Khanfir},~\bfnm{Robin}\binits{R.}}
(\byear{2024}).
\btitle{Tight universal bounds on the height times the width of random trees}.
\bjournal{Preprint available on arXiv}.
\bnote{arXiv:2501.00458}.
\end{barticle}
\endbibitem

\bibitem{drmota}
\begin{barticle}[author]
\bauthor{\bsnm{Drmota},~\bfnm{Michael}\binits{M.}} \AND
  \bauthor{\bsnm{Prodinger},~\bfnm{Helmut}\binits{H.}}
(\byear{2006}).
\btitle{{The Register Function for T-Ary Trees}}.
\bjournal{ACM Transactions on Algorithms}
\bvolume{2}
\bpages{318 -- 334}.
\bdoi{10.1145/1159892.1159894}
\end{barticle}
\endbibitem

\bibitem{duquesne_contour_stable}
\begin{barticle}[author]
\bauthor{\bsnm{Duquesne},~\bfnm{Thomas}\binits{T.}}
(\byear{2003}).
\btitle{{A limit theorem for the contour process of conditioned Galton--Watson
  trees}}.
\bjournal{The Annals of Probability}
\bvolume{31}
\bpages{996 -- 1027}.
\bdoi{10.1214/aop/1048516543}
\end{barticle}
\endbibitem

\bibitem{duquesne09}
\begin{barticle}[author]
\bauthor{\bsnm{Duquesne},~\bfnm{Thomas}\binits{T.}}
(\byear{2009}).
\btitle{{An elementary proof of Hawkes's conjecture on Galton-Watson trees.}}
\bjournal{Electronic Communications in Probability}
\bvolume{14}
\bpages{151 -- 164}.
\bdoi{10.1214/ECP.v14-1454}
\end{barticle}
\endbibitem

\bibitem{levytree_DLG}
\begin{bbook}[author]
\bauthor{\bsnm{Duquesne},~\bfnm{Thomas}\binits{T.}} \AND
  \bauthor{\bsnm{Le~Gall},~\bfnm{Jean-Fran\c{c}ois}\binits{J.-F.}}
(\byear{2002}).
\btitle{Random trees, {L\'evy} processes and spatial branching processes}.
\bseries{Ast\'erisque}
\bvolume{281}.
\bpublisher{Soci\'et\'e math\'ematique de France}.
\bmrnumber{1954248}
\end{bbook}
\endbibitem

\bibitem{duquesne_winkel}
\begin{barticle}[author]
\bauthor{\bsnm{Duquesne},~\bfnm{Thomas}\binits{T.}} \AND
  \bauthor{\bsnm{Winkel},~\bfnm{Matthias}\binits{M.}}
(\byear{2019}).
\btitle{{Hereditary tree growth and L{\'e}vy forests}}.
\bjournal{{Stochastic Processes and their Applications}}
\bvolume{129}
\bpages{3690 -- 3747}.
\end{barticle}
\endbibitem

\bibitem{esparza}
\begin{binproceedings}[author]
\bauthor{\bsnm{Esparza},~\bfnm{Javier}\binits{J.}},
  \bauthor{\bsnm{Luttenberger},~\bfnm{Michael}\binits{M.}} \AND
  \bauthor{\bsnm{Schlund},~\bfnm{Maximilian}\binits{M.}}
(\byear{2016}).
\btitle{History of {S}trahler {N}umbers — with a {P}reface}.
\bpublisher{International Conference on Language and Automata Theory and
  Applications, 2014}
\bnote{Available online at
  https://archive.model.in.tum.de/um/bibdb/esparza/latarevised16.pdf}.
\end{binproceedings}
\endbibitem

\bibitem{facbeneda}
\begin{barticle}[author]
\bauthor{\bsnm{Fac-Beneda},~\bfnm{Joanna}\binits{J.}}
(\byear{2013}).
\btitle{{Fractal structure of the Kashubian hydrographic system}}.
\bjournal{Journal of Hydrology}
\bvolume{488}
\bpages{48 -- 54}.
\bdoi{https://doi.org/10.1016/j.jhydrol.2013.02.033}
\end{barticle}
\endbibitem

\bibitem{feller1971}
\begin{bbook}[author]
\bauthor{\bsnm{Feller},~\bfnm{William}\binits{W.}}
(\byear{1971}).
\btitle{{An Introduction to Probability Theory and Its Applications. Vol. II.
  }}.
\bseries{Second}.
\bpublisher{John Wiley \& Sons Inc.}, \baddress{New York}.
\bmrnumber{MR0270403 (42 \#\#5292)}
\end{bbook}
\endbibitem

\bibitem{flajolet}
\begin{barticle}[author]
\bauthor{\bsnm{Flajolet},~\bfnm{Philippe}\binits{P.}},
  \bauthor{\bsnm{Raoult},~\bfnm{Jean-Claude}\binits{J.-C.}} \AND
  \bauthor{\bsnm{Vuillemin},~\bfnm{Jean~E.}\binits{J.~E.}}
(\byear{1979}).
\btitle{The number of registers required for evaluating arithmetic
  expressions}.
\bjournal{Theoretical Computer Science}
\bvolume{9}
\bpages{99 -- 125}.
\bdoi{https://doi.org/10.1016/0304-3975(79)90009-4}
\end{barticle}
\endbibitem

\bibitem{legall_trees}
\begin{barticle}[author]
\bauthor{\bsnm{Gall},~\bfnm{Jean-François~Le}\binits{J.-F.~L.}}
(\byear{2005}).
\btitle{{Random trees and applications}}.
\bjournal{Probability Surveys}
\bvolume{2}
\bpages{245 -- 311}.
\bdoi{10.1214/154957805100000140}
\end{barticle}
\endbibitem

\bibitem{lejan_legall}
\begin{barticle}[author]
\bauthor{\bsnm{Gall},~\bfnm{Jean-Francois~Le}\binits{J.-F.~L.}} \AND
  \bauthor{\bsnm{Jan},~\bfnm{Yves~Le}\binits{Y.~L.}}
(\byear{1998}).
\btitle{Branching Processes in Levy Processes: The Exploration Process}.
\bjournal{The Annals of Probability}
\bvolume{26}
\bpages{213 -- 252}.
\end{barticle}
\endbibitem

\bibitem{Grim74}
\begin{barticle}[author]
\bauthor{\bsnm{Grimvall},~\bfnm{Anders}\binits{A.}}
(\byear{1974}).
\btitle{{On the Convergence of Sequences of Branching Processes}}.
\bjournal{The Annals of Probability}
\bvolume{2}
\bpages{1027 -- 1045}.
\bdoi{10.1214/aop/1176996496}
\end{barticle}
\endbibitem

\bibitem{horton45}
\begin{barticle}[author]
\bauthor{\bsnm{Horton},~\bfnm{Robert~E.}\binits{R.~E.}}
(\byear{1945}).
\btitle{Erosional Development of streams and their drainage basins ;
  Hydrophysical approach to quantitative morphology}.
\bjournal{GSA Bulletin}
\bvolume{56}
\bpages{275 -- 370}.
\bdoi{10.1130/0016-7606(1945)56[275:EDOSAT]2.0.CO;2}
\end{barticle}
\endbibitem

\bibitem{karamata1930mode}
\begin{barticle}[author]
\bauthor{\bsnm{Karamata},~\bfnm{Jovan}\binits{J.}}
(\byear{1930}).
\btitle{{Sur un mode de croissance r{\'e}guli{\`e}re des fonctions}}.
\bjournal{Mathematica (Cluj)}
\bvolume{4}
\bpages{38 -- 53}.
\end{barticle}
\endbibitem

\bibitem{kemp}
\begin{barticle}[author]
\bauthor{\bsnm{Kemp},~\bfnm{Rainer}\binits{R.}}
(\byear{1979}).
\btitle{The average number of registers needed to evaluate a binary tree
  optimally}.
\bjournal{Acta Informatica}
\bvolume{11}
\bpages{363 -- 372}.
\end{barticle}
\endbibitem

\bibitem{kemperman}
\begin{bbook}[author]
\bauthor{\bsnm{Kemperman},~\bfnm{Johannes H.~B.}\binits{J.~H.~B.}}
(\byear{1961}).
\btitle{{The Passage Problem for a Stationary Markov Chain}}.
\bseries{Statistical research monographs}
\bvolume{1}.
\bpublisher{University of Chicago Press}.
\end{bbook}
\endbibitem

\bibitem{kesten86}
\begin{barticle}[author]
\bauthor{\bsnm{Kesten},~\bfnm{Harry}\binits{H.}}
(\byear{1986}).
\btitle{Subdiffusive behavior of random walk on a random cluster}.
\bjournal{Annales de l'Institut Henri Poincaré (B) Probability and Statistics}
\bvolume{22}
\bpages{425 -- 487}.
\end{barticle}
\endbibitem

\bibitem{Kha24}
\begin{barticle}[author]
\bauthor{\bsnm{Khanfir},~\bfnm{Robin}\binits{R.}}
(\byear{2024}).
\btitle{{Fluctuations of the Horton--Strahler number of stable Galton--Watson
  trees}}.
\bjournal{Preprint available on arXiv}.
\bnote{arXiv:2401.13771}.
\end{barticle}
\endbibitem

\bibitem{kortchemski_simple}
\begin{bincollection}[author]
\bauthor{\bsnm{Kortchemski},~\bfnm{Igor}\binits{I.}}
(\byear{2013}).
\btitle{A Simple Proof of {D}uquesne's Theorem on Contour Processes of
  Conditioned {G}alton--{W}atson Trees}.
In \bbooktitle{{S{\'e}minaire de Probabilit{\'e}s XLV}},
(\beditor{\bfnm{Catherine}\binits{C.}~\bsnm{Donati-Martin}},
  \beditor{\bfnm{Antoine}\binits{A.}~\bsnm{Lejay}} \AND
  \beditor{\bfnm{Alain}\binits{A.}~\bsnm{Rouault}}, eds.).
\bseries{Lecture Notes in Mathematics}
\bpages{537 -- 558}.
\bpublisher{Springer International Publishing}, \baddress{Heidelberg}.
\bdoi{10.1007/978-3-319-00321-4_20}
\end{bincollection}
\endbibitem

\bibitem{KRcauchy}
\begin{barticle}[author]
\bauthor{\bsnm{Kortchemski},~\bfnm{Igor}\binits{I.}} \AND
  \bauthor{\bsnm{Richier},~\bfnm{Loïc}\binits{L.}}
(\byear{2019}).
\btitle{{Condensation in critical Cauchy Bienaymé-Galton--Watson trees}}.
\bjournal{The Annals of Applied Probability}
\bvolume{29}
\bpages{1837 -- 1877}.
\end{barticle}
\endbibitem

\bibitem{hortonlaws}
\begin{barticle}[author]
\bauthor{\bsnm{Kovchegov},~\bfnm{Yevgeniy}\binits{Y.}} \AND
  \bauthor{\bsnm{Zaliapin},~\bfnm{Ilya}\binits{I.}}
(\byear{2020}).
\btitle{{Random self-similar trees: A mathematical theory of Horton laws}}.
\bjournal{Probability Surveys}
\bvolume{17}
\bpages{1 -- 213}.
\bdoi{10.1214/19-PS331}
\end{barticle}
\endbibitem

\bibitem{kovchegov}
\begin{barticle}[author]
\bauthor{\bsnm{Kovchegov},~\bfnm{Yevgeniy}\binits{Y.}} \AND
  \bauthor{\bsnm{Zaliapin},~\bfnm{Ilya}\binits{I.}}
(\byear{2021}).
\btitle{{Invariance and attraction properties of Galton–Watson trees}}.
\bjournal{Bernoulli}
\bvolume{27}
\bpages{1789 -- 1823}.
\bdoi{10.3150/20-BEJ1292}
\end{barticle}
\endbibitem

\bibitem{LlogLcriteria}
\begin{barticle}[author]
\bauthor{\bsnm{Lyons},~\bfnm{Russell}\binits{R.}},
  \bauthor{\bsnm{Pemantle},~\bfnm{Robin}\binits{R.}} \AND
  \bauthor{\bsnm{Peres},~\bfnm{Yuval}\binits{Y.}}
(\byear{1995}).
\btitle{{Conceptual Proofs of $L$ Log $L$ Criteria for Mean Behavior of
  Branching Processes}}.
\bjournal{The Annals of Probability}
\bvolume{23}
\bpages{1125 -- 1138}.
\bdoi{10.1214/aop/1176988176}
\end{barticle}
\endbibitem

\bibitem{marchal}
\begin{barticle}[author]
\bauthor{\bsnm{Marchal},~\bfnm{Philippe}\binits{P.}}
(\byear{2008}).
\btitle{{A note on the fragmentation of a stable tree}}.
\bjournal{{Discrete Mathematics \& Theoretical Computer Science}}
\bvolume{{DMTCS Proceedings vol. AI, Fifth Colloquium on Mathematics and
  Computer Science}}.
\bdoi{10.46298/dmtcs.3586}
\end{barticle}
\endbibitem

\bibitem{neveu}
\begin{barticle}[author]
\bauthor{\bsnm{Neveu},~\bfnm{Jacques}\binits{J.}}
(\byear{1986}).
\btitle{{Arbres et processus de Galton-Watson}}.
\bjournal{Annales de l'Intitut Henri Poincaré (B) Probability and Statistics}
\bvolume{22}
\bpages{199 -- 207}.
\end{barticle}
\endbibitem

\bibitem{peckham}
\begin{barticle}[author]
\bauthor{\bsnm{Peckham},~\bfnm{Scott~D.}\binits{S.~D.}}
(\byear{1995}).
\btitle{{New Results for Self-Similar Trees with Applications to River
  Networks}}.
\bjournal{Water Resources Research}
\bvolume{31}
\bpages{1023 -- 1029}.
\bdoi{https://doi.org/10.1029/94WR03155}
\end{barticle}
\endbibitem

\bibitem{skorokhod57}
\begin{barticle}[author]
\bauthor{\bsnm{Skorokhod},~\bfnm{Anatoliy~V.}\binits{A.~V.}}
(\byear{1957}).
\btitle{{Limit Theorems for Stochastic Processes with Independent Increments}}.
\bjournal{Theory of Probability \& Its Applications}
\bvolume{2}
\bpages{138 -- 171}.
\bdoi{10.1137/1102011}
\end{barticle}
\endbibitem

\bibitem{slack68}
\begin{barticle}[author]
\bauthor{\bsnm{Slack},~\bfnm{R.~S.}\binits{R.~S.}}
(\byear{1968}).
\btitle{A branching process with mean one and possibly infinite variance}.
\bjournal{Zeitschrift f{\"u}r Wahrscheinlichkeitstheorie und Verwandte Gebiete}
\bvolume{9}
\bpages{139 -- 145}.
\end{barticle}
\endbibitem

\bibitem{strahler52}
\begin{barticle}[author]
\bauthor{\bsnm{Strahler},~\bfnm{Arthur~N.}\binits{A.~N.}}
(\byear{1952}).
\btitle{Hypsometric (area-altitude) analysis of erosional topography}.
\bjournal{GSA Bulletin}
\bvolume{63}
\bpages{1117 -- 1142}.
\bdoi{10.1130/0016-7606(1952)63[1117:HAAOET]2.0.CO;2}
\end{barticle}
\endbibitem

\bibitem{vervaat}
\begin{barticle}[author]
\bauthor{\bsnm{Vervaat},~\bfnm{Wim}\binits{W.}}
(\byear{1979}).
\btitle{{A Relation between Brownian Bridge and Brownian Excursion}}.
\bjournal{The Annals of Probability}
\bvolume{7}
\bpages{143 -- 149}.
\end{barticle}
\endbibitem

\bibitem{Viennot}
\begin{bincollection}[author]
\bauthor{\bsnm{Viennot},~\bfnm{Xavier}\binits{X.}}
(\byear{1990}).
\btitle{Trees}.
In \bbooktitle{Mots, mélanges offert à M.P. Schützenberger}
\bpublisher{Hermès, Paris}
\bnote{Available online at http://www.xavierviennot.org/xavier/}.
\end{bincollection}
\endbibitem

\bibitem{wendel}
\begin{barticle}[author]
\bauthor{\bsnm{Wendel},~\bfnm{James~G.}\binits{J.~G.}}
(\byear{1975}).
\btitle{{Left-Continuous Random Walk and the Lagrange Expansion}}.
\bjournal{The American Mathematical Monthly}
\bvolume{82}
\bpages{494 -- 499}.
\end{barticle}
\endbibitem

\bibitem{zolotarev}
\begin{barticle}[author]
\bauthor{\bsnm{Zolotarev},~\bfnm{Vladimir~M.}\binits{V.~M.}}
(\byear{1957}).
\btitle{More Exact Statements of Several Theorems in the Theory of Branching
  Processes}.
\bjournal{Theory of Probability \& Its Applications}
\bvolume{2}
\bpages{245 -- 253}.
\bdoi{10.1137/1102016}
\end{barticle}
\endbibitem

\end{thebibliography}


\end{document}